\pdfoutput=1
\documentclass[a4paper,11pt]{amsart}
\usepackage[hmarginratio={1:1},vmarginratio={1:1},lmargin=80.0pt,tmargin=90.0pt]{geometry}

\usepackage{calligra}
\usepackage{fontenc}

\usepackage{mathtools}
\mathtoolsset{mathic}
\usepackage[final]{microtype}
\usepackage{multicol}
\usepackage{tabularx}
\usepackage{latexsym,exscale,enumitem,amsfonts,amssymb,mathtools}
\usepackage{amsmath,amsthm,amsfonts,amssymb,amscd,textcomp,xcolor,bbm}
\usepackage{thmtools}
\usepackage{stmaryrd}
\SetSymbolFont{stmry}{bold}{U}{stmry}{m}{n}
\usepackage{mathrsfs}
\usepackage{etex}
\usepackage{fancybox}

\usepackage{caption} 
\usepackage[all]{xy}
\SelectTips{cm}{}

\usepackage{tikz}
\tikzset{anchorbase/.style={baseline={([yshift=-0.5ex]current bounding box.center)}}}
\usetikzlibrary{decorations.markings}
\usetikzlibrary{decorations.pathreplacing}
\usetikzlibrary{arrows,shapes,positioning}
\tikzstyle directed=[postaction={decorate,decoration={markings,
    mark=at position #1 with {\arrow{>}}}}]
\tikzstyle rdirected=[postaction={decorate,decoration={markings,
    mark=at position #1 with {\arrow{<}}}}]

\tikzset{
    partial ellipse/.style args={#1:#2:#3}{
        insert path={+ (#1:#3) arc (#1:#2:#3)}
    }
}

\definecolor{sea}{RGB}{46,139,87}
\definecolor{tomato}{RGB}{255,99,71}
\definecolor{orchid}{RGB}{143,40,194}
\definecolor{lava}{RGB}{207,16,32}

\definecolor{mydarkblue}{RGB}{10,10,170}
\definecolor{mygray}{gray}{0.75}

\makeatletter
\renewcommand{\boxed}[1]{\text{\fboxsep=.1em\Ovalbox{\m@th$\displaystyle#1$}}}
\makeatother

\newcommand{\nicefrac}[2]{%
    \raise.5ex\hbox{#1}%
    \kern-.1em/\kern-.15em%
    \lower.25ex\hbox{#2}}

\newcommand{\C}{\mathbb{C}}
\newcommand{\Z}{\mathbb{Z}}

\newcommand{\N}{\mathbb{Z}_{\geq 0}}

\newcommand{\someR}{\mathbb{K}}

\newcommand{\shiftdeg}{\mathtt{a}}
\newcommand{\shiftme}[1]{\mathtt{#1}}

\newcommand{\qpar}{\mathrm{q}}
\newcommand{\vpar}{\mathrm{v}}

\newcommand{\zq}{{\mathscr{A}}}
\newcommand{\Zq}{{\mathscr{A}}_{\mathrm{gen}}}
\newcommand{\ZQ}{{\mathscr{A}}_n}
\newcommand{\CV}{\C[\vpar,\vpar^{-1}]}
\newcommand{\Cv}{\C(\vpar)}

\newcommand{\category}[1]{\boldsymbol{\mathrm{#1}}}
\newcommand{\Functor}[1]{\boldsymbol{#1}}
\newcommand{\functor}[1]{\boldsymbol{\mathcal{#1}}}
\newcommand{\nattrafo}[1]{\boldsymbol{\mathfrak{#1}}}

\newcommand{\vcomp}{\circ_v}
\newcommand{\hcomp}{\circ_h}

\newcommand{\qgtimes}{\,\widehat{\otimes}\,}

\newcommand{\Kar}{\boldsymbol{\mathrm{Kar}}}
\newcommand{\GG}[1]{[#1]_{\C(\vpar)}}

\newcommand{\Modl}[1]{#1\text{-}\mathrm{p}\boldsymbol{\mathrm{Mod}}^{\mathrm{gr}}}

\newcommand{\slmod}{\mathrm{U}_{\qpar}(\mathfrak{sl}_2)\text{-}\boldsymbol{\mathrm{Mod}}}
\newcommand{\slmods}{\mathrm{U}_{\qpar}(\mathfrak{sl}_2)\text{-}\boldsymbol{\mathrm{Mod}}_{\mathrm{s}}}

\newcommand{\dihedral}[1]{\mathrm{I}_2(#1)}

\newcommand{\cellrepgrg}{\mathrm{G}^{\vpar}}

\newcommand{\dihsgen}{{\color{sea}s}}
\newcommand{\dihtgen}{{\color{tomato}t}}
\newcommand{\dihugen}{{\color{orchid}u}}
\newcommand{\dihvgen}{{\color{orchid}v}}
\newcommand{\dihwgen}{{\color{orchid}w}}
\newcommand{\word}{{\color{orchid}w}}
\newcommand{\wnull}{{\color{orchid}w}_0}
\newcommand{\klbasis}{\theta}
\newcommand{\disgen}{\theta_{\color{sea}s}}
\newcommand{\ditgen}{\theta_{\color{tomato}t}}

\newcommand{\disfun}{\boldsymbol{\Theta}_{\color{sea}s}}
\newcommand{\ditfun}{\boldsymbol{\Theta}_{\color{tomato}t}}

\newcommand{\onecolor}{\boldsymbol{\mathcal{D}}_{\typea{1}}}

\newcommand{\dihedralcatbig}[1]{\boldsymbol{\mathcal{D}}^{\star}_{#1}}
\newcommand{\dihedralcat}[1]{\boldsymbol{\mathcal{D}}_{#1}}
\newcommand{\dicatsgen}{{\color{sea}s}}
\newcommand{\dicattgen}{{\color{tomato}t}}
\newcommand{\dicatstgen}{{\color{orchid}u}}
\newcommand{\dicattsgen}{{\color{orchid}v}}

\newcommand{\bullets}{\text{{\scalebox{1.075}{\color{sea}$\bullet$}}}}
\newcommand{\bullett}{\text{{\scalebox{.625}{\color{tomato}$\blacksquare$}}}}
\newcommand{\bulletst}{\text{{\scalebox{.675}{\color{orchid}$\blacklozenge$}}}}

\newcommand{\JWs}[1]{\mathrm{JW}^{\dicatsgen}_{#1}}
\newcommand{\JWt}[1]{\mathrm{JW}^{\dicattgen}_{#1}}
\newcommand{\JWst}[1]{\mathrm{JW}^{\dicatstgen}_{#1}}

\newcommand{\jws}[1]{\overline{\mathrm{JW}}^{\dicatsgen}_{#1}}

\newcommand{\jwst}[1]{\overline{\mathrm{JW}}^{\dicatstgen}_{#1}}

\newcommand{\bulletssmall}{\text{{\scalebox{.85}{\color{sea}$\bullet$}}}}
\newcommand{\bullettsmall}{\text{{\scalebox{.45}{\color{tomato}$\blacksquare$}}}}

\newcommand{\barbells}{\!
\begin{tikzpicture}[anchorbase, scale=.175]
	\draw [very thick, sea] (0,-.55) to (0,.55);
	\node at (0,-.55) {$\bulletssmall$};
	\node at (0,.55) {$\bulletssmall$};
\end{tikzpicture}\!
}
\newcommand{\barbellt}{\!
\begin{tikzpicture}[anchorbase, scale=.175]
	\draw [very thick, tomato] (0,-.55) to (0,.55);
	\node at (0,-.55) {$\bullettsmall$};
	\node at (0,.55) {$\bullettsmall$};
\end{tikzpicture}\!
}

\newcommand{\cellcatG}{\boldsymbol{\mathrm{G}}^{\mathrm{gr}}}
\newcommand{\cellcatGnext}{(\boldsymbol{\mathrm{G}}^{\prime}){}^{\mathrm{gr}}}

\newcommand{\cellcatADE}{\boldsymbol{\mathrm{G}}^{\mathrm{gr}}}

\newcommand{\cellcatA}[1]{\boldsymbol{\mathrm{A}}^{\mathrm{gr}}(#1)}

\newcommand{\cellcatAt}[1]{\boldsymbol{\tilde{\mathrm{A}}}{}^{\mathrm{gr}}(#1)}

\newcommand{\bifunctor}{\boldsymbol{\mathcal{V}}}
\newcommand{\functors}{\boldsymbol{\Theta}_{{\color{sea}s}}}
\newcommand{\functort}{\boldsymbol{\Theta}_{{\color{tomato}t}}}

\newcommand{\natid}{\mathfrak{id}}

\newcommand{\functorG}{\boldsymbol{\mathcal{G}}_{\infty}}
\newcommand{\functorGno}{\boldsymbol{\mathcal{G}}}

\newcommand{\functorADE}{\boldsymbol{\mathcal{G}}_{n}}

\newcommand{\functorADEt}{\boldsymbol{\widetilde{\mathcal{G}}}_{n}}

\newcommand{\Hom}{\mathrm{Hom}}
\newcommand{\twoHom}{2\mathrm{Hom}}
\newcommand{\Endf}{\mathrm{p}\boldsymbol{\mathrm{End}}}
\newcommand{\Endff}{\mathrm{p}\boldsymbol{\mathrm{End}}^{\star}}

\newcommand{\End}{\mathrm{End}}
\newcommand{\twoEnd}{2\mathrm{End}}

\newcommand{\pathm}{\,
\xy
(0,0)*{
\begin{tikzpicture}[anchorbase, scale=1]
	\draw [very thick, fill=black] (0,0) circle (.02cm);
\end{tikzpicture}
};
\endxy\,
}

\newcommand{\graphs}{{\color{sea}\underline{\mathrm{S}}}}
\newcommand{\grapht}{{\color{tomato}\overline{\mathrm{T}}}}

\newcommand{\pathG}{\mathrm{P}\g}

\newcommand{\algG}{\mathrm{Q}\g}

\newcommand{\algA}[1]{\mathrm{Q}\mathrm{A}_{#1}}
\newcommand{\algAt}[1]{\mathrm{Q}\tilde{\mathrm{A}}_{#1}}

\newcommand{\qalg}{Q}
\newcommand{\somealg}{\mathrm{A}}
\newcommand{\polyalg}{\mathrm{R}}
\newcommand{\coalg}{\mathrm{C}}
\newcommand{\dualalg}{\mathrm{D}}

\newcommand{\loopy}[1]{#1|#1}

\newcommand{\bbi}[1]{\raisebox{-.1cm}{{\color{sea}$\underline{\mathrm{#1}}$}}}
\newcommand{\bbj}[1]{\raisebox{.1cm}{{\color{tomato}$\overline{\mathrm{#1}}$}}}
\newcommand{\bbii}[1]{{\color{sea}\underline{\mathrm{#1}}}}
\newcommand{\bbjj}[1]{{\color{tomato}\overline{\mathrm{#1}}}}
\newcommand{\bbk}[1]{{\color{orchid}\mathrm{#1}}}

\newcommand{\gconnect}[1]{
{\ensuremath\!
\begin{tikzpicture}[anchorbase, scale=.25]
	\draw [very thick] (-.45,0) to (.45,0);
	\node at (0,-.35) {\tiny $#1$};
	\node at (0,.25) {\tiny $\phantom{a}$};
\end{tikzpicture}\!}
}

\newcommand{\iddia}{\mathrm{id}}

\newcommand{\idmap}{\mathrm{ID}}

\newcommand{\ADE}{\mathrm{ADE}}

\newcommand{\g}{G}
\newcommand{\G}{\mathrm{G}}

\newcommand{\typeA}{\mathrm{A}}
\newcommand{\typeD}{\mathrm{D}}
\newcommand{\typeE}{\mathrm{E}}
\newcommand{\typeAt}{\tilde{\mathrm{A}}}
\newcommand{\typeDt}{\tilde{\mathrm{D}}}
\newcommand{\typeEt}{\tilde{\mathrm{E}}}
\newcommand{\typea}[1]{\mathrm{A}_{#1}}
\newcommand{\typeat}[1]{\tilde{\mathrm{A}}_{#1}}
\newcommand{\typed}[1]{\mathrm{D}_{#1}}
\newcommand{\typee}[1]{\mathrm{E}_{#1}}

\newcommand{\functorGs}{\boldsymbol{\mathcal{H}}}

\newcommand{\scalargt}{{\color{orchid}\tau}}
\newcommand{\scalargtt}{{\color{orchid}\upsilon}}
\newcommand{\scalargttt}{{\color{orchid}\varsigma}}

\newcommand{\lscalar}{\lambda}
\newcommand{\llscalar}{\vec{\lscalar}}

\newcommand{\zgl}{{\mathscr{A}}_{\g,\llscalar,\qpar}}

\newcommand{\seamerge}{\,\tikz[baseline=1,scale=0.08]{\draw[very thick, sea] (0,0) to (1,1.5); \draw[very thick,sea] (2,0) to (1,1.5);\draw[very thick,sea] (1,1.5) to (1,3);}\,}
\newcommand{\seasplit}{\,\tikz[baseline=1,scale=0.08]{\draw[very thick,sea] (0,3) to (1,1.5); \draw[very thick,sea] (2,3) to (1,1.5);\draw[very thick, sea] (1,1.5) to (1,0);}\,}
\newcommand{\seadotdown}{\,\tikz[baseline=0.5,scale=0.08]{\draw[very thick,sea] (0,0.7) -- ++(0,2); \node[circle,fill,inner sep=1pt,sea] at (0,0.7) {};}\,}
\newcommand{\seadotup}{\,\tikz[baseline=0.5,scale=0.08]{\draw[very thick,sea] (0,1.7) -- ++(0,-2); \node[circle,fill,inner sep=1pt,sea] at (0,1.7) {};}\,}

\newcommand{\seabarb}{\,\tikz[baseline=-.5,scale=0.08]{\draw[very thick,sea] (0,1.7) -- ++(0,-2); \node[circle,fill,inner sep=1pt,sea] at (0,1.7) {}; \node[circle,fill,inner sep=1pt,sea] at (0,-.3) {};}\,}

\newcommand{\tomerge}{\,\tikz[baseline=1,scale=0.08]{\draw[very thick,tomato] (0,0) to (1,1.5); \draw[very thick,tomato] (2,0) to (1,1.5);\draw[very thick,tomato] (1,1.5) to (1,3);}\,}
\newcommand{\tosplit}{\,\tikz[baseline=1,scale=0.08]{\draw[very thick,tomato] (0,3) to (1,1.5); \draw[very thick,tomato] (2,3) to (1,1.5);\draw[very thick,tomato] (1,1.5) to (1,0);}\,}
\newcommand{\todotdown}{\,\tikz[baseline=0.5,scale=0.08]{\draw[very thick,tomato] (0,0.7) -- ++(0,2); \node[rectangle,fill,inner sep=1.35pt,tomato] at (0,0.7) {};}\,}
\newcommand{\todotup}{\,\tikz[baseline=0.5,scale=0.08]{\draw[very thick,tomato] (0,1.7) -- ++(0,-2); \node[rectangle,fill,inner sep=1.35pt,tomato] at (0,1.7) {};}\,}

\newcommand{\tobarb}{\,\tikz[baseline=-.5,scale=0.08]{\draw[very thick,tomato] (0,1.7) -- ++(0,-2); \node[rectangle,fill,inner sep=1.35pt,tomato] at (0,1.7) {}; \node[rectangle,fill,inner sep=1.35pt,tomato] at (0,-.3) {};}\,}

\theoremstyle{definition}
\newtheorem{theoremm}{Theorem}[section]

\declaretheorem[style=definition,name=Theorem,qed=$\square$,numberlike=theoremm]{theorem}

\declaretheorem[style=definition,name=Lemma,qed=$\square$,numberlike=theoremm]{lemma}
\declaretheorem[style=definition,name=Proposition,qed=$\square$,numberlike=theoremm]{proposition}

\declaretheorem[style=definition,name=Example,qed=$\blacktriangle$,numberlike=theorem]{example}
\declaretheorem[style=definition,name=Definition,qed=$\blacktriangle$,numberlike=theorem]{definition}
\declaretheorem[style=definition,name=Remark,qed=$\blacktriangle$,numberlike=theorem]{remark}

\declaretheorem[style=definition,name=Definition,numberlike=theorem]{definitionn}
\declaretheorem[style=definition,name=Example,numberlike=theorem]{examplen}
\declaretheorem[style=definition,name=Remark,numberlike=theorem]{remarkn}
\declaretheorem[style=definition,name=Corollary,numberlike=theorem]{corollaryn}

\declaretheorem[style=definition,name=Lemma,numberlike=theoremm]{lemman}
\declaretheorem[style=definition,name=Proposition,numberlike=theoremm]{propositionn}

\declaretheorem[style=definition,name=Theorem,numberwithin=section]{theoremmain}

\def\notation#1#2#3{\rlap{\hyperref[#1]{{\color{orchid}#2}}}\hspace*{8.2mm} \hbox to 47mm{#3\hfill}}

\newcommand{\makeqed}{\hfill\ensuremath{\square}}
\newcommand{\makeqedtri}{\hfill\ensuremath{\blacktriangle}}
\newcommand{\qedmake}{\hfill\ensuremath{\blacksquare}}

\allowdisplaybreaks

\numberwithin{equation}{section}
\usepackage[hypertexnames=false]{hyperref}
\usepackage{cleveref,bookmark}
\hypersetup{
    pdftoolbar=true,        
    pdfmenubar=true,        
    pdffitwindow=false,     
    pdfstartview={FitH},    
    pdftitle={Two-color Soergel calculus and simple transitive \texorpdfstring{$2$}{2}-representations},    
    pdfauthor={Marco Mackaay and Daniel Tubbenhauer},     
    pdfsubject={},   
    pdfcreator={Marco Mackaay and Daniel Tubbenhauer},   
    pdfproducer={Marco Mackaay and Daniel Tubbenhauer}, 
    pdfkeywords={}, 
    pdfnewwindow=true,      
    colorlinks=true,       
    linkcolor=mydarkblue,          
    citecolor=teal,        
    filecolor=magenta,      
    urlcolor=orchid,          
    linkbordercolor=lava,
    citebordercolor=teal,
    urlbordercolor=orchid,  
    linktocpage=true
}

\let\fullref\autoref
\def\makeautorefname#1#2{\expandafter\def\csname#1autorefname\endcsname{#2}}
\makeautorefname{equation}{Equation}%
\makeautorefname{footnote}{footnote}%
\makeautorefname{item}{item}%
\makeautorefname{figure}{Figure}%
\makeautorefname{table}{Table}%
\makeautorefname{part}{Part}%
\makeautorefname{appendix}{Appendix}%
\makeautorefname{chapter}{Chapter}%
\makeautorefname{section}{Section}%
\makeautorefname{subsection}{Section}%
\makeautorefname{subsubsection}{Section}%
\makeautorefname{paragraph}{Paragraph}%
\makeautorefname{subparagraph}{Paragraph}%
\makeautorefname{theorem}{Theorem}%
\makeautorefname{theo}{Theorem}%
\makeautorefname{thm}{Theorem}%
\makeautorefname{addendum}{Addendum}%
\makeautorefname{addend}{Addendum}%
\makeautorefname{add}{Addendum}%
\makeautorefname{maintheorem}{Main theorem}%
\makeautorefname{mainthm}{Main theorem}%
\makeautorefname{corollary}{Corollary}%
\makeautorefname{corol}{Corollary}%
\makeautorefname{coro}{Corollary}%
\makeautorefname{cor}{Corollary}%
\makeautorefname{lemma}{Lemma}%
\makeautorefname{lemm}{Lemma}%
\makeautorefname{lem}{Lemma}%
\makeautorefname{sublemma}{Sublemma}%
\makeautorefname{sublem}{Sublemma}%
\makeautorefname{subl}{Sublemma}%
\makeautorefname{proposition}{Proposition}%
\makeautorefname{proposit}{Proposition}%
\makeautorefname{propos}{Proposition}%
\makeautorefname{propo}{Proposition}%
\makeautorefname{prop}{Proposition}%
\makeautorefname{property}{Property}
\makeautorefname{proper}{Property}
\makeautorefname{scholium}{Scholium}%
\makeautorefname{step}{Step}%
\makeautorefname{conjecture}{Conjecture}%
\makeautorefname{conject}{Conjecture}%
\makeautorefname{conj}{Conjecture}%
\makeautorefname{question}{Question}
\makeautorefname{questn}{Question}
\makeautorefname{quest}{Question}
\makeautorefname{ques}{Question}
\makeautorefname{qn}{Question}
\makeautorefname{definition}{Definition}%
\makeautorefname{defin}{Definition}%
\makeautorefname{defi}{Definition}%
\makeautorefname{def}{Definition}%
\makeautorefname{dfn}{Definition}%
\makeautorefname{notation}{Notation}
\makeautorefname{nota}{Notation}
\makeautorefname{notn}{Notation}
\makeautorefname{remark}{Remark}%
\makeautorefname{rema}{Remark}%
\makeautorefname{rem}{Remark}%
\makeautorefname{rmk}{Remark}%
\makeautorefname{rk}{Remark}%
\makeautorefname{remarks}{Remarks}%
\makeautorefname{rems}{Remarks}%
\makeautorefname{rmks}{Remarks}%
\makeautorefname{rks}{Remarks}%
\makeautorefname{example}{Example}%
\makeautorefname{examp}{Example}%
\makeautorefname{exmp}{Example}%
\makeautorefname{exam}{Example}%
\makeautorefname{exa}{Example}%
\makeautorefname{algorithm}{Algorith}%
\makeautorefname{algo}{Algorith}%
\makeautorefname{alg}{Algorith}%
\makeautorefname{axiom}{Axiom}%
\makeautorefname{axi}{Axiom}%
\makeautorefname{ax}{Axiom}%
\makeautorefname{case}{Case}%
\makeautorefname{claim}{Claim}%
\makeautorefname{clm}{Claim}%
\makeautorefname{assumption}{Assumption}%
\makeautorefname{assumpt}{Assumption}%
\makeautorefname{conclusion}{Conclusion}%
\makeautorefname{concl}{Conclusion}%
\makeautorefname{conc}{Conclusion}%
\makeautorefname{condition}{Condition}%
\makeautorefname{condit}{Condition}%
\makeautorefname{cond}{Condition}%
\makeautorefname{construction}{Construction}%
\makeautorefname{construct}{Construction}%
\makeautorefname{const}{Construction}%
\makeautorefname{cons}{Construction}%
\makeautorefname{criterion}{Criterion}%
\makeautorefname{criter}{Criterion}%
\makeautorefname{crit}{Criterion}%
\makeautorefname{exercise}{Exercise}%
\makeautorefname{exer}{Exercise}%
\makeautorefname{exe}{Exercise}%
\makeautorefname{problem}{Problem}%
\makeautorefname{problm}{Problem}%
\makeautorefname{probm}{Problem}%
\makeautorefname{prob}{Problem}%
\makeautorefname{solution}{Solution}%
\makeautorefname{soln}{Solution}%
\makeautorefname{sol}{Solution}%
\makeautorefname{summary}{Summary}%
\makeautorefname{summ}{Summary}%
\makeautorefname{sum}{Summary}%
\makeautorefname{operation}{Operation}%
\makeautorefname{oper}{Operation}%
\makeautorefname{observation}{Observation}%
\makeautorefname{observn}{Observation}%
\makeautorefname{obser}{Observation}%
\makeautorefname{obs}{Observation}%
\makeautorefname{ob}{Observation}%
\makeautorefname{convention}{Convention}%
\makeautorefname{convent}{Convention}%
\makeautorefname{conv}{Convention}%
\makeautorefname{cvn}{Convention}%
\makeautorefname{warning}{Warning}%
\makeautorefname{warn}{Warning}%
\makeautorefname{note}{Note}%

\begin{document}
\vbadness=10001
\hbadness=10001
\renewcommand{\theequation}{\thesection.\arabic{equation}}
\title[Two-color Soergel calculus and simple transitive \texorpdfstring{$2$}{2}-representations]{Two-color Soergel calculus and simple transitive \texorpdfstring{$2$}{2}-representations}
\author[M. Mackaay and D. Tubbenhauer]{Marco Mackaay and Daniel Tubbenhauer}

\address{M.M.: Center for Mathematical Analysis, Geometry, and Dynamical Systems,
Departamento de Matem\'{a}tica, Instituto Superior T\'{e}cnico, 1049-001 Lisboa, Portugal \& Departamento de Matem\'{a}tica, FCT, Universidade do Algarve, Campus de Gambelas, 8005-139 Faro, Portugal}
\email{mmackaay@ualg.pt}

\address{D.T.: Institut f\"ur Mathematik, Universit\"at Z\"urich, Winterthurerstrasse 190, Campus Irchel, Office Y27J32, CH-8057 Z\"urich, Switzerland, \href{www.dtubbenhauer.com}{www.dtubbenhauer.com}}
\email{daniel.tubbenhauer@math.uzh.ch}

\begin{abstract}
In this paper we complete the $\ADE$-like
classification 
of simple transitive $2$-representations 
of Soergel bimodules
in finite dihedral type, under the assumption of gradeability. In particular, we use bipartite 
graphs and zigzag algebras of $\ADE$ type to give an explicit construction of a graded (non-strict) 
version of all these $2$-representations.

Moreover, 
we give simple combinatorial 
criteria for when two such $2$-representations are 
equivalent and for when their Grothendieck groups 
give rise to isomorphic representations. 

Finally, our construction 
also gives a large class of simple transitive $2$-representations 
in infinite dihedral type for general bipartite graphs.
%
%
%
\end{abstract}

\maketitle
\tableofcontents
%
%
\section{Introduction}\label{sec:intro}
An essential problem in classical representation 
theory is the classification of the simple representations 
of any given algebra, i.e. the parametrization of their 
isomorphism classes and the explicit construction of a 
representative of each class.  

In $2$-representation theory, the actions of algebras on 
vector spaces are replaced by functorial actions of 
$2$-categories on certain additive or abelian $2$-categories. 
The Grothendieck group of a $2$-representation is a classical 
representation. One can say that the 
$2$-representation decategorifies 
to the classical representation, or that the latter is 
a decategorification of the former. Vice versa, one can also 
say that the $2$-representation categorifies the classical 
representation to which it decategorifies, or 
that it is a categorification. Note that, in general, categorifications need 
not be unique.

Examples are $2$-representations of the $2$-categories which categorify representations of quantum groups, due to (Chuang--)Rouquier 
and Khovanov--Lauda, and $2$-representations of the $2$-category 
of Soergel bimodules, which categorify representations of Hecke algebras. 

Mazorchuk--Miemietz \cite{MM5} defined 
an appropriate 
$2$-categorical analogue of the simple representations 
of finite-dimensional algebras, which they called 
simple transitive $2$-representations (of 
finitary $2$-categories). The problem is that their classification is very hard and not well 
understood in general -- except when certain specific conditions are satisfied, 
as for Soergel bimodules in type $\typeA$ \cite[Theorem 21]{MM5} for example.

The authors of \cite{KMMZ} studied the 
so-called small quotient of Soergel bimodules and their 
simple transitive $2$-representations, for all finite Coxeter types. 
These $2$-representations are given by categories  
on which the bimodules act by endofunctors and the 
bimodule maps by natural transformations. Each of these categories 
is equivalent to the (projective or abelian) module 
category over the path algebra of a finite quiver, which 
can be obtained by doubling a certain Dynkin diagram. An almost complete 
classification was given in \cite{KMMZ}, which we now recall.

In every finite Coxeter type of rank strictly greater than 
two, all the simple transitive $2$-representations are 
equivalent to Mazorchuk and Miemietz's categorification of the 
cell representations of Hecke algebras,
the so-called cell $2$-representations.  

The rank two case is more delicate. In type $\dihedral{n}$, for any $n\in\Z_{>1}$, there 
is one cell $2$-representation of rank one and two higher rank cell $2$-representations, which 
correspond to the two possible bipartitions of the Dynkin diagram 
of type $\typea{n{-}1}$. (Here we distinguish a bipartition of 
a given graph from the opposite bipartition.) When $n$ is odd, 
these exhaust all simple transitive $2$-representations, up to 
equivalence. 

When $n=2,4$, it was already known that the 
same holds, see \cite{Zi1}. However, when $n$ is even and greater than four, 
it was shown in \cite{KMMZ} that there exist additional simple 
transitive $2$-representations which are {\em not} equivalent to cell 
$2$-representations. If one has $n\not\in\{12, 18, 30\}$, then there exist exactly 
two which correspond to the two possible bipartitions of a 
type $\typed{\text{\tiny$\frac{n}{2}{+}1$}}$ Dynkin diagram. 

For $n=12,18,30$, 
the possible existence of one more additional 
pair of inequivalent simple transitive $2$-representations 
was discovered, but not proved (not even conjectured) in \cite{KMMZ}. 
Should they exist, their underlying Dynkin diagrams were 
shown to be of type $\typee{6}$ for $n=12$, of type $\typee{7}$ for $n=18$, 
and of type $\typee{8}$ for $n=30$. 

The simple transitive $2$-representations with quivers of type $\typeA$ and 
$\typeD$ were constructed intrinsically in \cite{KMMZ}.
The ones of type $\typeA$ are equivalent to the aforementioned 
higher rank cell $2$-re\-presen\-tations, due to 
Mazorchuk--Miemietz \cite{MM1}, and can 
be constructed as subquotients of the $2$-category 
of Soergel bimodules, as in any finite Coxeter type. The 
type $\typeD$ simple transitive $2$-representations of 
$\dihedral{n}$, for $n>5$ even, were constructed in \cite{KMMZ} using an 
involution on the cell $2$-representations, 
mimicking a construction in \cite{MM}. Both 
constructions do not take into account the $\Z$-grading 
of the Soergel bimodules, but the $2$-representations 
admit a unique compatible grading.   

In this paper, we construct {\em all} graded simple transitive $2$-representations 
of the small quotient of the Soergel bimodules of type $\dihedral{n}$ 
by {\em different} means. We use Elias' \cite{El1} 
diagrammatic version of the latter $2$-category, 
the so-called two-color Soergel calculus. More precisely, given a Dynkin 
diagram of type $\typeA,\typeD$ or $\typeE$ with a bipartition, we 
define two degree-preserving, self-adjoint endofunctors $\functors$ and $\functort$ on the module category 
over the corresponding quiver -- which is a zigzag algebra 
of type $\ADE$ -- and, for each generating diagram in 
the two-color Soergel calculus, a natural transformation 
between composites of them such that all diagrammatic relations are preserved. 

Moreover, we show that 
two graded simple transitive $2$-representations are equivalent 
if and only if the corresponding bipartite graphs are 
isomorphic (as bipartite graphs). Finally, 
we also determine when graded simple transitive $2$-representations 
decategorify to isomorphic representations of the corresponding Hecke 
algebra, using a purely graph-theoretic property. 

Let us give some interesting 
consequences of our results. First,  
there are two inequivalent graded simple transitive $2$-representations 
of type $\typee{6}$ (or $\typee{8}$) for the Soergel bimodules
of type $\dihedral{12}$ (or $\dihedral{30}$), 
which decategorify to isomorphic representations of the associated
Hecke algebra. (The two simple transitive $2$-representations of type $\typee{7}$ 
for the type $\dihedral{18}$ Soergel bimodules have non-isomorphic decategorifications.) 
To the best of our 
knowledge these are the first examples of simple transitive $2$-representations of the 
{\em same} $2$-category which decategorify 
to isomorphic representations. 

Our construction also gives 
graded simple transitive $2$-representations 
of the $2$-ca\-tegory defined by the two-color Soergel calculus 
in type $\dihedral{\infty}$ for any 
bipartite graph, not just the ones of $\ADE$ type. (Note that this $2$-category is not always isomorphic to the $2$-category of Soergel bimodules 
in type $\dihedral{\infty}$, cf. \fullref{remark-gg-goes-wrong}.) For these 
$2$-representations, all of the above 
statements are still valid.

\subsubsection*{Potential further developments}\label{subsub:applications}

We hope that our construction 
will also be helpful for 
the construction of simple 
transitive $2$-representations of other $2$-categories, e.g. 
Soergel bimodules in  
other finite Coxeter types with more $2$-cells. 

Furthermore, since our preprint first appeared, we wrote a joint paper 
with Mazorchuk--Miemietz \cite{MMMT} in which we explain the relation 
between the simple transitive $2$-representations of the Soergel bimodules in  
finite dihedral type and those of the 
semisimplified subquotient of the module category of
quantum $\mathfrak{sl}_2$ at a root of unity, due to 
Kirillov--Ostrik \cite{KO1}, \cite{Os1} and others. 
This relation is based on Elias' algebraic quantum Satake 
equivalence \cite{El1}, \cite{El2}. There should be a similar 
story for $\mathfrak{sl}_{>2}$, and we hope 
that our paper will help to develop it. 
(See \cite{MMMT2} for some first steps in this
direction.)

And most of our constructions work over 
certain integral rings, see \eqref{eq:aform}.
So perhaps our work will also be useful for $2$-representation theory in 
finite characteristic.

\subsubsection*{Structure of the paper}\label{subsub:paper-sections}

\begin{enumerate}[label=(\roman*)]

\setlength\itemsep{.15cm}

\item In \fullref{sec:basics} we give the details of our graph-theoretical 
construction.

\item In \fullref{sec:2reps} we explain some notions 
concerning $2$-categories and $2$-representations, focusing on 
the graded and weak setups 
(by extending Mazorchuk and Miemietz's
constructions).

\item In \fullref{sec:dicat} we recall Elias' two-color Soergel calculus 
and define our $2$-action of it. Hereby we follow 
-- and generalize -- ideas from \cite{AT1} and \cite{KS1}.

\item Finally, \fullref{sec:proofs} contains all proofs.

\end{enumerate}

\begin{remarkn}\label{remark:colors}
We use colors in this paper, 
but they are not necessary to read the paper 
and just a service to the reader. 
Nevertheless, three colors should be mentioned, i.e. 
{\color{sea}sea-green} and {\color{tomato}tomato} 
denote the two different generators 
$\dihsgen$ and $\dihtgen$ of the dihedral 
groups, while {\color{orchid} dark orchid} 
denotes notions which play the role of a dummy and can be replaced 
by either of the two. Moreover, our notation is designed so that 
these colors can be distinguished 
in black-and-white:
\smallskip
\begin{enumerate}[label=(\roman*)]

\setlength\itemsep{.15cm}

\item The color {\color{sea}sea-green} is 
always either accompanied by the 
symbol $\dihsgen$, the associated notions are underlined 
or we use $\bullets$.

\item The color {\color{tomato}tomato} is 
always either accompanied by the 
symbol $\dihtgen$, the associated notions are overlined 
or we use $\bullett$.

\item The color {\color{orchid} dark orchid} is 
always either accompanied by the 
symbols $\dihugen$, $\dihvgen$ or $\dihwgen$, the 
associated notions 
are neither 
under- nor overlined 
or we use $\bulletst$.\makeqedtri

\end{enumerate}
\end{remarkn}
\noindent \textbf{Acknowledgements.} We especially 
like to thank Ben Elias, Nick Gurski and Volodymyr Mazorchuk for patiently 
answering our questions, and Geordie Williamson for a very fruitful blackboard discussion.

We also thank Nils Carqueville, Michael Ehrig, Lars Thorge Jensen, 
Vanessa Miemietz, SageMath, Antonio Sartori and Paul Wedrich for 
helpful comments and discussions, as well as Volodymyr Mazorchuk 
for many detailed and very helpful comments on a draft of this paper. 
Special thanks to the anonymous referee for a careful reading of this paper,
for very helpful suggestions, for 
pointing out an error in the first
version of this paper, and for improving the ``grammar of the world''.

M.M. likes to thank the Hausdorff Center for Mathematics in Bonn 
for sponsoring a research visit. D.T. thanks the 2016 European Championship 
for providing a quiet working atmosphere in the mathematical institute during which 
a big part of his work on this paper was done. 
M. M. is partially supported by FCT/Portugal
through the project UID/MAT/04459/2013.
\section{Bipartite graphs and dihedral \texorpdfstring{$2$}{2}-representations}\label{sec:basics}
In this section we state our main results 
(see \fullref{subsec:main-theorems}), and provide the background needed to understand them.
\subsection{Combinatorics of dihedral groups}\label{subsec:dihedral-stuff}

First, we recall some basic 
notions concerning the dihedral groups.

\subsubsection{The dihedral group and its associated Hecke algebra}\label{subsub:hecke-stuff}
 
We follow \cite{El1} with our conventions. Let
\begin{gather*}
\mathrm{W}_n=\langle\dihsgen,\dihtgen|\dihsgen^2=\dihtgen^2=1,
\dihsgen_n=\underbrace{\dots\dihsgen\dihtgen\dihsgen}_n=
\underbrace{\dots\dihtgen\dihsgen\dihtgen}_n=\dihtgen_n
\rangle\quad\text{and}\quad
\mathrm{W}_{\infty}=\langle\dihsgen,\dihtgen|\dihsgen^2=\dihtgen^2=1
\rangle
\end{gather*}
be the \textit{dihedral groups} of order $2n\in\Z_{>0}$ and of infinite order 
respectively, presented as the rank two Coxeter groups of type $\dihedral{n}$ and $\dihedral{\infty}$.
When no confusion is possible, we write $\mathrm{W}$ for either 
$\mathrm{W}_n$ or $\mathrm{W}_{\infty}$. The two generators $\dihsgen$ and $\dihtgen$ 
are always {\color{sea}sea-green s} and {\color{tomato}tomato t} colored. 
(Throughout, we allow $n=1$ which is to be understood 
by dropping one color, say 
{\color{tomato}tomato t}, and all notions involving it.)

Here and in the following, we denote by $\dihsgen_k$ 
respectively by $\dihtgen_k$ a sequence of length $k$, alternating in 
$\dihsgen$ and $\dihtgen$ with rightmost symbol $\dihsgen$ respectively 
$\dihtgen$. 
In general, we write an element of $\mathrm{W}$ as a finite word $\word=\word_l\cdots \word_1$ with $\word_k\in\{\dihsgen,\dihtgen\}$ (including the empty word $\emptyset$). We say that the word is reduced if it is equal to $\dihsgen_k$ or $\dihtgen_k$ for some $k\in\Z_{\geq 0}$. 
Throughout the remainder of this paper, when we write $\word\in\mathrm{W}$ as a word, we always assume that the word is reduced if not stated otherwise. 

Moreover, let $\mathrm{H}_n$ or $\mathrm{H}_{\infty}$ denote the associated \textit{Hecke algebra} over $\Cv$, with $\vpar$ being an indeterminate. Again, when no confusion is possible, we write $\mathrm{H}$ for $\mathrm{H}_n$ or $\mathrm{H}_\infty$. This algebra 
has generators $T_{\dihsgen}$ and $T_{\dihtgen}$, and relations
\[
T_{\dihsgen}^2=(\vpar^{-2}-1)\cdot T_{\dihsgen}+\vpar^{-2},\quad\quad
T_{\dihtgen}^2=(\vpar^{-2}-1)\cdot T_{\dihtgen}+\vpar^{-2},\quad\quad
T_{\dihsgen_n}=T_{\dihtgen_n}.
\]
For any $\word\in\mathrm{W}$, pick a reduced word $\word_l\cdots \word_1$ which represents it. Then the element $T_\word\in\mathrm{H}$ denotes the product $T_{\word_l}\cdots T_{\word_1}$ and $T_{\emptyset}=1$. The element 
$T_\word$ does not depend on the choice of the reduced word 
and $\{T_\word\mid \word\in\mathrm{W}\}$ forms a basis of $\mathrm{H}$. So 
for $\vpar=1$ (when working over $\CV$) we recover $\C[\mathrm{W}]$.

\subsubsection{The dihedral Kazhdan--Lusztig basis}\label{subsub:KL-stuff}

Denote by $\ell(\word)$ the length of $\word\in\mathrm{W}$ and by 
$\leq$ the Bruhat order on $\mathrm{W}$. For any $\word\in\mathrm{W}$ define
\[
\klbasis_{\word}=\vpar^{\ell(\word)}\cdot{\textstyle\sum_{\word^{\prime}\leq\word}}
T_{\word^{\prime}},\quad \word,\word^{\prime}\in\mathrm{W}.
\]
The set $\{\klbasis_{\word}\mid\word\in\mathrm{W}\}$ forms a basis 
of $\mathrm{H}$, called the \textit{Kazhdan--Lusztig basis}. 

For an expression $\word=
\word_{l}\cdots\word_1$
for $\word\in\mathrm{W}$ -- not necessarily 
reduced -- let $\klbasis_{\overline{\word}}=
\klbasis_{\word_{l}}\cdots\klbasis_{\word_1}$. This element does depend 
on the choice of expression for $\word$, even amongst reduced expressions. 
Choosing one reduced expression for each $\word\in\mathrm{W}$, we get another 
basis $\{\klbasis_{\overline{\word}}=
\klbasis_{\word_{\ell(\word)}}\cdots\klbasis_{\word_1}\mid \word\in\mathrm{W}\}$ 
for $\mathrm{H}$, called the \textit{Bott--Samelson basis}, cf. \fullref{subsec:diacat}. 
(Note that we write $\klbasis_{\overline{\word}}$ to distinguish the Bott--Samelson from 
the Kazhdan--Lusztig basis.)

\begin{example}\label{example:KLbasis}
One has $\disgen=\vpar(T_{\dihsgen}+1)$, $\ditgen=\vpar(T_{\dihtgen}+1)$, 
and an easy calculation gives 
$\klbasis_{\dihsgen\dihtgen\dihsgen}=\klbasis_{\overline{\dihsgen\dihtgen\dihsgen}}-\disgen
=\disgen\ditgen\disgen-\disgen$. In general, $\klbasis_{\word}=\klbasis_{\overline{\word}} \mp \text{``lower order terms''}$.
\end{example}

For any $n\in\Z_{>0}$, the group $\mathrm{W}_n$ has a unique longest element
\[
\wnull(n)=\wnull=\dihsgen_n=\dihtgen_n.
\]
(In case ``$n=\infty$'' there is no such $\wnull$.) 
In this paper, we only consider categorifications of $\mathrm{H}_n$-representations which are killed by $\klbasis_{\wnull(n)}=\klbasis_{\wnull}$. In fact, the only decategorification of a simple 
transitive $2$-representation of $\mathrm{H}_n$ 
which is not killed by $\klbasis_{\wnull}$ is, 
by \cite[Theorem 18]{MM5}, the trivial 
$\mathrm{H}_n$-representation, cf. \fullref{theorem:classification}.

\subsubsection{The defining relations satisfied by the Kazhdan--Lusztig basis elements}\label{subsub:KL-stuff-genrel}

Define 
recursively the integers $d^k_{l}$ via:
\begin{gather}\label{eq:numbers}
\begin{aligned}
d_1^1=1,\quad d^k_{l}=0,&\quad\text{ unless }0<k\leq l\text{ and }l-k\text{ is even},\\\
d^k_{l}&=d_{l-1}^{k-1}-d_{l-2}^{k}.
\end{aligned}
\end{gather}
Let $[2]_{\vpar}=\vpar+\vpar^{-1}$. When considering the basis $\{\klbasis_{\overline{\word}}\mid \word\in\mathrm{W}\}$, the 
defining relations of $\mathrm{H}$ -- by e.g. \cite[Section 2.2]{El1} -- are
\begin{gather}\label{eq:def-relationsa}
\disgen\disgen=[2]_{\vpar}\cdot\disgen,\quad\quad
\ditgen\ditgen=[2]_{\vpar}\cdot\ditgen,
\\
\label{eq:def-relations}
{\textstyle\sum_{k\in\Z_{\geq 0}}}\, d^k_n\cdot 
\klbasis_{\overline{\dihsgen_k}}={\textstyle\sum_{k\in\Z_{\geq 0}}}\, d^k_n\cdot \klbasis_{\overline{\dihtgen_k}}.
\end{gather}
Note that either sum in \eqref{eq:def-relations} is equal to $\klbasis_{\wnull}$. 

\begin{example}\label{example:reorder}
The first few non-zero numbers from \eqref{eq:numbers} 
are as follows.
\[
\begin{tikzpicture}[anchorbase, scale=.45]
	\draw [thick] (.5,.5) to (.5,-7.5);
	\draw [thick] (-.5,-.5) to (7.5,-.5);
	\node at (-1,-4) {$l$};
	\node at (4,1) {$k$};
	\node at (0,-1) {$1$};
	\node at (0,-2) {$2$};
	\node at (0,-3) {$3$};
	\node at (0,-4) {$4$};
	\node at (0,-5) {$5$};
	\node at (0,-6) {$6$};
	\node at (0,-6.75) {$\vdots$};
	\node at (1,0) {$1$};
	\node at (2,0) {$2$};
	\node at (3,0) {$3$};
	\node at (4,0) {$4$};
	\node at (5,0) {$5$};
	\node at (6,0) {$6$};
	\node at (7,0) {$\,\cdots$};
	\node at (1,-1) {$1$};
	\node at (2,-2) {$1$};
	\node at (3,-3) {$1$};
	\node at (1,-3) {$-1$};
	\node at (4,-4) {$1$};
	\node at (2,-4) {$-2$};
	\node at (5,-5) {$1$};
	\node at (3,-5) {$-3$};
	\node at (1,-5) {$1$};
	\node at (6,-6) {$1$};
	\node at (4,-6) {$-4$};
	\node at (2,-6) {$3$};
	\node at (7,-6.75) {$\ddots$};
	\node at (5,-6.75) {$\ddots$};
	\node at (3,-6.75) {$\ddots$};
	\draw [thick, rotate around={315:(1.65,-4.5)}, dotted] (1.65,-4.5) ellipse (.55cm and 1.35cm);
	\draw [thick, dotted, ->] (2,-4.9) to (2,-5.7);
\end{tikzpicture}.
\]
Hence, if $n=l=3$, then \eqref{eq:def-relations} gives
\begin{gather}\label{eq:n-is-three}
\disgen\ditgen\disgen-
\disgen
=
\ditgen\disgen\ditgen-
\ditgen\;(=\klbasis_{\wnull}).
\end{gather}
When ``$n=\infty$'', the only relations are the ones from \eqref{eq:def-relationsa}.
\end{example}
\subsection{Bipartite graphs}\label{subsec:bipartite-stuff}

Next, we recall some basics about bipartite 
graphs and fix notation which we use throughout.

\subsubsection{A reminder on bipartite graphs}\label{subsub:basic-graphs}

Let $\g$ be a connected, unoriented, finite graph without loops 
and with at most one edge between each pair of vertices.
Let $V=V(\g)$ be the set of vertices of $\g$ and assume that the vertices are numbered. If these numbers are divided into 
two disjoint subsets $I$ and $J$, such that     
\[
V=
\graphs
\,{\textstyle\coprod}\,
\grapht,
\quad\quad
\graphs=\{\bbii{i}\mid i\in I\},
\quad\quad
\grapht=\{\bbjj{j}\mid j\in J\},
\]
with no edges connecting vertices within each set, then we call the triple $(\g,\graphs,\grapht)$ a {\em bipartite graph}. When no confusion is possible, we simply write $\g$ for a bipartite graph. Note that a bipartition of $\g$ is the same as a two-coloring of its vertices.

Two bipartite graphs $(\g,\graphs,\grapht)$ and $(\g^{\prime},\graphs^{\prime},\grapht^{\prime})$ 
are called \textit{isomorphic} if there is an isomorphism of graphs between $\g$ and $\g^{\prime}$ which sends $\graphs$ to $\graphs^{\prime}$ and $\grapht$ to $\grapht^{\prime}$.

\begin{example}\label{example-bigraph}
One crucial example in this paper is the type $\typee{6}$ graph:
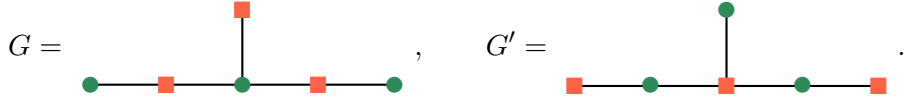
\[
\g=
\begin{tikzpicture}[anchorbase, scale=1]
	\draw [thick] (0,0) to (4,0);
	\draw [thick] (2,0) to (2,1);
	\node at (0,-.01) {\Large $\bullets$};
	\node at (1,-.01) {\Large $\bullett$};
	\node at (2,-.01) {\Large $\bullets$};
	\node at (3,-.01) {\Large $\bullett$};
	\node at (4,-.01) {\Large $\bullets$};
	\node at (2,.99) {\Large $\bullett$};
\end{tikzpicture}
,\quad\quad
\g^{\prime}=
\begin{tikzpicture}[anchorbase, scale=1]
	\draw [thick] (0,0) to (4,0);
	\draw [thick] (2,0) to (2,1);
	\node at (0,0) {\Large $\bullett$};
	\node at (1,0) {\Large $\bullets$};
	\node at (2,0) {\Large $\bullett$};
	\node at (3,0) {\Large $\bullets$};
	\node at (4,0) {\Large $\bullett$};
	\node at (2,1) {\Large $\bullets$};
\end{tikzpicture}.
\]
These two-colorings give non-isomorphic bipartite 
graphs of type $\typee{6}$. As we will see later, these will give rise to 
two inequivalent $2$-representations categorifying the same 
$\mathrm{H}$-module, see \fullref{example:SVD} 
(keeping \fullref{example-scequal} in mind).
\end{example}

We write $\bbii{i}\gconnect{}\bbjj{j}(=\bbjj{j}\gconnect{}\bbii{i})$ in case 
$\bbii{i}$ and $\bbjj{j}$ are connected in $\g$.

\subsubsection{Adjacency matrices and spectra}\label{subsub:basic-graphs-spectra}

Given any graph $\g$, the 
{\em adjacency matrix} $A(\g)$ of 
$\g$ is the symmetric $|V|\times|V|$-matrix 
whose only non-zero entries are $A(\g)_{\bbii{i},\bbjj{j}}=1$ for 
$\bbii{i}\gconnect{}\bbjj{j}$.
The {\em spectrum} $S_{\g}$ of $\g$ is the multiset of all eigenvalues of $A(\g)$ 
(which are all real), repeating each one of them according to its multiplicity. 

Since we assume $\g$ to be bipartite, we can clearly choose an ordering of the vertices 
(which we will always do from now on) such that
\begin{gather}\label{eq:isadmatrix}
A(\g)=
\begin{pmatrix}
0 & A\\
A^{\mathrm{T}} & 0
\end{pmatrix}
\end{gather}
for some matrix $A$ of size $|\graphs|\times|\grapht|$. 
Next, recall 
that $S_{\g}$ is a 
symmetric set (see e.g. \cite[Proposition 3.4.1]{BH1}), and
the fact from linear algebra that
\begin{gather}\label{eq:eigenvalue}
AA^{\mathrm{T}}\text{ has an eigenvalue }\alpha\neq 0
\;\Leftrightarrow\;
A^{\mathrm{T}}A\text{ has an eigenvalue }\alpha\neq 0.
\end{gather}
Thus, the non-zero elements (counting multiplicities) of $S_{\g}$ are $\pm\sqrt{\alpha}$ 
for $\alpha$ as in \eqref{eq:eigenvalue}.

\subsubsection{Spectrum-color-equivalence}\label{subsub:basic-graphs-equi}

Given two bipartite graphs $\g$ and $\g^{\prime}$. 
We call them
{\em spectrum-color-equivalent} 
if $\vert \graphs\vert =\vert \graphs^{\prime}\vert $, $\vert \grapht \vert =\vert \grapht^{\prime}\vert $ and $S_{\g}=S_{\g^{\prime}}$, and {\em spec\-trum-color-inequivalent} otherwise. This clearly gives rise to an equivalence relation.

\begin{example}\label{example-scequal}
Take the two graphs $\g$ and $\g^{\prime}$ from \fullref{example-bigraph}. 
Then $\g$ and $\g^{\prime}$ are non-isomorphic as bipartite graphs, but they 
are spectrum-color-equivalent. 

More generally, any 
bipartite graph with $\vert\graphs\vert=\vert\grapht\vert$ is 
spectrum-color-equivalent 
to the bipartite graph with the opposite two-coloring.
\end{example}
\subsection{Quivers and categorical representations of dihedral groups}\label{subsec:quiver-stuff}

Fix a bipartite graph $\g$. The {\em double quiver} $\qalg_{\g}$ associated to $\g$ is the oriented 
graph obtained from $\g$ by doubling each edge and giving opposite orientations 
to the two resulting edges, called arrows. Such an arrow is denoted by $\bbjj{j}|\bbii{i}$ if it starts at 
$\bbii{i}$ and ends at $\bbjj{j}$, and by $\bbii{i}|\bbjj{j}$ if 
it starts at 
$\bbjj{j}$ and ends at $\bbii{i}$ (i.e. we are using the 
``operator notation'').

\begin{definition}\label{definition-partner}
Two distinct arrows of $\qalg_{\g}$ 
are called 
\textit{partners} if they come from the same edge in $\g$.
\end{definition}

Thus, each arrow  
has precisely one partner pointing in the opposite 
direction.

\begin{example}\label{example-double}
An example of a double-quiver is:
\[
\g=
\begin{tikzpicture}[anchorbase, scale=1]
	\draw [thick] (.15,0) to (.85,0);
	\draw [thick] (1.15,0) to (1.85,0);
	\node at (0,0) {$\bbi{1}$};
	\node at (1,0) {$\bbj{2}$};
	\node at (2,0) {$\bbi{3}$};
\end{tikzpicture}
\;
\rightsquigarrow
\;
\qalg_{\g}=
\begin{tikzpicture}[anchorbase, scale=1]
	\draw [->] (.15,.075) to (.85,.075);
	\draw [<-] (.15,-.075) to (.85,-.075);
	\draw [->] (1.15,.075) to (1.85,.075);
	\draw [<-] (1.15,-.075) to (1.85,-.075);
	\node at (0,0) {$\bbi{1}$};
	\node at (1,0) {$\bbj{2}$};
	\node at (2,0) {$\bbi{3}$};
\end{tikzpicture}.
\]
In the above example, $\bbjj{2}|\bbii{1}$ and $\bbii{1}|\bbjj{2}$ are partners, and so 
are $\bbii{3}|\bbjj{2}$ and $\bbjj{2}|\bbii{3}$.
\end{example}

As already pointed out, we will see that 
only bipartite graphs of $\ADE$ type give rise to 
$2$-representations of Soergel bimodules in finite dihedral type.

\subsubsection{The path algebra associated to a bipartite graph}\label{subsub:path-alg}

We work over certain base rings $\Zq$, $\ZQ$ or $\zgl$, all of which 
are subrings of $\C$. 
We will impose some technical conditions on these 
which we discuss in \fullref{remark:ground-ring},
and the reader should think of these 
as playing the role of an integral form.
 
If no confusion can arise, we 
simply write $\zq$ for either $\Zq$, $\ZQ$ or $\zgl$.
Of course, we can always extend the scalars to $\C$ 
if necessary.

Let $\pathG$ be the path algebra of $\qalg_{\g}$ over $\zq$, such that 
multiplication on the left is given by post-composition, on  
the right by pre-composition. We denote the multiplication by $\pathm$ and consider 
$\pathG$ to be graded by the path length.

By convention, 
$\bbii{i}$ and $\bbjj{j}$ denote the corresponding paths of length zero. Paths of length one are in one-to-one correspondence 
with arrows of $\qalg_{\g}$, and we call them arrows too.

\begin{definition}\label{definition:quiverG}
Let $\algG$ denote the quotient algebra obtained from 
$\pathG$ by the following relations.
\smallskip
\begin{enumerate}[label=$\vartriangleright$]

\setlength\itemsep{.15cm}

\item \textbf{The two and three steps relations.}
\renewcommand{\theequation}{\arabic{section}.QG1}
\begin{gather}\label{eq:two-steps}
\text{The composite of two arrows is zero unless they are partners.}
\end{gather}
\renewcommand{\theequation}{\arabic{section}.QG2}
\begin{gather}\label{eq:three-steps}
\text{The composite of three arrows is zero.}
\end{gather}

\item \textbf{All non-zero partner composites are equal.} 
Assume that $\bbii{i}\gconnect{}\bbjj{j}_a$, 
for $a=1,\dots,b$. Then:
\renewcommand{\theequation}{\arabic{section}.QG3}
\begin{gather}\label{eq:partner}
\bbii{i}|\bbjj{j}_1\pathm\bbjj{j}_1|\bbii{i}
=
\dots
=
\bbii{i}|\bbjj{j}_b\pathm\bbjj{j}_b|\bbii{i}
=
\bbii{i}|\bbii{i}
,
\end{gather}
\renewcommand{\theequation}{\arabic{section}-\arabic{equation}}
where the path $\bbii{i}|\bbii{i}$, called \textit{loop}, is defined 
by the above equations.

Similar relations hold with $\bbii{i}$ and $\bbjj{j}$ swapped. 
\end{enumerate}
\smallskip
(Note that the relation \eqref{eq:three-steps} is a consequence 
of \eqref{eq:two-steps} and \eqref{eq:partner} as 
long as $\g$ has two or more edges.)
We call $\algG$ the {\em type $\g$-quiver algebra}.
\end{definition}

The defining relations of $\algG$ are homogeneous, so $\algG$ inherits the path length grading. 
Clearly, the primitive idempotents of $\algG$ are the $\bbk{i}$'s.

\begin{example}\label{example:typeAquiver}
The relations \eqref{eq:two-steps}, \eqref{eq:three-steps} 
and \eqref{eq:partner} might be familiar 
to readers who know about the so-called zigzag algebras 
in the spirit of \cite{HK1}. 

Let us give two examples. Fix $m\in\Z_{\geq 0}$. Let $\g$ be a Dynkin graph of type 
$\typea{m}$ or $\typeat{2m-1}$ with $m$, respectively $2m$ vertices. 
Then the associated $\qalg_{\g}$'s are of the following form 
(if we consider the evident two-coloring of $\g$).
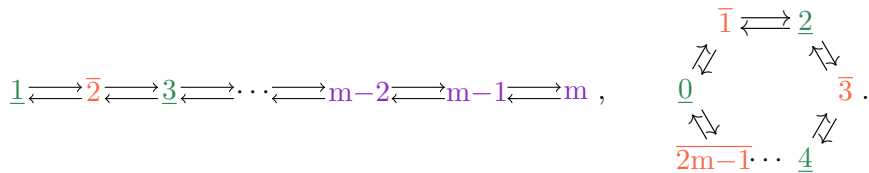
\begin{gather*}
\begin{tikzpicture}[anchorbase, scale=1]
	\draw [->] (.15,.075) to (.85,.075);
	\draw [->] (.85,-.075) to (.15,-.075);
	\draw [->] (1.15,.075) to (1.85,.075);
	\draw [->] (1.85,-.075) to (1.15,-.075);
	\draw [->] (2.15,.075) to (2.85,.075);
	\draw [->] (2.85,-.075) to (2.15,-.075);
	\draw [->] (3.35,.075) to (4.05,.075);
	\draw [->] (4.05,-.075) to (3.35,-.075);
	\draw [->] (4.925,.075) to (5.625,.075);
	\draw [->] (5.625,-.075) to (4.925,-.075);
	\draw [->] (6.45,.075) to (7.15,.075);
	\draw [->] (7.15,-.075) to (6.45,-.075);
	\node at (0,0) {$\bbi{1}$};
	\node at (1,0) {$\bbj{2}$};
	\node at (2,0) {$\bbi{3}$};
	\node at (3.1,0) {$\,\cdots$};
	\node at (4.5,0) {$\bbk{m{-}2}$};
	\node at (6.05,0) {$\bbk{m{-}1}$};
	\node at (7.35,0) {$\bbk{m}$};
\end{tikzpicture}
,\quad\quad
\xy 0;/r.25pc/:
(-10,0)*+{\bbi{0}}="1";
(-5,8.66)*+{\bbj{1}}="2";
(5,8.66)*+{\bbi{2}}="3";
(10,0)*+{\bbj{3}}="4";
(5,-8.66)*+{\bbi{4}}="5";
(-5,-8.66)*+{\hspace*{-.3cm}\bbj{2m{-}1}}="6";
(0,-8.66)*+{\cdots}="7";
{\ar@<2pt>@{->} "1";"2" };
{\ar@<2pt>@{->} "2";"3" };
{\ar@<2pt>@{->} "3";"4" };
{\ar@<2pt>@{->} "4";"5" };
{\ar@<2pt>@{->} "1";"6" };
{\ar@<2pt>@{->} "2";"1" };
{\ar@<2pt>@{->} "3";"2" };
{\ar@<2pt>@{->} "4";"3" };
{\ar@<2pt>@{->} "5";"4" };
{\ar@<2pt>@{->} "6";"1" };
\endxy.
\end{gather*}
The defining relations of the associated 
quiver algebras are of the form
\begin{gather*}
\eqref{eq:two-steps}\colon\quad
\bbk{i{+}2}|\bbk{i{+}1}\pathm\bbk{i{+}1}|\bbk{i}
=0=
\bbk{i{-}2}|\bbk{i{-}1}\pathm\bbk{i{-}1}|\bbk{i},
\\
\eqref{eq:partner}\colon\quad
\bbk{i}|\bbk{i{+}1}\pathm\bbk{i{+}1}|\bbk{i}
=
\bbk{i}|\bbk{i}
=
\bbk{i}|\bbk{i{-}1}\pathm\bbk{i{-}1}|\bbk{i},
\end{gather*}
with indices modulo $2m$ in the cyclic case. 

The quiver algebra for the Dynkin graph of type $\typea{m}$ will be
of importance later on and we denote it by $\algA{m}$. Similarly, 
we denote by $\algAt{2m-1}$ its cyclic counterpart.
\end{example}

In the above and throughout; for $m=0$ and $m=1$ we let $\algG=\algA{0}=\zq$ 
and $\algG=\algA{1}\cong\zq[X]/(X^2)$, by convention.

\begin{remark}\label{remark:typeAquiver}
The algebras from \fullref{example:typeAquiver} 
also appear in the context of categorical braid 
group actions, e.g. in \cite{GTW} and \cite{KS1}. Note that Khovanov--Seidel 
have the additional relation $\bbii{1}|\bbjj{2}|\bbii{1}=0$.
\end{remark}

\subsubsection{Some bimodules}\label{subsub:path-alg-bimod}

Let us denote by $P_{\bbk{i}}$ and ${}_{\bbk{i}}P$ the 
left and the right ideal of $\algG$ generated 
by $\bbk{i}$, respectively. They are 
clearly graded $\algG$-modules. 

\begin{example}\label{example:projectives-via-arrows}
We can easily visualize 
them for $\algA{m}$ with $1<i<m$ as
\begin{gather}\label{eq:projectives}
\raisebox{-.35cm}{$P_{\bbk{i}}=
\raisebox{.35cm}{\begin{tikzpicture}[anchorbase, scale=1]
	\draw [->] (.85,-.075) to (.15,-.075);
	\draw [->] (1.15,.075) to (1.85,.075);
	\draw [directed=.825] (.85,.15) to [out=135, in=180] (1,.7) to [out=0, in=45] (1.15,.15);
	\fill [white] (1,.7) circle (.2cm);
	\node at (1,0) {$\bbk{i}$};
	\node at (1,.7) {\tiny $\loopy{\bbk{i}}$};
	\node at (1.5,-.125) {\tiny $\bbk{i{+}1}|\bbk{i}$};
	\node at (.5,.125) {\tiny $\bbk{i{-}1}|\bbk{i}$};
\end{tikzpicture}}
,\quad\quad
{}_{\bbk{i}}P=
\raisebox{.35cm}{\begin{tikzpicture}[anchorbase, scale=1]
	\draw [->] (.15,.075) to (.85,.075);
	\draw [->] (1.85,-.075) to (1.15,-.075);
	\draw [directed=.825] (.85,.15) to [out=135, in=180] (1,.7) to [out=0, in=45] (1.15,.15);	
	\fill [white] (1,.7) circle (.2cm);
	\node at (1,0) {$\bbk{i}$};
	\node at (1,.7) {\tiny $\loopy{\bbk{i}}$};
	\node at (1.55,.125) {\tiny $\bbk{i}|\bbk{i{+}1}$};
	\node at (.5,-.125) {\tiny $\bbk{i}|\bbk{i{-}1}$};
\end{tikzpicture}}$.}
\end{gather}
They have basis vectors 
$\bbk{i},\bbk{i{-}1}|\bbk{i},\bbk{i{+}1}|\bbk{i},\loopy{\bbk{i}}$ and $\bbk{i},\bbk{i}|\bbk{i{-}1},\bbk{i}|\bbk{i{+}1},\loopy{\bbk{i}}$, respectively, of degree $0,1,1,2$. So they are 
free of $\zq$-rank three for any  
$m\neq 0,1$.
\end{example}

Because the $\bbk{i}$ form a complete set of orthogonal primitive idempotents in $\algG$, the $P_{\bbk{i}}$ and the ${}_{\bbk{i}}P$ are graded, indecomposable projective modules and all graded, 
indecomposable left -- respectively right -- $\algG$-modules are of the form $P_{\bbk{i}}\{\shiftdeg\}$, 
respectively ${}_{\bbk{i}}P\{\shiftdeg\}$, for some shift $\shiftdeg\in\Z$.
Let $\otimes=\otimes_{\zq}$ be the 
tensor product over $\zq$. Then  
$P_{\bbk{i}}\{\shiftme{-1}\}\otimes {}_{\bbk{i}}P$ is 
clearly a graded $\algG$-bimodule. 
These bimodules will be important in this paper.

\subsubsection{Endofunctors associated to bipartite graphs}\label{subsub:endo-stuff}

We denote by $\cellcatG=\Modl{\algG}$ 
the category of graded projective,  
(left) $\algG$-modules, which are free of finite $\zq$-rank.
Our next goal, following \cite[Section 2]{KS1} and \cite[Section 3]{AT1}, 
is to define 
endofunctors 
\[
\bifunctor_{\bbk{i}}\colon\cellcatG\to\cellcatG,\quad i\in\g. 
\] 
Let $\qgtimes=\otimes_{\algG}$ be the 
tensor product over $\algG$. We define
\begin{gather*}
\bifunctor_{\bbk{i}}(X)=P_{\bbk{i}}\{\shiftme{-1}\}\otimes {}_{\bbk{i}}P\qgtimes X,\quad\quad
\bifunctor_{\bbk{i}}(f)=\idmap_{P_{\bbk{i}}\{\shiftme{-1}\}\otimes {}_{\bbk{i}}P}\qgtimes f.
\end{gather*}
Here $X,Y\in\cellcatG$ and $f\in\Hom_{\cellcatG}(X,Y)$.
(Since $P_{\bbk{i}}\{\shiftme{-1}\}\otimes {}_{\bbk{i}}P$ is clearly
{\em biprojective}, i.e. projective both as a left and as a right $\algG$-module, 
the functor $\bifunctor_{\bbk{i}}$ sends graded projectives to graded projectives.) As 
in \cite[Section 3.3]{AT1} we immediately obtain
\begin{gather}\label{eq:some-relations-infty}
\bifunctor_{\bbk{i}}(P_{\bbk{j}})\cong\begin{cases}P_{\bbk{i}}\{\shiftme{-1}\}\oplus P_{\bbk{i}}\{\shiftme{+1}\},&\text{if }\bbk{i}=\bbk{j},\\
P_{\bbk{i}},&\text{if }\bbk{i}\gconnect{}\bbk{j},\\
0,& \text{otherwise}.
\end{cases}
\end{gather}

\begin{example}\label{example:calc-functor}
By looking at \eqref{eq:projectives} one observes that 
${}_{\bbk{i}}P\qgtimes P_{\bbk{i}}\cong \zq(\bbk{i})\oplus\zq(\bbk{i}|\bbk{i})$. 
This iso\-morphism is given by 
multiplication of paths.
Consequently, the degree two $\algG$-en\-domorphism 
$\bifunctor_{\bbk{i}}(\bbk{i}|\bbk{i})$ of 
$P_{\bbk{i}}\{\shiftme{-1}\}\oplus P_{\bbk{i}}\{\shiftme{+1}\}$ sends 
the copy $P_{\bbk{i}}\{\shiftme{-1}\}$ identically to $P_{\bbk{i}}\{\shiftme{+1}\}$, 
and is zero elsewhere (recall that 
$\bbk{i}|\bbk{i}\pathm\bbk{i}|\bbk{i}=0$).
\end{example}

Following \cite[Section 3.3]{AT1} we define
\begin{gather*}
\functors={\textstyle\bigoplus_{\,\bbii{i}\in\g}}\bifunctor_{\bbii{i}},\quad\quad
\functort={\textstyle\bigoplus_{\,\bbjj{j}\in\g}}\bifunctor_{\bbjj{j}}.
\end{gather*}

\begin{example}\label{example:quiver-sum}
We sum over the graph of type $\typea{m}$ as (in case $m$ is odd):
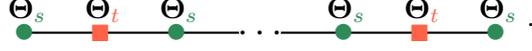
\[
\begin{tikzpicture}[anchorbase, scale=1]
	\draw [thick] (0,0) to (2.825,0);
	\draw [thick] (3.375,0) to (6.2,0);
	\node at (0,-.01) {\Large $\bullets$};
	\node at (1,-.01) {\Large $\bullett$};
	\node at (2,-.01) {\Large $\bullets$};
	\node at (3.1,-.01) {\Large $\,\cdots$};
	\node at (4.2,-.01) {\Large $\bullets$};
	\node at (5.2,-.01) {\Large $\bullett$};
	\node at (6.2,-.01) {\Large $\bullets$};
	\node at (0,.3) {$\,\functors$};
	\node at (1,.3) {$\,\functort$};
	\node at (2,.3) {$\,\functors$};
	\node at (4.2,.3) {$\,\functors$};
	\node at (5.2,.3) {$\,\functort$};
	\node at (6.2,.3) {$\,\functors$};
\end{tikzpicture}.
\]
Note that the opposite two-coloring of $\g$ switches 
$\functors$ and $\functort$. In general, this switch need not be a natural isomorphism.
\end{example}

By using \eqref{eq:some-relations-infty}, one directly checks that
\begin{gather}\label{eq:functor-projectives-infty}
\begin{aligned}
\functors(P_{\bbk{i}})\cong
\begin{cases}
P_{\bbii{i}}\{\shiftme{-1}\}\oplus P_{\bbii{i}}\{\shiftme{+1}\},&\text{if }i\in\graphs,
\\
{\textstyle\bigoplus_{\bbii{j}\gconnect{}\bbjj{i}}}\,P_{\bbii{j}},&\text{if }i\in\grapht,\end{cases}
,
\functort(P_{\bbk{i}})\cong
\begin{cases}P_{\bbjj{i}}\{\shiftme{-1}\}\oplus P_{\bbjj{i}}\{\shiftme{+1}\},&\text{if }i\in\grapht,
\\
{\textstyle\bigoplus_{\bbii{i}\gconnect{}\bbjj{j}}}\,P_{\bbjj{j}},&\text{if }i\in\graphs.
\end{cases}
\end{aligned}
\end{gather}

\subsubsection{Dihedral modules associated to bipartite graphs}\label{subsub:di-modules-bipartite}

\begin{definition}\label{definition:module-for-G}
Let $\cellrepgrg$ be the $\Cv$-vector space 
on the basis $\{\bbk{i}\mid i\in\g\}$. 
Define a $\mathrm{H}_{\infty}$-action on $\cellrepgrg$ 
via 
\begin{equation}\label{eq:action-downstairs}
\disgen\cdot\bbii{i}=[2]_{\vpar}\cdot\bbii{i},\quad \disgen\cdot\bbjj{i}=
{\textstyle\sum_{\bbii{j}\gconnect{}\bbjj{i}}}\,\bbii{j},\quad
\ditgen\cdot\bbjj{i}=[2]_{\vpar}\cdot\bbjj{i},\quad\ditgen\cdot\bbii{i}=
{\textstyle\sum_{\bbii{i}\gconnect{}\bbjj{j}}}\,\bbjj{j}.
\end{equation}
(The reader is encouraged to verify that 
this gives indeed a $\mathrm{H}_{\infty}$-module structure.) 
We denote the associated algebra homomorphism by 
$\G\colon\mathrm{H}_{\infty}\to\cellrepgrg$.
\end{definition}

Let $\GG{\cellcatG}=K_0^{\oplus}(\cellcatG)\otimes_{\Z[\vpar,\vpar^{-1}]}\Cv$ denote the split Grothendieck group tensored with the field $\Cv$. (As usual, the grading shift decategorifies 
to multiplication by $\vpar$.) If we denote by $[\cdot]$ 
a class in $\GG{\cellcatG}$, then we get a 
so-called weak categorification 
(not to be confused with the weak -- in the sense of non-strict -- setup that we 
will meet below), as the following 
proposition shows.

\begin{proposition}\label{proposition:action-infty}
The functors $\functors$ 
and $\functort$ decategorify to $\Cv$-linear endomorphisms on $\GG{\cellcatG}$,  
which thus becomes an $\mathrm{H}_{\infty}$-module. The $\Cv$-linear map  
\begin{equation}\label{eq:matching-actions}
\zeta_{\G}\colon\cellrepgrg\stackrel{\cong}{\longrightarrow}\GG{\cellcatG},\quad \zeta_{\G}(i)=[P_{\bbk{i}}], 
\end{equation}
is an isomorphism of $\mathrm{H}_{\infty}$-modules, intertwining the actions of $\disgen$ and $[\functors]$, and that of $\ditgen$ and $[\functort]$.
\end{proposition}
 
\begin{proposition}\label{proposition:action-finite}
Let $\g$ be of $\ADE$ type and $n$ be its Coxeter number, i.e.:
\smallskip
\begin{enumerate}

\setlength\itemsep{.15cm}

\renewcommand{\theenumi}{(A)}
\renewcommand{\labelenumi}{\theenumi}

\item \label{enum:typeA} For $\g$ of type $\typea{m}$ we let $n=m+1$.

\renewcommand{\theenumi}{(D)}
\renewcommand{\labelenumi}{\theenumi}

\item \label{enum:typeD} For $\g$ of type $\typed{m}$ we let $n=2m-2$.

\renewcommand{\theenumi}{(E)}
\renewcommand{\labelenumi}{\theenumi}

\item \label{enum:typeE} For $\g$ of type $\typee{6}$ we let $n=12$, for 
$\g$ of type $\typee{7}$ we let $n=18$, and $\g$ of type $\typee{8}$ we let $n=30$.

\end{enumerate}
\smallskip
Then the $\mathrm{H}_\infty$ actions on $\cellrepgrg$ and $\GG{\cellcatG}$ 
descend to $\mathrm{H}_{n}$-actions 
which are matched by $\zeta_{\G}$ 
as in \eqref{eq:matching-actions}. (We denote 
these by $\G_n\colon\mathrm{H}_{n}\to\cellrepgrg$).

These are the only $\g$ (with more than one vertex) 
and $n$ for which this holds. 
\end{proposition}

\begin{remark}\label{remark:action-finite}
\fullref{proposition:action-finite} appears 
as a special case of \cite[Proposition 3.8]{Lu1}, and was rediscovered 
in \cite[Sections 5, 6 and 7]{KMMZ}. Note hereby that Lusztig proves 
his statement using the combinatorics of cells 
in Coxeter groups, while \cite{KMMZ} uses categorical results.
In this paper, we give an independent proof using spectral graph theory.
\end{remark}

\begin{example}\label{example:typeA-calc1}
Write $\cellcatA{3}=\Modl{\algA{3}}$ and $\cellcatAt{3}=\Modl{\algAt{3}}$.
Then $[\functors]$ 
and $[\functort]$ act on $\GG{\cellcatA{3}}$ 
and $\GG{\cellcatAt{3}}$ via
\begin{gather*}
[\functors]
=
\begin{pmatrix} 
[2]_{\vpar} & 0 & 1\\ 
0 & [2]_{\vpar} & 1\\ 
0 & 0 & 0\\ 
\end{pmatrix}
,\quad\quad
[\functort]
=
\begin{pmatrix} 
0 & 0 & 0\\ 
0 & 0 & 0\\ 
1 & 1 & [2]_{\vpar}\\ 
\end{pmatrix}
\quad(\text{in type }\typea{3}),
\\
[\functors]
=
\begin{pmatrix} 
[2]_{\vpar} & 0 & 1 & 1\\ 
0 & [2]_{\vpar} & 1 & 1\\ 
0 & 0 & 0 & 0\\
0 & 0 & 0 & 0\\ 
\end{pmatrix}
,\quad\quad
[\functort]
=
\begin{pmatrix} 
0 & 0 & 0 & 0\\ 
0 & 0 & 0 & 0\\ 
1 & 1 & [2]_{\vpar} & 0\\
1 & 1 & 0 & [2]_{\vpar}\\
\end{pmatrix}
\quad(\text{in type }\typeat{3}).
\end{gather*}
(These are
written on the bases $\{[P_{\bbii{1}}],[P_{\bbii{3}}],[P_{\bbjj{2}}]\}$ 
and $\{[P_{\bbii{0}}],[P_{\bbii{2}}],[P_{\bbjj{1}}],[P_{\bbjj{3}}]\}$.) 

These matrices give $\GG{\cellcatA{3}}$ 
and $\GG{\cellcatAt{3}}$
the structure of an $\mathrm{H}_{\infty}$-module.
Additionally, the relation from
\eqref{eq:def-relations} holds in type $\typea{3}$ since 
we have
\begin{gather*}
[\functors][\functort][\functors][\functort]
-2\cdot[\functors][\functort]
=
[\functort][\functors][\functort][\functors]
-2\cdot[\functort][\functors].
\end{gather*}
This shows that $\GG{\cellcatA{3}}$ has the structure 
of an $\mathrm{H}_{4}$-module.
\end{example}
\subsection{Strong dihedral \texorpdfstring{$2$}{2}-representations}\label{subsec:main-theorems}
In \fullref{sec:2reps} we will explain how to adapt 
Marzorchuk and Miemietz's definition of finitary $2$-categories, 
their $2$-re\-presen\-tations and related notions (see 
e.g. \cite{MM1}, \cite{MM3} or \cite{MM5}) to our setup. 
At this point of the paper it is enough to roughly see them as a ``higher analog'' of graded, 
finite-dimensional algebras and their graded, finite-dimensional representations.  

By using \fullref{proposition:action-infty}, we can identify $\GG{\cellcatG}$ with $\cellrepgrg$ as 
$\mathrm{H}$-modules. The next theorem says that the functorial action of $\disfun$ and $\ditfun$  on $\cellcatG$ 
can be extended to a $2$-representation of $\dihedralcat{\infty}$ and/or $\dihedralcat{n}$, 
the $2$-categories -- defined by generating $2$-morphisms and relations --
given by the two-color Soergel calculi, due to Elias \cite{El1}. 
(We will recall them in \fullref{subsec:diacat}.)

We use the same base 
ring $\zq$ to define $\dihedralcat{\infty}$ and $\dihedralcat{n}$ as we did for 
the quiver algebras in \fullref{subsec:quiver-stuff}. (More details will be given in 
\fullref{subsec:diacat}.) Elias' construction of $\dihedralcat{\infty}$ requires the 
choice of an invertible element $\qpar$ in the base ring (which 
is ultimately embedded in the complex numbers). When 
$\qpar$ is not a root of unity, $\dihedralcat{\infty}$ is 
equivalent to the $2$-category of Soergel bimodules in type 
$\dihedral{\infty}$. When $\qpar$ is 
a complex, primitive $2n$th root of unity, this is no longer 
true -- cf. \fullref{remark-gg-goes-wrong} -- but $\dihedralcat{n}$ is equivalent 
to the $2$-category of Soergel bimodules in type $\dihedral{n}$.

\begin{theoremmain}(\textbf{Dihedral $2$-actions.})\label{theorem:main-theorem}
\begin{enumerate}[label=(\alph*)]

\setlength\itemsep{.15cm}

\item For any bipartite graph $\g$, there is at least one value of 
$\qpar\in\C-\{0\}$ for which there exists 
an additive, degree-preserving, $\zq$-linear, weak $2$-functor
(defined in \fullref{subsec:diacataction})
\[
\functorG\colon\dihedralcatbig{\infty}\to\Endff(\cellcatG).
\]

\item If $\g$ is as in \ref{enum:typeA}, \ref{enum:typeD} or \ref{enum:typeE} 
and $\qpar$ is a complex, primitive $2n$th root of unity, then $\functorG$ gives rise to an additive, degree-preserving, $\zq$-linear, weak $2$-functor 
\[
\functorADE\colon\dihedralcatbig{n}\to\Endff(\cellcatADE)
\] 
such that the following diagram commutes 
\[
\begin{xy}
  \xymatrix{
      \Kar(\dihedralcat{n}) \ar[d]_{\GG{\cdot}}\ar[rr]^/-.1cm/{\functorADEt} & &   \Endf(\cellcatADE) \ar[d]^{\GG{\cdot}}  \\
      \mathrm{H}_n \ar[rr]_/-.1cm/{\G_n}             & &   \End(\cellrepgrg).  
  }
\end{xy}
\]
These are the only (non-trivial) $\g$'s and $n$'s for which this holds.\makeqed
\end{enumerate}
\end{theoremmain}
Here $\Kar$ denotes the Karoubi envelope, $\boldsymbol{\widetilde{\cdot}}$ 
the lift of a functor to the Karoubi envelope and $\Endf$ denotes 
the $2$-category of biprojective endofunctors, see \fullref{example:main-2-cats}. We will explain the meaning of  
${}^{\star}$ in the next section.

\begin{remark}\label{remark-gg-goes-wrong}
In $\ADE$ type, we need $\qpar$ in the Soergel 
calculus to be a complex, primitive root of unity. In that case -- as already remarked --  
$\dihedralcat{\infty}$ does not quite categorify the Hecke algebra 
$\mathrm{H}_{\infty}$, see \cite[Remarks 5.31 and 5.32]{El1}.
However, if one allows infinite $\ADE$ type graphs, then one can work with a generic parameter. 
The blueprint example is the infinite type $\typeA$ graph as considered in \cite{AT1}, for example.
\end{remark} 

For the following two theorems we switch to $\C$ for our ground field, and we keep the parameter $\qpar$ fixed.

\begin{theoremmain}(\textbf{Equivalences and isomorphisms.})\label{theorem:main-proposition}
\begin{enumerate}[label=(\alph*)]

\setlength\itemsep{.15cm}

\item All $2$-representations as in \fullref{theorem:main-theorem} 
are graded simple transitive 
$2$-representations of $\Kar(\dihedralcat{\infty})^{\star}$ respectively 
of $\Kar(\dihedralcat{n})^{\star}$. 

\item Two $2$-representations as in \fullref{theorem:main-theorem} are equivalent if and only if their bipartite graphs are isomorphic. 
 
\item Two $2$-representations as in \fullref{theorem:main-theorem} decategorify to isomorphic $\mathrm{H}$-modules if and only if their bipartite graphs are spectrum-color-equivalent.

\item All $2$-representations as in \fullref{theorem:main-theorem} factor through  
graded simple transitive $2$-representations of 
$\Kar(\dihedralcat{\infty}^{\mathrm{f}})^{\star}$ respectively 
of $\Kar(\dihedralcat{n}^{\mathrm{f}})^{\star}$.\makeqed
\end{enumerate}
\end{theoremmain}

Hereby ${}^{\mathrm{f}}$ means that we work over the 
coinvariant algebra.

We stress that 
(c) of \fullref{theorem:main-proposition} holds for the
analogs 
from \fullref{proposition:action-infty} 
and \fullref{proposition:action-finite} as well.

With these theorems we can complete the classification from 
\cite{KMMZ}, where rank means the rank on the level of the Grothendieck groups 
(for $n=1$ cf. \fullref{example:main-2-cats-2-reps}).

\begin{theoremmain}(\textbf{Classification.})\label{theorem:classification}
There is a bijection between the equivalence classes of graded simple transitive 
$2$-representations of $\dihedralcatbig{n}$ (of rank $>1$) and the isomorphism classes of bipartite graphs as 
in \ref{enum:typeA}, \ref{enum:typeD} or \ref{enum:typeE} with Coxeter number $n$, for $n\in\Z_{>1}$.
\end{theoremmain}

This completes the classification from \cite{KMMZ}, cf. \fullref{remark:different-qs}.

\begin{examplen}\label{example:cases-count}
For $n=8$, there are 
four $\g$'s of type \ref{enum:typeA}, \ref{enum:typeD} or \ref{enum:typeE} 
which are non-isomorphic as bipartite graphs:
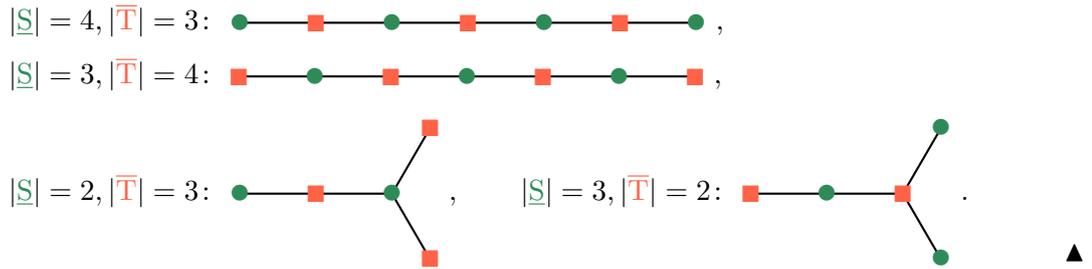
\begin{gather*}
\begin{aligned}
|\graphs|&=4, |\grapht|=3\colon
\begin{tikzpicture}[anchorbase, scale=1]
	\draw [thick] (0,0) to (6,0);
	\node at (0,-.01) {\Large $\bullets$};
	\node at (1,-.01) {\Large $\bullett$};
	\node at (2,-.01) {\Large $\bullets$};
	\node at (3,-.01) {\Large $\bullett$};
	\node at (4,-.01) {\Large $\bullets$};
	\node at (5,-.01) {\Large $\bullett$};
	\node at (6,-.01) {\Large $\bullets$};
\end{tikzpicture},
\\
|\graphs|&=3, |\grapht|=4\colon
\begin{tikzpicture}[anchorbase, scale=1]
	\draw [thick] (0,0) to (6,0);
	\node at (0,-.01) {\Large $\bullett$};
	\node at (1,-.01) {\Large $\bullets$};
	\node at (2,-.01) {\Large $\bullett$};
	\node at (3,-.01) {\Large $\bullets$};
	\node at (4,-.01) {\Large $\bullett$};
	\node at (5,-.01) {\Large $\bullets$};
	\node at (6,-.01) {\Large $\bullett$};
\end{tikzpicture},
\\
|\graphs|&=2, |\grapht|=3\colon
\begin{tikzpicture}[anchorbase, scale=1]
	\draw [thick] (3.8,0) to (5.8,0);
	\draw [thick] (6.3,-.866) to (5.8,0) to (6.3,.866);
	\node at (3.8,0) {\Large $\bullets$};
	\node at (4.8,0) {\Large $\bullett$};
	\node at (5.8,0) {\Large $\bullets$};
	\node at (6.3,.866) {\Large $\bullett$};
	\node at (6.3,-.866) {\Large $\bullett$};
\end{tikzpicture}
,\quad\quad
|\graphs|=3, |\grapht|=2\colon
\begin{tikzpicture}[anchorbase, scale=1]
	\draw [thick] (3.8,0) to (5.8,0);
	\draw [thick] (6.3,-.866) to (5.8,0) to (6.3,.866);
	\node at (3.8,0) {\Large $\bullett$};
	\node at (4.8,0) {\Large $\bullets$};
	\node at (5.8,0) {\Large $\bullett$};
	\node at (6.3,.866) {\Large $\bullets$};
	\node at (6.3,-.866) {\Large $\bullets$};
\end{tikzpicture}.
\hspace{1.25cm}
\raisebox{-.80cm}{\makeqedtri}
\hspace{-1.25cm}
\end{aligned}
\end{gather*}
\end{examplen}

In general, we get the following complete list of equivalence classes 
of higher rank graded simple transitive $2$-representations of 
$\dihedralcatbig{n}$, for $n\in\Z_{>1}$:
\begin{enumerate}[label=(\roman*)]

\setlength\itemsep{.15cm}

\item For odd $n$, there is just one, since the opposite coloring of 
a type $\typeA$ bipartite graph gives an isomorphic 
bipartite graph in this case.

\item For even $n\not\in\{2,4,12,18,30\}$, there are four, since the opposite colorings 
of $\typeA$ and $\typeD$ graphs give non-isomorphic 
bipartite graphs.

\item For $n\in\{2,4\}$, there are two, since the corresponding
type $\typeD$ graphs are of type $\typeA$ in these cases.

\item For $n\in\{12,18,30\}$, there are six, since there 
are two additional non-isomorphic bipartite graphs
of type $\typeE$ in each case.
\end{enumerate}

\begin{examplen}\label{example:cases-count-gg}
The spectrum is a full graph invariant for type $\ADE$ graphs (see e.g. \eqref{eq:spectrum}). 
Thus, in order to check if two inequivalent graded simple 
transitive $2$-re\-presentations of 
$\dihedralcatbig{n}$ decategorify to isomorphic $\mathrm{H}_n$-modules, we only need to compare their two-colorings. 
Outside types $\typee{6}$ and $\typee{8}$ nothing interesting 
happens. 
But the two two-colorings in \fullref{example-bigraph} give inequivalent bipartite graphs which are 
spectrum-color-equivalent. Therefore, the corresponding graded simple transitive $2$-re\-pre\-sen\-ta\-tions of $\dihedralcatbig{12}$ are inequivalent, but decategorify to isomorphic $\mathrm{H}_{12}$-modules 
(see also \fullref{example:SVD}). The same holds for the two type $\typee{8}$ graded simple transitive $2$-representations of $\dihedralcatbig{30}$.

In the infinite case the story is more delicate. 
As stated above, for all bipartite graphs, i.e. 
not necessarily of $\ADE$ type, we can construct graded simple 
transitive $2$-representations of 
$\dihedralcatbig{\infty}$, 
cf. \fullref{subsub:functor-bi-graphs}. 
By a classical result of Schwenk \cite{Sch1}, ``almost all'' trees are 
not determined by their 
spectrum -- in the sense that there are other, non-isomorphic trees with the same 
spectrum. 
However, a lot of them will be spectrum-color-equivalent. Thus, already for trees there are 
plenty of examples of inequivalent graded simple transitive $2$-representations of 
$\dihedralcatbig{\infty}$ which decategorify to isomorphic $\mathrm{H}_{\infty}$-modules, e.g.  
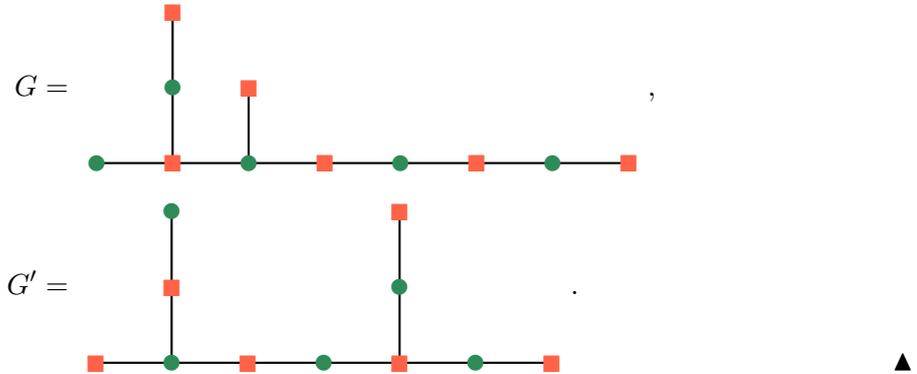
\begin{gather*}
\begin{aligned}
\g&=
\begin{tikzpicture}[anchorbase, scale=1]
	\draw [thick] (0,0) to (7,0);
	\draw [thick] (1,0) to (1,2);
	\draw [thick] (2,0) to (2,1);
	\node at (0,-.01) {\Large $\bullets$};
	\node at (1,-.01) {\Large $\bullett$};
	\node at (2,-.01) {\Large $\bullets$};
	\node at (3,-.01) {\Large $\bullett$};
	\node at (4,-.01) {\Large $\bullets$};
	\node at (5,-.01) {\Large $\bullett$};
	\node at (6,-.01) {\Large $\bullets$};
	\node at (7,-.01) {\Large $\bullett$};
	\node at (1,.99) {\Large $\bullets$};
	\node at (1,1.99) {\Large $\bullett$};
	\node at (2,.99) {\Large $\bullett$};
\end{tikzpicture},\\
\g^{\prime}&=
\begin{tikzpicture}[anchorbase, scale=1]
	\draw [thick] (0,0) to (6,0);
	\draw [thick] (1,0) to (1,2);
	\draw [thick] (4,0) to (4,2);
	\node at (0,-.01) {\Large $\bullett$};
	\node at (1,-.01) {\Large $\bullets$};
	\node at (2,-.01) {\Large $\bullett$};
	\node at (3,-.01) {\Large $\bullets$};
	\node at (4,-.01) {\Large $\bullett$};
	\node at (5,-.01) {\Large $\bullets$};
	\node at (6,-.01) {\Large $\bullett$};
	\node at (1,.99) {\Large $\bullett$};
	\node at (1,1.99) {\Large $\bullets$};
	\node at (4,.99) {\Large $\bullets$};
	\node at (4,1.99) {\Large $\bullett$};
\end{tikzpicture}.
\hspace{4.125cm}
\raisebox{-1.00cm}{\makeqedtri}
\hspace{-4.125cm}
\end{aligned}
\end{gather*}
\end{examplen}
\section{Graded \texorpdfstring{$2$}{2}-representations}\label{sec:2reps}

For us, Mazorchuk and Miemietz's setting 
(see e.g. \cite{MM1}, \cite{MM3} or \cite{MM5}) with finitary 
$2$-categories and strict $2$-representations is too restrictive. 

The two-color 
Soergel calculus is defined over the graded algebra of polynomials on the geometric representation of 
$\mathrm{W}$ (\textit{the polynomial algebra}, for short), which is finite-dimensional 
in each degree, but infinite-dimensional as a whole. In contrast, Mazorchuk--Miemietz 
always work over the so-called \textit{coinvariant algebra}, which is a finite-dimensional quotient of the polynomial algebra. 
This quotient inherits a grading, but they do not use it. 

Further, they consider strict $2$-representations. We will define our 
$2$-representations 
using the two-color Soergel calculus over the polynomial algebra and then prove that 
they descend to a quotient defined over the coinvariant algebra. This is 
done by prescribing the image of each $1$-morphism and each generating $2$-morphism, and checking the 
diagrammatic relations. On the level 
of $1$-morphisms these $2$-representations 
sometimes preserve composition only up to natural $2$-isomorphisms, i.e. 
our $2$-representations are given by weak $2$-functors 
(also called non-strict or pseudo). Fortunately, every weak $2$-representation 
in our sense can be strictified (see \fullref{remark:strictification}) and 
e.g. the classification results from \cite{KMMZ} remain true in our setup.

In the following abstract setup we work over a field $\someR$ for simplicity.

\subsection{Some basics about (graded finitary) \texorpdfstring{$2$}{2}-categories}\label{subsec:basics}

We use strict $2$-categories -- which we simply call $2$-categories -- and bicategories. We use strict and weak $2$-functors, which are carefully specified in each case.   

Let $\category{C}^\star$ be an additive, graded, $\someR$-linear $2$-category, i.e. a category enriched over 
the category of additive, graded, $\someR$-linear categories. 
We always assume that $\category{C}^\star$ has 
finitely many objects, up to equivalence, and that 
its $2$-morphism spaces are locally finite, i.e. finite-dimensional in each degree 
with the grading bounded from below. 

Moreover, in our setup, the $1$-morphisms admit grading shifts. 
Let $X\{\shiftme{a}\}$ denote a given $1$-morphism $X$ shifted $\shiftme{a}\in\Z$ 
degrees such that the identity $2$-morphism on $X$ gives rise to a homogeneous 
$2$-isomorphism $X\Rightarrow X\{\shiftme{a}\}$ of degree $\shiftme{a}$. In general, the $2$-morphisms in $\category{C}^\star$ 
are $\someR$-linear combinations of homogeneous ones, where a homogeneous  
$2$-morphism from $X$ to $Y$ of degree $d$ becomes 
homogeneous of degree $d+\shiftme{b}-\shiftme{a}$ when seen as a 
$2$-morphism from $X\{\shiftme{a}\}$ to $Y\{\shiftme{b}\}$.

There is a $2$-subcategory $\category{C}$ of $\category{C}^\star$ which has the same objects and $1$-morphisms, but only contains the degree-preserving $2$-morphisms. 
It still has the degree shifts on $1$-morphisms, but the $2$-morphism spaces are 
no longer graded. In $\category{C}$ the $1$-morphisms $X\{\shiftme{a}\}$ 
and $X\{\shiftme{b}\}$ are in general only isomorphic for $\shiftme{a}=\shiftme{b}$. 

One can recover $\category{C}^\star$ from $\category{C}$, because 
(for any $1$-morphisms $X,Y$ in $\category{C}^\star$ or $\category{C}$):
\[
\twoHom_{\category{C}^\star}(X,Y)
=
{\textstyle\bigoplus_{\shiftme{a}\in\Z}}\,
\twoHom_{\category{C}}(X\{\shiftme{a}\},Y).
\]

We also assume that, for each pair of 
objects $x,y$, the hom-category 
$\Hom_{\category{C}}(x,y)$ is idempotent complete and Krull-Schmidt. (In the diagrammatic Soergel $2$-categories, we therefore have to take the 
Karoubi envelope of each hom-category.) We also assume that the identity $1$-morphisms are indecomposable. 

If the split Grothendieck group of $\category{C}$ has finite 
rank over $\Z[\vpar,\vpar^{-1}]$, where $\vpar$ corresponds to the degree shift $\{\shiftme{1}\}$, we say that 
$\category{C}^\star$ is {\em graded finitary}. If, additionally, 
the $2$-morphism spaces are finite-dimensional, we say that 
$\category{C}^\star$ is {\em graded $2$-finitary}. 

If the split Grothendieck group of $\category{C}$ has countably 
infinite $\Z[\vpar,\vpar^{-1}]$-rank, we say that $\category{C}^\star$ is {\em graded locally finitary}. 
If, additionally, the $2$-morphism spaces are finite-dimensional, we say that 
$\category{C}^\star$ is {\em graded locally $2$-finitary}. 

We also use {\em graded $1$-finitary} categories, whose $1$-morphism spaces 
are graded and finite-dimensional.

\begin{example}\label{example:main-2-cats-first}
The Soergel bimodules of any Coxeter type, when really defined using bimodules over the polynomial algebra $\polyalg$, form a bicategory. This is easy to see, e.g. $\polyalg\otimes_{\polyalg} \polyalg$ is only isomorphic to $\polyalg$ and not equal to it. Moreover, it is not a small $2$-category, e.g. the isomorphism class of the polynomial algebra $\polyalg$ in this $2$-category is not a set. Luckily, any bicategory is weakly equivalent to a $2$-category -- as follows from a strictification theorem due to Mac Lane (see e.g. \cite[Section XI.3]{ML1}) for monoidal categories and to B{\'e}nabou \cite{Ben1} for bicategories. (Alternatively, see \cite[Theorem 2.3]{Lei1}.) And in our case, the $2$-categories are small. 

However, for our purposes we need a concrete version of such $2$-categories. 
Thus, we use the Karoubi envelope of the two-color Soergel calculi $\dihedralcat{n}$ and $\dihedralcat{\infty}$, respectively, which we recall in 
\fullref{subsec:diacat}. As we will see 
in \fullref{proposition:locally-finitary}, $\Kar(\dihedralcat{n})^\star$ and 
$\Kar(\dihedralcat{\infty})^\star$ are graded finitary and graded locally finitary $2$-categories.

The two-color Soergel calculi are defined over the polynomial algebra, but admit quotients to $2$-categories defined over the coinvariant algebra, denoted by $\dihedralcat{n}^{\mathrm{f}}$ and $\dihedralcat{\infty}^{\mathrm{f}}$. 
We explain this more carefully later -- in 
the proof of part (iv) of \fullref{theorem:main-proposition}. By \fullref{proposition:locally-finitary} and by construction, $\Kar(\dihedralcat{n}^{\mathrm{f}})^\star$ and $\Kar(\dihedralcat{\infty}^{\mathrm{f}})^\star$ are graded $2$-finitary and graded 
locally $2$-finitary $2$-categories, respectively.
\end{example}

\begin{examplen}\label{example:main-2-cats}
Let $\somealg$ be a graded, finite-dimensional algebra. Then its category of graded 
projective, finite-dimensional (left) $\somealg$-modules 
$\Modl{\somealg}$ is an additive, graded 
$1$-finitary, $\someR$-linear category 
which is idempotent complete and Krull-Schmidt. We view it as being 
small by taking equivalence classes of $\somealg$-modules.

We call a finite-dimensional $\somealg$-bimodule $X$
\textit{biprojective} if it is projective as a left and as a right $\somealg$-module (but not necessarily as an $\somealg$-bimodule). Given a graded biprojective, finite-dimensional $\somealg$-bimodule $X$. Then $X\otimes_{\somealg}-$ gives rise to an exact endofunctor of $\Modl{\somealg}$. An endofunctor of $\Modl{\somealg}$ is called {\em biprojective}, if it is isomorphic to a direct summand of a finite direct sum
of endofunctors of the form $X\otimes_{\somealg}-$, where $X$ is a biprojective $\somealg$-bimodule.

Let $\Endf(\Modl{\somealg})$ be the $2$-category with the unique object $\Modl{\somealg}$, whose $1$-morphisms are biprojective endofunctors on $\Modl{\somealg}$, and 
whose $2$-mor\-phisms are degree-preserving natural transformations between these. 
Hereby recall that the Godement product induces the horizontal composition $\hcomp$ via:
\begin{gather}\label{eq:Godement}
\begin{aligned}
\Functor{\mathcal{F}},\Functor{\mathcal{G}}\colon X\to Y, \quad\quad 
\Functor{\mathcal{H}},\Functor{\mathcal{I}}\colon Y\to Z, \quad\quad 
\nattrafo{f}\colon\Functor{\mathcal{F}}\Rightarrow\Functor{\mathcal{G}},\quad\quad
\nattrafo{g}\colon\Functor{\mathcal{H}}\Rightarrow\Functor{\mathcal{I}},\\
\nattrafo{g}\hcomp\nattrafo{f}\colon
\Functor{\mathcal{H}}\Functor{\mathcal{F}}
\Rightarrow\Functor{\mathcal{I}}\Functor{\mathcal{G}},
\quad
(\nattrafo{g}\hcomp\nattrafo{f})_X=
\nattrafo{g}_{\mathcal{G}(X)}\circ\Functor{\mathcal{H}}(\nattrafo{f}_X)
=\Functor{\mathcal{I}}(\nattrafo{f}_X)\circ\nattrafo{g}_{\mathcal{F}(X)}.
\hspace{1.975cm}
\raisebox{-.1cm}{\makeqedtri}
\hspace{-1.975cm}
\end{aligned}
\end{gather}
\end{examplen}

\begin{remark}\label{remark:strict}
All diagrammatic $2$-categories which appear in this paper, e.g. 
$\dihedralcat{n}$ as in \fullref{subsec:diacat}, are, by definition, strict.
In contrast, $\Endf(\Modl{\somealg})$, introduced in \fullref{example:main-2-cats}, is a bicategory. 
Still, we always view it as being strict by using the B{\'e}nabou--Mac Lane coherence theorem.
\end{remark}
\subsection{Graded simple transitive \texorpdfstring{$2$}{2}-representations}\label{subsec:basic-2cat-reps}

Our next goal is to  
define a graded and weak version of 
certain $2$-representations due 
to Mazorchuk--Miemietz, see e.g. \cite{MM3} or \cite{MM5} where more details can be found.

\subsubsection{\texorpdfstring{$2$}{2}-representations: definitions}\label{subsub:graded-2reps}

Let $\boldsymbol{\mathfrak{A}}^f_{\mathrm{gr}}$ be the $2$-category 
whose objects are additive, graded $1$-finitary, 
$\someR$-linear (small) categories; $1$-morphisms are additive, 
degree-preserving, $\someR$-linear functors; $2$-morphisms 
are homogeneous degree-zero natural transformations. 
For example, $\Modl{\algG}$ is an object of $\boldsymbol{\mathfrak{A}}^f_{\mathrm{gr}}$.

\begin{definitionn}\label{definition:MM2reps}
Let $\category{C}^\star$ be a graded (locally) ($2$-)finitary $2$-category. Then a {\em graded $2$-finitary, weak $2$-representation} of $\category{C}^\star$ is an additive, $\someR$-linear, weak  $2$-functor 
\[
\mathbf{M}\colon \category{C}^\star\to \left(\boldsymbol{\mathfrak{A}}^f_{\mathrm{gr}}\right)^\star
\]
which preserves degrees and commutes with shifts as follows. For any indecomposable $1$-morphism $F$ in $\category{C}^\star$ and any indecomposable object $X\in \coprod_{x}\mathbf{M}(x)$ 
(note that objects in $\coprod_{x}\mathbf{M}(x)$ are $1$-morphisms 
in the target $2$-category) we have 
\[
\mathbf{M}(F\{\shiftme{a}\})(X\{\shiftme{b}\})=\mathbf{M}(F)(X)\{\shiftme{a+b}\},\qquad \shiftme{a},\shiftme{b}\in\Z.
\hspace{3.25cm}
\raisebox{-.01cm}{\makeqedtri}
\hspace{-3.25cm}
\]
\end{definitionn}

These form a bicategory, whose $1$-morphisms are weak natural transformations, with degree-zero structural $2$-isomorphisms across the squares in their definition, and whose $2$-morphisms are degree-zero modifications. 

\begin{remark}\label{remark:weak}
We require a graded $2$-finitary, weak $2$-representation to be a weak  $2$-functor 
which preserves degrees. 
Such a $2$-functor restricts to, and is uniquely determined by, an additive, $\someR$-linear, weak  $2$-functor $\mathbf{M}\colon \category{C}\to \boldsymbol{\mathfrak{A}}^f_{\mathrm{gr}}$. We will use both weak  $2$-functors almost interchangeably.  
\end{remark}

Mazorchuk--Miemietz \cite{MM5} defined simple 
transitive, strict $2$-representations, which are a 
categorical analogue of simple representations of finite-dimensional algebras. Their 
definition remains (almost) unchanged in the our setting. 

\begin{definition}\label{definition:simple-transitive-2reps} 
We say that a graded $2$-finitary, weak $2$-representation $\mathbf{M}$ of a 
graded (locally) ($2\text{-})$finitary $2$-category $\category{C}^\star$ is 
{\em transitive} if for any two indecomposable objects 
$X,Y$ in $\coprod_{x}\mathbf{M}(x)$ there exists a 
$1$-morphism $F$ in $\category{C}^\star$ such that $Y$ is isomorphic to a graded direct summand of $\mathbf{M}(F)(X)$. 
It is called {\em graded simple transitive} if, 
additionally, $\coprod_{x}\mathbf{M}(x)$ has no non-zero 
proper $\category{C}^\star$-invariant ideals.
\end{definition}

(We say ``graded simple transitive'' and omit the 
``weak'', cf. \fullref{remark:strictification}.)

\begin{remark}\label{remark:weak2}
Suppose that $Y\{\shiftme{a}\}$ is isomorphic to a direct summand of $\mathbf{M}(F)(X)$, for some $\shiftme{a}\in\Z$. 
Then $Y$ is isomorphic to a direct summand of 
$\mathbf{M}(F\{\shiftme{-a}\})(X)$. 
\end{remark}

\begin{example}\label{example:main-2-cats-2-reps}
Let $\onecolor$ denote the one-color (say {\color{sea}sea-green} $\dicatsgen$)
Soergel calculus of Coxeter type $\typea{1}$. 
This is defined as $\dihedralcat{n}$ in \fullref{subsec:diacat}, 
but dropping the second color (say {\color{tomato}tomato} $\dicattgen$) and the relations in which it is involved. 
This $2$-category has one object (which we do not specify here). 
The $1$-morphisms are formal direct sums of finite words (``tensor products'') of shifts of $\dicatsgen$, 
and $2$-morphisms are degree zero Soergel diagrams. The $1$-morphisms 
might not be indecomposable, e.g. we have  $\dicatsgen\dicatsgen\cong\dicatsgen\{\shiftme{-1}\}\oplus\dicatsgen\{\shiftme{+1}\}$. 
In particular, $\onecolor$ is an additive, $\someR$-linear $2$-category and $\onecolor\cong\Kar(\onecolor)$ is 
idempotent complete and Krull-Schmidt. Finally, 
$\Kar(\onecolor)^\star$ is a graded finitary 
$2$-category. 

Now, consider the coinvariant algebra $\coalg^+_{\typea{1}}=\dualalg$ of the 
Weyl group of type $\typea{1}$, the so-called
dual numbers 
$\dualalg\cong\someR[X]/(X^2)$ -- with $X$ of degree two. Then the $2$-category 
$\onecolor^{\mathrm{f}}$ which is defined over the coinvariant algebra, is a quotient of $\onecolor$. The ${}^\star$ 
of its Karoubi envelope $\Kar(\onecolor^\mathrm{f})^\star$ is a graded $2$-finitary $2$-category.  

We can define a graded $2$-finitary, weak $2$-representation of $\Kar(\onecolor)^{\star}$ 
on the category $(\Modl{\dualalg})^\star$, by sending 
the empty word $\emptyset$ to the endofunctor $\dualalg\otimes_{\dualalg} -$ and 
$\dicatsgen$ to the endofunctor $\dualalg\otimes \dualalg\otimes_{\dualalg} -$. Note 
that $\dualalg$ and $\dualalg\otimes \dualalg$ are graded biprojective, 
$\dualalg$-bimodules. This $2$-representation 
is simple transitive, see \cite[Section 3.4]{MM}, and also graded.
By construction, it descends to $\Kar(\onecolor^{\mathrm{f}})^\star$.

In the setup of bipartite graphs, this $2$-representation is given by a  
graph with one {\color{sea}sea-green s} colored vertex $\bullets$. One can check that there is no other 
graded simple transitive $2$-re\-presentation for the corresponding small quotient.  
\end{example}

\subsubsection{\texorpdfstring{$2$}{2}-representations: strictifications}\label{subsub:strict-stuff}

Let $\category{C}$ be a small $2$-category, $\boldsymbol{\mathrm{Cat}}$ the 
$2$-ca\-te\-gory of small categories and 
$\functor{F}_{\mathrm{w}}\colon\category{C}\to\boldsymbol{\mathrm{Cat}}$ any 
weak $2$-functor. 
By \cite[Section 4.2]{Po1} this $2$-fun\-ctor can be strictified: There exists a 
strict $2$-functor $\functor{F}_{\mathrm{s}}\colon\category{C}\to\boldsymbol{\mathrm{Cat}}$ 
and a weak natural $2$-iso\-mor\-phism $\nattrafo{f}\colon\functor{F}_{\mathrm{w}}\Rightarrow\functor{F}_{\mathrm{s}}$. Moreover, given two such weak 
$2$-functors $\functor{F}_{\mathrm{w}}$ and $\functor{G}_{\mathrm{w}}$, the $2$-categories 
of weak -- respectively strict -- 
natural transformations and modifications between $\functor{F}_{\mathrm{w}}$ 
and $\functor{G}_{\mathrm{w}}$, respectively between $\functor{F}_{\mathrm{s}}$ 
and $\functor{G}_{\mathrm{s}}$, are equivalent by conjugation with $\nattrafo{f}$ and $\nattrafo{g}$. 

A close inspection of Power's arguments shows that the same holds when $\boldsymbol{\mathrm{Cat}}$ is replaced by $\boldsymbol{\mathfrak{A}}^f_{\mathrm{gr}}$, as long as $\category{C}$ is graded (locally) ($2$-)finitary with finitely many objects. 

(We thank Nick Gurski for explaining to us the relevance of $2$-monads for strictification and giving us some pointers to the literature.)

\begin{remark}\label{remark:strictification}
Note that the two-color Soergel calculi (and their Karoubi envelopes), 
in the finite and the infinite case, are $2$-categories with finitely many objects.
Thus, by the above, the (graded simple transitive) weak $2$-representations 
in this paper are weakly equivalent to (graded simple transitive) strict $2$-representations.
\end{remark}
\section{The two-color Soergel calculus and its \texorpdfstring{$2$}{2}-action}\label{sec:dicat}
Next, we construct the $2$-representations from \fullref{theorem:main-theorem}. 
\subsection{Soergel calculus in two colors}\label{subsec:diacat}

Fix $n\in\Z_{>0}$ or ``$n=\infty$''.
We recall Elias' two-color Soergel calculi  
$\dihedralcatbig{n}$ and  
$\dihedralcatbig{\infty}$, as defined in \cite[Sections 5 and 6]{El1}. 
We write $\dihedralcatbig{}$ if no confusion can arise, 
and also use ``unstarred'' versions $\dihedralcat{}$ of these.

\begin{remark}\label{remark:ground-ring}
Let us now fix the technical conditions defining our ground rings.

For weak categorifications, 
as in \fullref{subsub:di-modules-bipartite}, we can work 
over $\Zq=\Z[\qpar,\qpar^{-1}]$ for a fixed (generic) $\qpar\in\C-\{0\}$. 

For strong categorifications, 
as in \fullref{subsec:main-theorems}, we have to adjoin 
the inverse of some quantum integers as well. For $\g$ of $\ADE$ type with Coxeter 
number $n>2$ (cf. \ref{enum:typeA}, \ref{enum:typeD} or \ref{enum:typeE}), let 
$\qpar$ be a complex, primitive $2n$th root of unity and let
\begin{gather}\label{eq:aform}
\ZQ=
\Z[\scalebox{.9}{$\nicefrac{$1$}{$2$}$},\qpar,\qpar^{-1}, 
\scalebox{.9}{$\nicefrac{$1$}{$[2]_\qpar$}$},\dots,
\scalebox{.9}{$\nicefrac{$1$}{$[n{-}1]_\qpar$}$}].
\end{gather}
For $n=1$ and $n=2$ we let $\ZQ=\Z[\scalebox{.9}{$\nicefrac{$1$}{$2$}$}]$ 
and $\Z[\scalebox{.9}{$\nicefrac{$1$}{$2$}$},\pm \sqrt{-1}]$, respectively. 
Hereby recall that the quantum integers are defined by
\[
[z]_{\qpar}=\qpar^{z-1}+\qpar^{z-3}+\dots+\qpar^{3-z}+\qpar^{1-z},\; 
z\in\Z_{>0},\quad [0]_{\qpar}=0.
\]

For more general bipartite graphs we use $\zgl$ as in 
\fullref{subsub:functor-bi-graphs}, which is obtained from $\Z$ 
by adjoining a finite number of complex numbers and their inverses. 
To be precise, we adjoin the \eqref{eq:barb2}-weightings as in \fullref{subsub:functor-infty-2}.
\end{remark}

The definition of those calculi depends on a fixed parameter $\qpar\in\C-\{0\}$, 
which is as in \fullref{subsec:quiver-stuff} and \fullref{remark:ground-ring}. 
Our ground ring is again the appropriate $\zq$.

\subsubsection{Soergel diagrams}\label{subsub:SDs}

We consider the following generating $2$-morphisms.
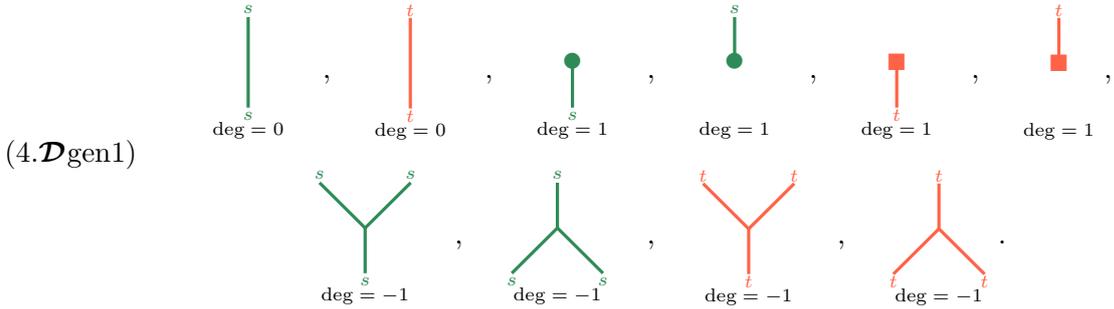
\renewcommand{\theequation}{\arabic{section}.$\dihedralcat{}$gen1}
\begin{gather}\label{eq:dicat-gen1}
\begin{aligned}
\begin{tikzpicture}[anchorbase, scale=.3]
	\draw [very thick, sea] (0,-2) to (0,2);
	\node at (0,-2.35) {\tiny $\dicatsgen$};
	\node at (0,2.35) {\tiny $\dicatsgen$};
	\node at (0,-3) {\tiny $\text{deg}=0$};
\end{tikzpicture}
\quad,\quad
\begin{tikzpicture}[anchorbase, scale=.3]
	\draw [very thick, tomato] (0,-2) to (0,2);
	\node at (0,-2.35) {\tiny $\dicattgen$};
	\node at (0,2.35) {\tiny $\dicattgen$};
	\node at (0,-3) {\tiny $\text{deg}=0$};
\end{tikzpicture}
\quad,\quad
\begin{tikzpicture}[anchorbase, scale=.3]
	\draw [very thick, sea] (0,-2) to (0,0);
	\node at (0,-2.35) {\tiny $\dicatsgen$};
	\node at (0,2.35) {\tiny $\phantom{a}$};
	\node at (0,0) {\Large $\bullets$};
	\node at (0,-3) {\tiny $\text{deg}=1$};
\end{tikzpicture}
&\quad,\quad
\begin{tikzpicture}[anchorbase, scale=.3]
	\draw [very thick, sea] (0,2) to (0,0);
	\node at (0,2.35) {\tiny $\dicatsgen$};
	\node at (0,-2.35) {\tiny $\phantom{a}$};
	\node at (0,0) {\Large $\bullets$};
	\node at (0,-3) {\tiny $\text{deg}=1$};
\end{tikzpicture}
\quad,\quad
\begin{tikzpicture}[anchorbase, scale=.3]
	\draw [very thick, tomato] (0,-2) to (0,0);
	\node at (0,-2.35) {\tiny $\dicattgen$};
	\node at (0,2.35) {\tiny $\phantom{a}$};
	\node at (0,0) {\Large $\bullett$};
	\node at (0,-3) {\tiny $\text{deg}=1$};
\end{tikzpicture}
\quad,\quad
\begin{tikzpicture}[anchorbase, scale=.3]
	\draw [very thick, tomato] (0,2) to (0,0);
	\node at (0,2.35) {\tiny $\dicattgen$};
	\node at (0,-2.35) {\tiny $\phantom{a}$};
	\node at (0,0) {\Large $\bullett$};
	\node at (0,-3) {\tiny $\text{deg}=1$};
\end{tikzpicture},
\\
\begin{tikzpicture}[anchorbase, scale=.3]
	\draw [very thick, sea] (0,-2) to (0,0) to (-2,2);
	\draw [very thick, sea] (0,0) to (2,2);
	\node at (0,-2.35) {\tiny $\dicatsgen$};
	\node at (-2,2.35) {\tiny $\dicatsgen$};
	\node at (2,2.35) {\tiny $\dicatsgen$};
	\node at (0,-3) {\tiny $\text{deg}=-1$};
\end{tikzpicture}
\quad,\quad
\begin{tikzpicture}[anchorbase, scale=.3]
	\draw [very thick, sea] (0,2) to (0,0) to (-2,-2);
	\draw [very thick, sea] (0,0) to (2,-2);
	\node at (0,2.35) {\tiny $\dicatsgen$};
	\node at (-2,-2.35) {\tiny $\dicatsgen$};
	\node at (2,-2.35) {\tiny $\dicatsgen$};
	\node at (0,-3) {\tiny $\text{deg}=-1$};
\end{tikzpicture}
&\quad,\quad
\begin{tikzpicture}[anchorbase, scale=.3]
	\draw [very thick, tomato] (0,-2) to (0,0) to (-2,2);
	\draw [very thick, tomato] (0,0) to (2,2);
	\node at (0,-2.35) {\tiny $\dicattgen$};
	\node at (-2,2.35) {\tiny $\dicattgen$};
	\node at (2,2.35) {\tiny $\dicattgen$};
	\node at (0,-3) {\tiny $\text{deg}=-1$};
\end{tikzpicture}
\quad,\quad
\begin{tikzpicture}[anchorbase, scale=.3]
	\draw [very thick, tomato] (0,2) to (0,0) to (-2,-2);
	\draw [very thick, tomato] (0,0) to (2,-2);
	\node at (0,2.35) {\tiny $\dicattgen$};
	\node at (-2,-2.35) {\tiny $\dicattgen$};
	\node at (2,-2.35) {\tiny $\dicattgen$};
	\node at (0,-3) {\tiny $\text{deg}=-1$};
\end{tikzpicture}.
\end{aligned}
\end{gather}
\renewcommand{\theequation}{\arabic{section}.\arabic{equation}}
(Recall that the objects $\dicatsgen$ and $\dicattgen$ correspond to $\disgen$ and $\ditgen$ of $\mathrm{H}$, and not to the Coxeter generators 
$\dihsgen$ and $\dihtgen$ of $\mathrm{W}$, see e.g. \cite[Theorems 5.29 and 6.24]{El1}.)

We give these $2$-morphisms the indicated degrees, and call them \textit{identities}, 
(\textit{end} and \textit{start}) \textit{dots}, and \textit{trivalent vertices} 
(\textit{split} and \textit{merge}) respectively. 

All generators in \eqref{eq:dicat-gen1} are independent of $n$. The following degree-zero generator depends on $n$ and is called the \textit{$2n$-vertex}:
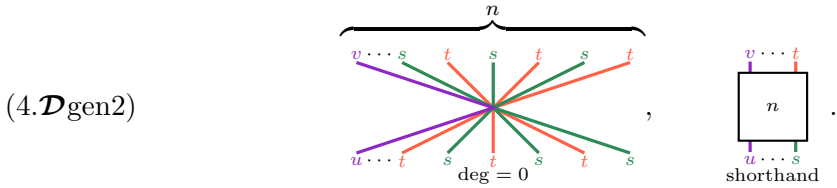
\renewcommand{\theequation}{\arabic{section}.$\dihedralcat{}$gen2}
\begin{gather}\label{eq:dicat-gen2}
\raisebox{-.1cm}{$\overbrace{\begin{tikzpicture}[anchorbase, scale=.3]
	\draw [very thick, tomato] (2,2) to (0,0) to (-2,2);
	\draw [very thick, tomato] (6,2) to (0,0);
	\draw [very thick, tomato] (0,-2) to (0,0);
	\draw [very thick, tomato] (4,-2) to (0,0) to (-4,-2);
	\draw [very thick, sea] (2,-2) to (0,0) to (-2,-2);
	\draw [very thick, sea] (6,-2) to (0,0);
	\draw [very thick, sea] (0,2) to (0,0);
	\draw [very thick, sea] (4,2) to (0,0) to (-4,2);
	\draw [very thick, orchid] (-6,2) to (0,0) to (-6,-2);
	\node at (-6,2.35) {\tiny $\dicattsgen$};
	\node at (-4,2.35) {\tiny $\dicatsgen$};
	\node at (-2,2.35) {\tiny $\dicattgen$};
	\node at (0,2.35) {\tiny $\dicatsgen$};
	\node at (2,2.35) {\tiny $\dicattgen$};
	\node at (4,2.35) {\tiny $\dicatsgen$};
	\node at (6,2.35) {\tiny $\dicattgen$};
	\node at (-6,-2.35) {\tiny $\dicatstgen$};
	\node at (-4,-2.35) {\tiny $\dicattgen$};
	\node at (-2,-2.35) {\tiny $\dicatsgen$};
	\node at (0,-2.35) {\tiny $\dicattgen$};
	\node at (2,-2.35) {\tiny $\dicatsgen$};
	\node at (4,-2.35) {\tiny $\dicattgen$};
	\node at (6,-2.35) {\tiny $\dicatsgen$};
	\node at (-5,-2.35) {\tiny $\,\cdots$};
	\node at (-5,2.35) {\tiny $\,\cdots$};
	\node at (0,-3) {\tiny $\text{deg}=0$};
\end{tikzpicture}}^n$}
,\quad\quad
\raisebox{-.075cm}{$\begin{tikzpicture}[anchorbase, scale=.3]
	\draw [very thick, orchid] (-1,-2) to (-1,-1.5);
	\draw [very thick, sea] (1,-2) to (1,-1.5);
	\draw [very thick, orchid] (-1,2) to (-1,1.5);
	\draw [very thick, tomato] (1,2) to (1,1.5);
	\draw [thick] (-1.5,-1.5) rectangle (1.5,1.5);
	\node at (-1,2.35) {\tiny $\dicattsgen$};
	\node at (1,2.35) {\tiny $\dicattgen$};
	\node at (-1,-2.35) {\tiny $\dicatstgen$};
	\node at (1,-2.35) {\tiny $\dicatsgen$};
	\node at (0,-2.35) {\tiny $\,\cdots$};
	\node at (0,2.35) {\tiny $\,\cdots$};
	\node at (0,0) {\tiny $n$};
	\node at (0,-3) {\tiny $\text{shorthand}$};
\end{tikzpicture}$}.
\end{gather}
\renewcommand{\theequation}{\arabic{section}.\arabic{equation}}
Here $\dicatstgen$ is either $\dicatsgen$ or $\dicattgen$, 
and $\dicattsgen$ is the opposite, depending on $n$.
The color inverted ($\dicatsgen\rightleftarrows\dicattgen$) version of the $2n$-vertex 
in \eqref{eq:dicat-gen2} is our last generator. Throughout 
the paper, we will use the shorthand for the $2n$-vertices as in \eqref{eq:dicat-gen2}.

The vertical composition $g\vcomp f$ is given by gluing 
diagram $g$ on top of diagram $f$ (in case the colors match), the horizontal composition 
$g\hcomp f$ by putting $g$ to the left of $f$. Our conventions 
are best illustrated in an example.

\begin{example}\label{example:reading-conventions}
For instance, in case $n=3$:
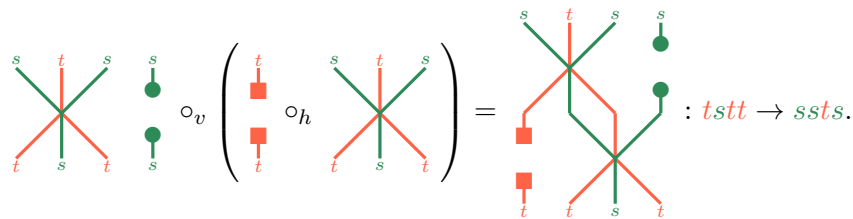
\[
\begin{tikzpicture}[anchorbase, scale=.3]
	\draw [very thick, tomato] (2,-2) to (0,0) to (-2,-2);
	\draw [very thick, tomato] (0,2) to (0,0);
	\draw [very thick, sea] (2,2) to (0,0) to (-2,2);
	\draw [very thick, sea] (0,-2) to (0,0);
	\node at (-2,2.35) {\tiny $\dicatsgen$};
	\node at (0,2.35) {\tiny $\dicattgen$};
	\node at (2,2.35) {\tiny $\dicatsgen$};
	\node at (-2,-2.35) {\tiny $\dicattgen$};
	\node at (0,-2.35) {\tiny $\dicatsgen$};
	\node at (2,-2.35) {\tiny $\dicattgen$};
	\draw [very thick, sea] (4,2) to (4,1);
	\draw [very thick, sea] (4,-2) to (4,-1);
	\node at (4,-1) {\Large $\bullets$};
	\node at (4,1) {\Large $\bullets$};
	\node at (4,2.35) {\tiny $\dicatsgen$};
	\node at (4,-2.35) {\tiny $\dicatsgen$};
\end{tikzpicture}
\vcomp
\left(
\begin{tikzpicture}[anchorbase, scale=.3]
	\draw [very thick, tomato] (0,2) to (0,1);
	\draw [very thick, tomato] (0,-2) to (0,-1);
	\node at (0,-1) {\Large $\bullett$};
	\node at (0,1) {\Large $\bullett$};
	\node at (0,2.35) {\tiny $\dicattgen$};
	\node at (0,-2.35) {\tiny $\dicattgen$};
\end{tikzpicture}
\hcomp
\begin{tikzpicture}[anchorbase, scale=.3]
	\draw [very thick, tomato] (2,-2) to (0,0) to (-2,-2);
	\draw [very thick, tomato] (0,2) to (0,0);
	\draw [very thick, sea] (2,2) to (0,0) to (-2,2);
	\draw [very thick, sea] (0,-2) to (0,0);
	\node at (-2,2.35) {\tiny $\dicatsgen$};
	\node at (0,2.35) {\tiny $\dicattgen$};
	\node at (2,2.35) {\tiny $\dicatsgen$};
	\node at (-2,-2.35) {\tiny $\dicattgen$};
	\node at (0,-2.35) {\tiny $\dicatsgen$};
	\node at (2,-2.35) {\tiny $\dicattgen$};
\end{tikzpicture}
\right)
=
\begin{tikzpicture}[anchorbase, scale=.3]
	\draw [very thick, tomato] (2,-2) to (0,0) to (-2,-2);
	\draw [very thick, tomato] (0,2) to (0,0);
	\draw [very thick, sea] (2,2) to (0,0) to (-2,2);
	\draw [very thick, sea] (0,-2) to (0,0);
	\node at (-2,2.35) {\tiny $\dicatsgen$};
	\node at (0,2.35) {\tiny $\dicattgen$};
	\node at (2,2.35) {\tiny $\dicatsgen$};
	\draw [very thick, sea] (4,2) to (4,1);
	\draw [very thick, sea] (4,-2) to (4,-1);
	\node at (4,-1) {\Large $\bullets$};
	\node at (4,1) {\Large $\bullets$};
	\node at (4,2.35) {\tiny $\dicatsgen$};
	\draw [very thick, tomato] (-2,-2) to (-2,-3);
	\draw [very thick, tomato] (-2,-6) to (-2,-5);
	\node at (-2,-5) {\Large $\bullett$};
	\node at (-2,-3) {\Large $\bullett$};
	\node at (-2,-6.35) {\tiny $\dicattgen$};
	\draw [very thick, tomato] (4,-6) to (2,-4) to (0,-6);
	\draw [very thick, tomato] (2,-2) to (2,-4);
	\draw [very thick, sea] (4,-2) to (2,-4) to (0,-2);
	\draw [very thick, sea] (2,-6) to (2,-4);
	\node at (0,-6.35) {\tiny $\dicattgen$};
	\node at (2,-6.35) {\tiny $\dicatsgen$};
	\node at (4,-6.35) {\tiny $\dicattgen$};
\end{tikzpicture}
\colon
\dicattgen\dicatsgen\dicattgen\dicattgen
\to\dicatsgen\dicatsgen\dicattgen\dicatsgen.
\]
Here we have also indicated our reading conventions for Soergel diagrams.
\end{example}

\begin{definition}\label{definition:soergel-dia}
Given two words 
$\word,\word^{\prime}$ in the symbols $\dicatsgen$ and $\dicattgen$, 
a Soergel diagram from $\word$ to $\word^{\prime}$ is 
a diagram $\vcomp$-$\hcomp$-generated 
by the generators from \eqref{eq:dicat-gen1} and \eqref{eq:dicat-gen2} 
such that the outgoing edges correspond 
color-wise to the entries of $\word$, $\word^{\prime}$. 

The degree of a Soergel diagram is, by definition, the sum of the degrees 
of its generators. (The Soergel diagram $\varnothing\colon\emptyset\to\emptyset$ is, by convention, of degree zero.)
\end{definition}

The top and bottom of a Soergel diagram correspond to direct sums 
(we do not assume that words in $\dicatsgen$ and $\dicattgen$ 
are reduced) of 
so-called \textit{Bott--Samelson bimodules}, which categorify Bott--Samelson basis elements $\theta_{\overline{\word}}$, for $\word\in\mathrm{W}$, cf. \fullref{subsub:KL-stuff}. In general, the Bott--Samelson bimodules do not coincide with the indecomposable Soergel bimodules, which categorify the Kazhdan--Lusztig basis elements $\klbasis_{\word}$.

\begin{example}\label{example:pitchfork}
The following shorthand notations and their color inversion ($\dicatsgen\rightleftarrows\dicattgen$)
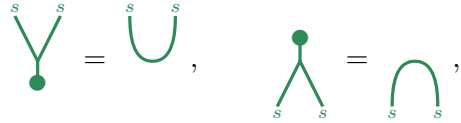
\[
\begin{tikzpicture}[anchorbase, scale=.3]
	\draw [very thick, sea] (0,2) to (1,0) to (2,2);
	\draw [very thick, sea] (1,0) to (1,-1);
	\node at (0,2.35) {\tiny $\dicatsgen$};
	\node at (2,2.35) {\tiny $\dicatsgen$};
	\node at (0,-2.35) {\tiny $\phantom{a}$};
	\node at (1,-1) {\Large $\bullets$};
\end{tikzpicture}
=
\begin{tikzpicture}[anchorbase, scale=.3]
	\draw [very thick, sea] (0,2) to [out=270, in=180] (1,0) to [out=0, in=270] (2,2);
	\node at (0,2.35) {\tiny $\dicatsgen$};
	\node at (2,2.35) {\tiny $\dicatsgen$};
	\node at (0,-2.35) {\tiny $\phantom{a}$};
\end{tikzpicture}
,\quad\quad
\begin{tikzpicture}[anchorbase, scale=.3]
	\draw [very thick, sea] (0,-2) to (1,0) to (2,-2);
	\draw [very thick, sea] (1,0) to (1,1);
	\node at (0,-2.35) {\tiny $\dicatsgen$};
	\node at (2,-2.35) {\tiny $\dicatsgen$};
	\node at (0,2.35) {\tiny $\phantom{a}$};
	\node at (1,1) {\Large $\bullets$};
\end{tikzpicture}
=
\begin{tikzpicture}[anchorbase, scale=.3]
	\draw [very thick, sea] (0,-2) to [out=90, in=180] (1,0) to [out=0, in=90] (2,-2);
	\node at (0,-2.35) {\tiny $\dicatsgen$};
	\node at (2,-2.35) {\tiny $\dicatsgen$};
	\node at (0,2.35) {\tiny $\phantom{a}$};
\end{tikzpicture},
\]
define cup and cap Soergel diagrams, which are of degree zero.
\end{example}

\subsubsection{``Dihedral Jones--Wenzl projectors''}\label{subsub:JWs}

From now on we write $\iddia$ for Soergel diagrams consisting 
of just identity strands.

\begin{definition}\label{definition:JW}
For $k\in\N$, we define $\JWs{k}$ 
to be the formal $\zq$-linear combination of 
Soergel diagram obtained as follows. 
Set $\JWs{0}=\varnothing$, $\JWs{1}=\iddia$ and
\begin{gather}\label{eq:JW}
\JWs{k}
=
\begin{tikzpicture}[anchorbase, scale=.3]
	\draw [very thick, sea] (5,1.5) to (5,2);
	\draw [very thick, tomato] (3,1.5) to (3,2);
	\draw [very thick, sea] (1,1.5) to (1,2);
	\draw [very thick, orchid] (-1,1.5) to (-1,2);
	\draw [very thick, orchid] (-3,1.5) to (-3,2);
	\draw [very thick, orchid] (-5,1.5) to (-5,2);
	\draw [very thick, orchid] (-7,1.5) to (-7,2);
	\draw [very thick, sea] (5,-2) to (5,-1.5);
	\draw [very thick, tomato] (3,-2) to (3,-1.5);
	\draw [very thick, sea] (1,-2) to (1,-1.5);
	\draw [very thick, orchid] (-1,-2) to (-1,-1.5);
	\draw [very thick, orchid] (-3,-2) to (-3,-1.5);
	\draw [very thick, orchid] (-5,-2) to (-5,-1.5);
	\draw [very thick, orchid] (-7,-2) to (-7,-1.5) to (-7,1.5);
	\draw [thick, dotted] (-5.5,-1.5) rectangle (5.5,1.5);
	\node at (5,-2.35) {\tiny $\dicatsgen$};
	\node at (5,2.35) {\tiny $\dicatsgen$};
	\node at (3,-2.35) {\tiny $\dicattgen$};
	\node at (3,2.35) {\tiny $\dicattgen$};
	\node at (1,-2.35) {\tiny $\dicatsgen$};
	\node at (1,2.35) {\tiny $\dicatsgen$};
	\node at (-1,-2.35) {\tiny $\dicatstgen$};
	\node at (-1,2.35) {\tiny $\dicatstgen$};
	\node at (-3,-2.35) {\tiny $\dicattsgen$};
	\node at (-3,2.35) {\tiny $\dicattsgen$};
	\node at (-5,-2.35) {\tiny $\dicatstgen$};
	\node at (-5,2.35) {\tiny $\dicatstgen$};
	\node at (-7,-2.35) {\tiny $\dicattsgen$};
	\node at (-7,2.35) {\tiny $\dicattsgen$};
	\node at (0,-2.35) {\tiny $\,\cdots$};
	\node at (0,2.35) {\tiny $\,\cdots$};
	\node at (0,0) {\tiny $\JWs{k{-}1}$};
\end{tikzpicture}
+
\tfrac{[k{-}2]_{\qpar}}{[k{-}1]_{\qpar}}\cdot
\begin{tikzpicture}[anchorbase, scale=.3]
	\draw [very thick, sea] (5,1.5) to (5,3);
	\draw [very thick, tomato] (3,1.5) to (3,3);
	\draw [very thick, sea] (1,1.5) to (1,3);
	\draw [very thick, orchid] (-1,1.5) to (-1,3);
	\draw [very thick, orchid] (-3,1.5) to (-3,2) to (-5,2.75);
	\draw [very thick, orchid] (-5,1.5) to (-5,2);
	\draw [very thick, orchid] (-7,1.5) to (-7,2) to (-5,2.75) to (-5,3) to (-5,3.25) to (-7,4) to (-7,4.5);
	\draw [very thick, orchid] (-5,3.25) to (-3,4) to (-3,4.5);
	\draw [very thick, orchid] (-5,4) to (-5,4.5);
	\draw [very thick, orchid] (-1,3) to (-1,4.5);
	\draw [very thick, sea] (1,3) to (1,4.5);
	\draw [very thick, tomato] (3,3) to (3,4.5);
	\draw [very thick, sea] (5,3) to (5,4.5);
	\draw [very thick, sea] (5,-2) to (5,-1.5);
	\draw [very thick, tomato] (3,-2) to (3,-1.5);
	\draw [very thick, sea] (1,-2) to (1,-1.5);
	\draw [very thick, orchid] (-1,-2) to (-1,-1.5);
	\draw [very thick, orchid] (-3,-2) to (-3,-1.5);
	\draw [very thick, orchid] (-5,-2) to (-5,-1.5);
	\draw [very thick, orchid] (-7,-2) to (-7,-1.5) to (-7,1.5);
	\draw [very thick, sea] (5,7.5) to (5,8);
	\draw [very thick, tomato] (3,7.5) to (3,8);
	\draw [very thick, sea] (1,7.5) to (1,8);
	\draw [very thick, orchid] (-1,7.5) to (-1,8);
	\draw [very thick, orchid] (-3,7.5) to (-3,8);
	\draw [very thick, orchid] (-5,7.5) to (-5,8);
	\draw [very thick, orchid] (-7,4.5) to (-7,8);
	\node at (-5,2.1) {\Large $\bulletst$};
	\node at (-5,3.9) {\Large $\bulletst$};
	\draw [thick, dotted] (-5.5,-1.5) rectangle (5.5,1.5);
	\draw [thick, dotted] (-5.5,4.5) rectangle (5.5,7.5);
	\node at (5,-2.35) {\tiny $\dicatsgen$};
	\node at (3,-2.35) {\tiny $\dicattgen$};
	\node at (1,-2.35) {\tiny $\dicatsgen$};
	\node at (-1,-2.35) {\tiny $\dicatstgen$};
	\node at (-3,-2.35) {\tiny $\dicattsgen$};
	\node at (-5,-2.35) {\tiny $\dicatstgen$};
	\node at (-7,-2.35) {\tiny $\dicattsgen$};
	\node at (5,8.35) {\tiny $\dicatsgen$};
	\node at (3,8.35) {\tiny $\dicattgen$};
	\node at (1,8.35) {\tiny $\dicatsgen$};
	\node at (-1,8.35) {\tiny $\dicatstgen$};
	\node at (-3,8.35) {\tiny $\dicattsgen$};
	\node at (-5,8.35) {\tiny $\dicatstgen$};
	\node at (-7,8.35) {\tiny $\dicattsgen$};
	\node at (0,-2.35) {\tiny $\,\cdots$};
	\node at (0,8.35) {\tiny $\,\cdots$};
	\node at (0,3) {\tiny $\,\cdots$};
	\node at (0,0) {\tiny $\JWs{k{-}1}$};
	\node at (0,6) {\tiny $\JWs{k{-}1}$};
\end{tikzpicture}.
\end{gather}
Similarly, we define $\JWt{k}$. Note that for any 
$k\in\N$, $\JWs{k}$ is a $\zq$-linear combination of degree zero Soergel diagrams.
\end{definition}

If $\qpar$ is not a root of unity or $\qpar$ is equal to 
$\pm 1$, then $\JWst{k}$ is well-defined for any $k\in\N$ 
(if one works in e.g. $\C$). If $\qpar$ is a 
complex, primitive $2n$th root of unity, then 
$\JWst{k}$ (with recursion as in 
\fullref{definition:JW}) is well-defined for $0\leq k \leq n$, and 
we can work over $\ZQ$.
Moreover, we require $\qpar$ to be a complex, primitive $2n$th root of unity in the 
proof of \fullref{proposition:locally-finitary} in order to make $\JWst{n-1}$ rotationally 
invariant, see e.g. \cite[before Proposition 1.2]{El1}.

\begin{remark}\label{remark:JW}
Formula \eqref{eq:JW} comes from Wenzl's formula for the Jones--Wenzl projectors in the 
Temperley--Lieb algebra, see \cite{We1}. See also \cite[Definition 4.12]{AT1}. 
\end{remark}

\begin{remark}\label{remark:jw}
From $\JWst{k}$ we obtain a certain diagram $\jwst{k}$. Since we do not need it very often  
in this paper, we refer to \cite[below Definition 5.14]{El1} for its definition. 
For our purposes it is enough to know that any 
$2$-functor maps $\jwst{k}$ to zero if and only if it maps $\JWst{k}$ to zero.
The picture to keep in mind is:
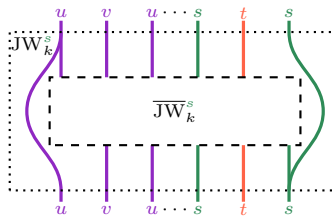
\[
\begin{tikzpicture}[anchorbase, scale=.3]
	\draw [very thick, tomato] (3,1.5) to (3,4);
	\draw [very thick, sea] (1,1.5) to (1,4);
	\draw [very thick, orchid] (-1,1.5) to (-1,4);
	\draw [very thick, orchid] (-3,1.5) to (-3,4);
	\draw [very thick, orchid] (-5,1.5) to (-5,4);
	\draw [very thick, orchid] (-5,3.5) to [out=270, in=90] (-6.5,0) to [out=270, in=90] (-5,-3.5) to (-5,-4);
	\draw [very thick, sea] (5,-4) to (5,-1.5);
	\draw [very thick, sea] (5,4) to (5,3.5) to [out=270, in=90] (6.5,0) to [out=270, in=90] (5,-3.5);
	\draw [very thick, tomato] (3,-4) to (3,-1.5);
	\draw [very thick, sea] (1,-4) to (1,-1.5);
	\draw [very thick, orchid] (-1,-4) to (-1,-1.5);
	\draw [very thick, orchid] (-3,-4) to (-3,-1.5);
	\draw [thick, dashed] (-5.5,-1.5) rectangle (5.5,1.5);
	\draw [thick, dotted] (-7.25,-3.5) rectangle (7.25,3.5);
	\node at (5,-4.35) {\tiny $\dicatsgen$};
	\node at (5,4.35) {\tiny $\dicatsgen$};
	\node at (3,-4.35) {\tiny $\dicattgen$};
	\node at (3,4.35) {\tiny $\dicattgen$};
	\node at (1,-4.35) {\tiny $\dicatsgen$};
	\node at (1,4.35) {\tiny $\dicatsgen$};
	\node at (-1,-4.35) {\tiny $\dicatstgen$};
	\node at (-1,4.35) {\tiny $\dicatstgen$};
	\node at (-3,-4.35) {\tiny $\dicattsgen$};
	\node at (-3,4.35) {\tiny $\dicattsgen$};
	\node at (-5,-4.35) {\tiny $\dicatstgen$};
	\node at (-5,4.35) {\tiny $\dicatstgen$};
	\node at (0,-4.35) {\tiny $\,\cdots$};
	\node at (0,4.35) {\tiny $\,\cdots$};
	\node at (0,0) {\tiny $\jws{k}$};
	\node at (-6.25,2.9) {\tiny $\JWs{k}$};
\end{tikzpicture}
\]
using the convention 
that ``open'' strings are to be closed with dots. This is almost 
a definition of $\jws{k}$, given the relations in the Soergel calculus. 
\end{remark}

\begin{example}\label{example:jw}
With three strands we have the following:
\[
\JWs{3}=
\begin{tikzpicture}[anchorbase, scale=.3]
	\draw [very thick, sea] (-2,-2) to (-2,2);
	\draw [very thick, tomato] (0,-2) to (0,2);
	\draw [very thick, sea] (2,-2) to (2,2);
	\node at (-2,-2.35) {\tiny $\dicatsgen$};
	\node at (-2,2.35) {\tiny $\dicatsgen$};
	\node at (0,-2.35) {\tiny $\dicattgen$};
	\node at (0,2.35) {\tiny $\dicattgen$};
	\node at (2,-2.35) {\tiny $\dicatsgen$};
	\node at (2,2.35) {\tiny $\dicatsgen$};
\end{tikzpicture}
+
\scalebox{.9}{$\nicefrac{$1$}{$[2]_{\qpar}$}$}
\cdot
\begin{tikzpicture}[anchorbase, scale=.3]
	\draw [very thick, sea] (-2,-2) to (0,-.5) to (0,.5) to (-2,2);
	\draw [very thick, tomato] (0,-2) to (0,-1.25);
	\draw [very thick, tomato] (0,1.25) to (0,2);
	\draw [very thick, sea] (2,-2) to (0,-.5) to (0,.5) to (2,2);
	\node at (0,-1.25) {\Large $\bullett$};
	\node at (0,1.25) {\Large $\bullett$};
	\node at (-2,-2.35) {\tiny $\dicatsgen$};
	\node at (-2,2.35) {\tiny $\dicatsgen$};
	\node at (0,-2.35) {\tiny $\dicattgen$};
	\node at (0,2.35) {\tiny $\dicattgen$};
	\node at (2,-2.35) {\tiny $\dicatsgen$};
	\node at (2,2.35) {\tiny $\dicatsgen$};
\end{tikzpicture}
\rightsquigarrow
\jws{3}=
\begin{tikzpicture}[anchorbase, scale=.3]
	\draw [very thick, sea] (2,-2) to (2,-1.25);
	\draw [very thick, sea] (0,1.25) to (0,2);
	\draw [very thick, tomato] (0,-2) to [out=90, in=270] (2,2);
	\node at (0,1.25) {\Large $\bullets$};
	\node at (2,-1.25) {\Large $\bullets$};
	\node at (0,-2.35) {\tiny $\dicattgen$};
	\node at (0,2.35) {\tiny $\dicatsgen$};
	\node at (2,-2.35) {\tiny $\dicatsgen$};
	\node at (2,2.35) {\tiny $\dicattgen$};
\end{tikzpicture}
+
\scalebox{.9}{$\nicefrac{$1$}{$[2]_{\qpar}$}$}
\cdot
\begin{tikzpicture}[anchorbase, scale=.3]
	\draw [very thick, tomato] (0,-2) to (0,-1.25);
	\draw [very thick, tomato] (2,1.25) to (2,2);
	\draw [very thick, sea] (2,-2) to [out=90, in=270] (0,2);
	\node at (0,-1.25) {\Large $\bullett$};
	\node at (2,1.25) {\Large $\bullett$};
	\node at (0,-2.35) {\tiny $\dicattgen$};
	\node at (0,2.35) {\tiny $\dicatsgen$};
	\node at (2,-2.35) {\tiny $\dicatsgen$};
	\node at (2,2.35) {\tiny $\dicattgen$};
\end{tikzpicture}.
\]
More examples can be found in e.g. \cite[Examples 5.16 and 5.17]{El1}.
\end{example}

\subsubsection{The two-color Soergel calculus}\label{subsub:thedhc}

Recall that we have fixed $n\in\Z_{>0}$ or ``$n=\infty$''. In the first case, let $\qpar$ be any primitive $2n$th root of unity. 
In the second case, let $\qpar$ be any non-zero complex number.

\begin{definition}\label{definition:dihedralcat}
Denote by $\dihedralcat{n}$
the additive closure of the $\zq$-linear 
$2$-category, determined by the following data:
\smallskip
\begin{enumerate}[label=(\roman*)]

\setlength\itemsep{.15cm}

\item There is one (not further specified) object.

\item The $1$-morphisms are formal shifts of finite words $\word$ in the 
symbols $\dicatsgen$ and $\dicattgen$. 
(To simplify notation, we often omit these shifts in the diagrams.)

\item The $2$-morphism space 
$\twoHom_{\dihedralcat{n}}(\word\{\shiftme{a}\},\word^{\prime}\{\shiftme{b}\})$ 
is the $\zq$-linear span of all Soergel diagrams from $\word$ 
to $\word^{\prime}$ of degree $\shiftme{b}-\shiftme{a}$, quotiented by the 
relations from \eqref{eq:far-comm} to \eqref{eq:twocolor3}.

\item Vertical composition $\vcomp$ is induced by the vertical gluing of Soergel diagrams, while 
the horizontal composition $\hcomp$ is induced by putting diagrams next to each other horizontally. 
(Using the same reading conventions as above.)
\end{enumerate}
\smallskip
For the usual Eckmann--Hilton relation for $2$-morphisms to hold, 
we have to impose the far-commutativity relation 
(here $f,g$ are two arbitrary Soergel diagrams):
\renewcommand{\theequation}{\arabic{section}.EH}
\begin{gather}\label{eq:far-comm}
\begin{tikzpicture}[anchorbase, scale=.3]
	\draw [very thick, orchid] (-1,-2) to (-1,-1.25);
	\draw [very thick, orchid] (-1,1.25) to (-1,5);
	\draw [very thick, orchid] (-3,-2) to (-3,-1.25);
	\draw [very thick, orchid] (-3,1.25) to (-3,5);
	\draw [very thick, orchid] (1,-2) to (1,1.75);
	\draw [very thick, orchid] (1,4.25) to (1,5);
	\draw [very thick, orchid] (3,-2) to (3,1.75);
	\draw [very thick, orchid] (3,4.25) to (3,5);
	\draw [thick] (-3.25,-1.25) rectangle (-.75,1.25);
	\draw [thick] (.75,1.75) rectangle (3.25,4.25);
	\node at (-3,-2.35) {\tiny $\word_{l}$};
	\node at (-3,5.55) {\tiny $\word_{l^{\prime}}^{\prime}$};
	\node at (-1,-2.35) {\tiny $\word_{k{+}1}$};
	\node at (-1,5.55) {\tiny $\word_{k^{\prime}{+}1}^{\prime}$};
	\node at (1,-2.35) {\tiny $\word_{k}$};
	\node at (1,5.55) {\tiny $\word_{k^{\prime}}^{\prime}$};
	\node at (3,-2.35) {\tiny $\word_{1}$};
	\node at (3,5.55) {\tiny $\word_{1}^{\prime}$};
	\node at (-2,-1.625) {\,$\cdots$};
	\node at (2,-1.625) {\,$\cdots$};
	\node at (-2,4.625) {\,$\cdots$};
	\node at (2,4.625) {\,$\cdots$};
	\node at (-2,0) {\tiny $g$};
	\node at (2,3) {\tiny $f$};
\end{tikzpicture}
=
\begin{tikzpicture}[anchorbase, scale=.3]
	\draw [very thick, orchid] (-1,-2) to (-1,1.75);
	\draw [very thick, orchid] (-1,4.25) to (-1,5);
	\draw [very thick, orchid] (-3,-2) to (-3,1.75);
	\draw [very thick, orchid] (-3,4.25) to (-3,5);
	\draw [very thick, orchid] (1,-2) to (1,-1.25);
	\draw [very thick, orchid] (1,1.25) to (1,5);
	\draw [very thick, orchid] (3,-2) to (3,-1.25);
	\draw [very thick, orchid] (3,1.25) to (3,5);
	\draw [thick] (-3.25,1.75) rectangle (-.75,4.25);
	\draw [thick] (.75,-1.25) rectangle (3.25,1.25);
	\node at (-3,-2.35) {\tiny $\word_{l}$};
	\node at (-3,5.55) {\tiny $\word_{l^{\prime}}^{\prime}$};
	\node at (-1,-2.35) {\tiny $\word_{k{+}1}$};
	\node at (-1,5.55) {\tiny $\word_{k^{\prime}{+}1}^{\prime}$};
	\node at (1,-2.35) {\tiny $\word_{k}$};
	\node at (1,5.55) {\tiny $\word_{k^{\prime}}^{\prime}$};
	\node at (3,-2.35) {\tiny $\word_{1}$};
	\node at (3,5.55) {\tiny $\word_{1}^{\prime}$};
	\node at (-2,-1.625) {\,$\cdots$};
	\node at (2,-1.625) {\,$\cdots$};
	\node at (-2,4.625) {\,$\cdots$};
	\node at (2,4.625) {\,$\cdots$};
	\node at (-2,3) {\tiny $g$};
	\node at (2,0) {\tiny $f$};
\end{tikzpicture}.
\end{gather}
\renewcommand{\theequation}{\arabic{section}.\arabic{equation}}
The other relations among Soergel diagrams are the following. 
First, the Frobenius relations (including the horizontal mirror of \eqref{eq:frob2}):
\\
\noindent\begin{tabularx}{0.99\textwidth}{XX}
\renewcommand{\theequation}{\arabic{section}.Fr1}
\begin{equation}\hspace{-5.2cm}\label{eq:frob1}
\begin{tikzpicture}[anchorbase, scale=.3]
	\draw [very thick, sea] (0,-2) to (1,.5) to (0,2);
	\draw [very thick, sea] (4,-2) to (3,-.5) to (4,2);
	\draw [very thick, sea] (1,.5) to (3,-.5);
	\node at (0,-2.35) {\tiny $\dicatsgen$};
	\node at (0,2.35) {\tiny $\dicatsgen$};
	\node at (4,-2.35) {\tiny $\dicatsgen$};
	\node at (4,2.35) {\tiny $\dicatsgen$};
\end{tikzpicture}
\!=\!
\begin{tikzpicture}[anchorbase, scale=.3]
	\draw [very thick, sea] (0,-2) to (2,-1) to (4,-2);
	\draw [very thick, sea] (0,2) to (2,1) to (4,2);
	\draw [very thick, sea] (2,-1) to (2,1);
	\node at (0,-2.35) {\tiny $\dicatsgen$};
	\node at (0,2.35) {\tiny $\dicatsgen$};
	\node at (4,-2.35) {\tiny $\dicatsgen$};
	\node at (4,2.35) {\tiny $\dicatsgen$};
\end{tikzpicture}
\!=\!
\begin{tikzpicture}[anchorbase, scale=.3]
	\draw [very thick, sea] (0,-2) to (1,-.5) to (0,2);
	\draw [very thick, sea] (4,-2) to (3,.5) to (4,2);
	\draw [very thick, sea] (1,-.5) to (3,.5);
	\node at (0,-2.35) {\tiny $\dicatsgen$};
	\node at (0,2.35) {\tiny $\dicatsgen$};
	\node at (4,-2.35) {\tiny $\dicatsgen$};
	\node at (4,2.35) {\tiny $\dicatsgen$};
\end{tikzpicture},
\end{equation} 
\renewcommand{\theequation}{\arabic{section}.\arabic{equation}}
&
\renewcommand{\theequation}{\arabic{section}.Fr2} 
\begin{equation}\hspace{-7.2cm}\label{eq:frob2}
\begin{tikzpicture}[anchorbase, scale=.3]
	\draw [very thick, sea] (0,-2) to (0,0) to (2,2);
	\draw [very thick, sea] (0,0) to (-1,1);
	\node at (0,-2.35) {\tiny $\dicatsgen$};
	\node at (2,2.35) {\tiny $\dicatsgen$};
	\node at (-1,1) {\Large $\bullets$};
\end{tikzpicture}
\!=\!
\begin{tikzpicture}[anchorbase, scale=.3]
	\draw [very thick, sea] (0,-2) to (0,2);
	\node at (0,-2.35) {\tiny $\dicatsgen$};
	\node at (0,2.35) {\tiny $\dicatsgen$};
\end{tikzpicture}
\!=\!
\begin{tikzpicture}[anchorbase, scale=.3]
	\draw [very thick, sea] (0,-2) to (0,0) to (-2,2);
	\draw [very thick, sea] (0,0) to (1,1);
	\node at (0,-2.35) {\tiny $\dicatsgen$};
	\node at (-2,2.35) {\tiny $\dicatsgen$};
	\node at (1,1) {\Large $\bullets$};
\end{tikzpicture}.
\end{equation}
\renewcommand{\theequation}{\arabic{section}.\arabic{equation}}
\end{tabularx}\\
Then the needle relation (including its horizontal mirror):
\renewcommand{\theequation}{\arabic{section}.Ne} 
\begin{gather}\label{eq:needle}
\begin{tikzpicture}[anchorbase, scale=.3]
	\draw [very thick, sea] (0,-2) to (0,0) to (-1,1);
	\draw [very thick, sea] (0,0) to (1,1);
	\draw [very thick, sea] (-1,1) to [out=135, in=180] (0,3) to [out=0, in=45] (1,1);
	\node at (0,-2.35) {\tiny $\dicatsgen$};
	\node at (0,-3) {\tiny $\phantom{a}$};
\end{tikzpicture}
=
0.
\end{gather}
\renewcommand{\theequation}{\arabic{section}.\arabic{equation}}
The next relations, still independent of $n$, are the barbell forcing relations:
\\
\noindent\begin{tabularx}{0.99\textwidth}{XX}
\renewcommand{\theequation}{\arabic{section}.BF1}
\begin{equation}\hspace{-7.5cm}\label{eq:barb1}
\begin{tikzpicture}[anchorbase, scale=.3]
	\draw [very thick, sea] (0,-2) to (0,2);
	\draw [very thick, sea] (-2,-1) to (-2,1);
	\node at (0,-2.35) {\tiny $\dicatsgen$};
	\node at (0,2.35) {\tiny $\dicatsgen$};
	\node at (-2,-1) {\Large $\bullets$};
	\node at (-2,1) {\Large $\bullets$};
\end{tikzpicture}
\!=\!
2\cdot
\begin{tikzpicture}[anchorbase, scale=.3]
	\draw [very thick, sea] (0,-2) to (0,-1);
	\draw [very thick, sea] (0,1) to (0,2);
	\node at (0,-2.35) {\tiny $\dicatsgen$};
	\node at (0,2.35) {\tiny $\dicatsgen$};
	\node at (0,-1) {\Large $\bullets$};
	\node at (0,1) {\Large $\bullets$};
\end{tikzpicture}
\!-\!
\begin{tikzpicture}[anchorbase, scale=.3]
	\draw [very thick, sea] (0,-2) to (0,2);
	\draw [very thick, sea] (2,-1) to (2,1);
	\node at (0,-2.35) {\tiny $\dicatsgen$};
	\node at (0,2.35) {\tiny $\dicatsgen$};
	\node at (2,-1) {\Large $\bullets$};
	\node at (2,1) {\Large $\bullets$};
\end{tikzpicture},
\end{equation}
\renewcommand{\theequation}{\arabic{section}.\arabic{equation}} 
&
\renewcommand{\theequation}{\arabic{section}.BF2'}
\begin{equation}\hspace{-5.0cm}\label{eq:barb2-prime}
\begin{tikzpicture}[anchorbase, scale=.3]
	\draw [very thick, sea] (0,-2) to (0,2);
	\draw [very thick, tomato] (-2,-1) to (-2,1);
	\node at (0,-2.35) {\tiny $\dicatsgen$};
	\node at (0,2.35) {\tiny $\dicatsgen$};
	\node at (-2,-1) {\Large $\bullett$};
	\node at (-2,1) {\Large $\bullett$};
\end{tikzpicture}
\!=\!
\begin{tikzpicture}[anchorbase, scale=.3]
	\draw [very thick, sea] (0,-2) to (0,2);
	\draw [very thick, tomato] (2,-1) to (2,1);
	\node at (0,-2.35) {\tiny $\dicatsgen$};
	\node at (0,2.35) {\tiny $\dicatsgen$};
	\node at (2,-1) {\Large $\bullett$};
	\node at (2,1) {\Large $\bullett$};
\end{tikzpicture}
\!+
[2]_{\qpar}\cdot
\begin{tikzpicture}[anchorbase, scale=.3]
	\draw [very thick, sea] (0,-2) to (0,2);
	\draw [very thick, sea] (2,-1) to (2,1);
	\node at (0,-2.35) {\tiny $\dicatsgen$};
	\node at (0,2.35) {\tiny $\dicatsgen$};
	\node at (2,-1) {\Large $\bullets$};
	\node at (2,1) {\Large $\bullets$};
\end{tikzpicture}
\!-[2]_{\qpar}\cdot
\begin{tikzpicture}[anchorbase, scale=.3]
	\draw [very thick, sea] (0,-2) to (0,-1);
	\draw [very thick, sea] (0,1) to (0,2);
	\node at (0,-2.35) {\tiny $\dicatsgen$};
	\node at (0,2.35) {\tiny $\dicatsgen$};
	\node at (0,-1) {\Large $\bullets$};
	\node at (0,1) {\Large $\bullets$};
\end{tikzpicture}.
\end{equation}
\renewcommand{\theequation}{\arabic{section}.\arabic{equation}}
\end{tabularx}\\
Finally, three relations which depend on $n$ 
(including the horizontal mirror of the displayed ones, the versions for even $n$, 
and all possibilities for the position of the dot 
in \eqref{eq:twocolor2}):
\\
\noindent\begin{tabularx}{0.99\textwidth}{XX}
\renewcommand{\theequation}{\arabic{section}.$2n$v1}
\begin{equation}\hspace{-6cm}\label{eq:twocolor1}
\begin{tikzpicture}[anchorbase, scale=.3]
	\draw [very thick, tomato] (-1,2) to (-1,1.25);
	\draw [very thick, tomato] (1,2) to (1,1.25);
	\draw [very thick, tomato] (1,1.625) to (3,1.625) to (3,2);
	\draw [very thick, sea] (-1,-2) to (-1,-1.25);
	\draw [very thick, sea] (1,-2) to (1,-1.25);
	\draw [thick] (-1.5,-1.25) rectangle (1.5,1.25);
	\node at (-1,2.35) {\tiny $\dicattgen$};
	\node at (1,2.35) {\tiny $\dicattgen$};
	\node at (3,2.35) {\tiny $\dicattgen$};
	\node at (-1,-2.35) {\tiny $\dicatsgen$};
	\node at (1,-2.35) {\tiny $\dicatsgen$};
	\node at (0,-2.35) {\tiny $\cdots$};
	\node at (0,2.35) {\tiny $\cdots$};
	\node at (0,0) {\tiny $n$};
	\node at (0,-3.85) {$\phantom{s}$};
	\node at (0,3.85) {$\phantom{t}$};
\end{tikzpicture}
=
\begin{tikzpicture}[anchorbase, scale=.3]
	\draw [very thick, tomato] (0,.5) to (0,-.5);
	\draw [very thick, tomato] (-1,2) to (-1,1.25);
	\draw [very thick, tomato] (1,2) to (1,1.25);
	\draw [very thick, tomato] (3,-.5) to (3,2);
	\draw [very thick, sea] (-1,-2) to (-1,.5);
	\draw [very thick, sea] (2,-2) to (2,-1.25);
	\draw [very thick, sea] (2,-.5) to (2,.5);
	\draw [very thick, sea] (-1,-1.625) to (0,-1.625) to (0,-1.25);
	\draw [thick] (-1.5,1.25) rectangle (2.5,.5);
	\draw [thick] (-.5,-1.25) rectangle (3.5,-.5);
	\node at (-1,2.35) {\tiny $\dicattgen$};
	\node at (1,2.35) {\tiny $\dicattgen$};
	\node at (3,2.35) {\tiny $\dicattgen$};
	\node at (-1,-2.35) {\tiny $\dicatsgen$};
	\node at (2,-2.35) {\tiny $\dicatsgen$};
	\node at (.5,-2.35) {\tiny $\cdots$};
	\node at (0,2.35) {\tiny $\cdots$};
	\node at (1,0) {\tiny $\cdots$};
	\node at (.5,.875) {\tiny $n$};
	\node at (1.5,-.875) {\tiny $n$};
	\node at (0,-3.85) {$\phantom{s}$};
	\node at (0,3.85) {$\phantom{t}$};
\end{tikzpicture},
\end{equation}
\renewcommand{\theequation}{\arabic{section}.\arabic{equation}} 
&
\renewcommand{\theequation}{\arabic{section}.$2n$v2}
\begin{equation}\hspace{-6cm}\label{eq:twocolor2}
\begin{tikzpicture}[anchorbase, scale=.3]
	\draw [very thick, tomato] (-2,2) to (-2,1.25);
	\draw [very thick, tomato] (2,2) to (2,1.25);
	\draw [very thick, tomato] (-1,-2) to (-1,-1.25);
	\draw [very thick, tomato] (1,-2) to (1,-1.25);
	\draw [very thick, tomato] (0,2) to (0,1.25);
	\draw [very thick, sea] (-2,-2) to (-2,-1.25);
	\draw [very thick, sea] (2,-2) to (2,-1.25);
	\draw [very thick, sea] (-1,2) to (-1,1.25);
	\draw [very thick, sea] (1,2) to (1,1.25);
	\draw [very thick, sea] (0,-2) to (0,-1.25);
	\draw [thick] (-2.5,-1.25) rectangle (2.5,1.25);
	\node at (-1,-2.35) {\tiny $\dicattgen$};
	\node at (1,-2.35) {\tiny $\dicattgen$};
	\node at (-2,2.35) {\tiny $\dicattgen$};
	\node at (2,2.35) {\tiny $\dicattgen$};
	\node at (0,2) {\Large $\bullett$};
	\node at (-1,2.35) {\tiny $\dicatsgen$};
	\node at (1,2.35) {\tiny $\dicatsgen$};
	\node at (-2,-2.35) {\tiny $\dicatsgen$};
	\node at (2,-2.35) {\tiny $\dicatsgen$};
	\node at (0,-2.35) {\tiny $\dicatsgen$};
	\node at (-1.35,-2.35) {\tiny $\cdot$};
	\node at (-1.35,2.35) {\tiny $\cdot$};
	\node at (-1.65,-2.35) {\tiny $\cdot$};
	\node at (-1.65,2.35) {\tiny $\cdot$};
	\node at (1.35,-2.35) {\tiny $\cdot$};
	\node at (1.35,2.35) {\tiny $\cdot$};
	\node at (1.65,-2.35) {\tiny $\cdot$};
	\node at (1.65,2.35) {\tiny $\cdot$};
	\node at (0,0) {\tiny $n$};
	\node at (0,-3.85) {$\phantom{s}$};
	\node at (0,3.85) {$\phantom{t}$};
\end{tikzpicture}
=
\begin{tikzpicture}[anchorbase, scale=.3]
	\draw [very thick, tomato] (-2,2) to (-2,1.25);
	\draw [very thick, tomato] (2,2) to (2,1.25);
	\draw [very thick, tomato] (-1,-2) to (-1,-1.25);
	\draw [very thick, tomato] (1,-2) to (1,-1.25);
	\draw [very thick, sea] (-2,-2) to (-2,-1.25);
	\draw [very thick, sea] (2,-2) to (2,-1.25);
	\draw [very thick, sea] (-1,2) to (-1,1.625) to (0,1.625) to (0,1.25);
	\draw [very thick, sea] (1,2) to (1,1.625) to (0,1.625);
	\draw [very thick, sea] (0,-2) to (0,-1.25);
	\draw [thick, dashed] (-2.5,-1.25) rectangle (2.5,1.25);
	\node at (-1,-2.35) {\tiny $\dicattgen$};
	\node at (1,-2.35) {\tiny $\dicattgen$};
	\node at (-2,2.35) {\tiny $\dicattgen$};
	\node at (2,2.35) {\tiny $\dicattgen$};
	\node at (-1,2.35) {\tiny $\dicatsgen$};
	\node at (1,2.35) {\tiny $\dicatsgen$};
	\node at (-2,-2.35) {\tiny $\dicatsgen$};
	\node at (2,-2.35) {\tiny $\dicatsgen$};
	\node at (0,-2.35) {\tiny $\dicatsgen$};
	\node at (-1.35,-2.35) {\tiny $\cdot$};
	\node at (-1.35,2.35) {\tiny $\cdot$};
	\node at (-1.65,-2.35) {\tiny $\cdot$};
	\node at (-1.65,2.35) {\tiny $\cdot$};
	\node at (1.35,-2.35) {\tiny $\cdot$};
	\node at (1.35,2.35) {\tiny $\cdot$};
	\node at (1.65,-2.35) {\tiny $\cdot$};
	\node at (1.65,2.35) {\tiny $\cdot$};
	\node at (0,0) {\tiny $\jwst{n}$};
	\node at (0,-3.85) {$\phantom{s}$};
	\node at (0,3.85) {$\phantom{t}$};
\end{tikzpicture},
\end{equation}
\renewcommand{\theequation}{\arabic{section}.\arabic{equation}}
\end{tabularx}
\renewcommand{\theequation}{\arabic{section}.$2n$v3}
\begin{equation}\label{eq:twocolor3}
\begin{tikzpicture}[anchorbase, scale=.3]
	\draw [very thick, sea] (-1,2) to (-1,1.25);
	\draw [very thick, sea] (1,2) to (1,1.25);
	\draw [very thick, tomato] (-1,-2) to (-1,-1.25);
	\draw [very thick, tomato] (1,-2) to (1,-1.25);
	\draw [very thick, sea] (1,2) to [out=90, in=180] (1.5,2.5) to [out=0, in=90] (2,2) to (2,-3.5);
	\draw [very thick, sea] (-1,2) to [out=90, in=180] (1.5,3.5) to [out=0, in=90] (4,2) to (4,-3.5);
	\draw [very thick, tomato] (-1,-2) to [out=270, in=0] (-1.5,-2.5) to [out=180, in=270] (-2,-2) to (-2,3.5);
	\draw [very thick, tomato] (1,-2) to [out=270, in=0] (-1.5,-3.5) to [out=180, in=270] (-4,-2) to (-4,3.5);
	\draw [thick] (-1.5,-1.25) rectangle (1.5,1.25);
	\node at (4,-3.85) {\tiny $\dicatsgen$};
	\node at (2,-3.85) {\tiny $\dicatsgen$};
	\node at (-4,3.85) {\tiny $\dicattgen$};
	\node at (-2,3.85) {\tiny $\dicattgen$};
	\node at (-3,3.85) {\tiny $\cdots$};
	\node at (3,-3.85) {\tiny $\cdots$};
	\node at (0,-1.75) {\tiny $\cdots$};
	\node at (0,1.75) {\tiny $\cdots$};
	\node at (0,0) {\tiny $n$};
\end{tikzpicture}
=
\begin{tikzpicture}[anchorbase, scale=.3]
	\draw [very thick, tomato] (-1,2) to (-1,1.25);
	\draw [very thick, tomato] (1,2) to (1,1.25);
	\draw [very thick, sea] (-1,-2) to (-1,-1.25);
	\draw [very thick, sea] (1,-2) to (1,-1.25);
	\draw [thick] (-1.5,-1.25) rectangle (1.5,1.25);
	\node at (-1,2.35) {\tiny $\dicattgen$};
	\node at (1,2.35) {\tiny $\dicattgen$};
	\node at (-1,-2.35) {\tiny $\dicatsgen$};
	\node at (1,-2.35) {\tiny $\dicatsgen$};
	\node at (0,-2.35) {\tiny $\cdots$};
	\node at (0,2.35) {\tiny $\cdots$};
	\node at (0,0) {\tiny $n$};
	\node at (0,-3.85) {$\phantom{s}$};
	\node at (0,3.85) {$\phantom{t}$};
\end{tikzpicture}
=
\begin{tikzpicture}[anchorbase, scale=.3]
	\draw [very thick, sea] (1,2) to (1,1.25);
	\draw [very thick, sea] (-1,2) to (-1,1.25);
	\draw [very thick, tomato] (1,-2) to (1,-1.25);
	\draw [very thick, tomato] (-1,-2) to (-1,-1.25);
	\draw [very thick, sea] (-1,2) to [out=90, in=0] (-1.5,2.5) to [out=180, in=90] (-2,2) to (-2,-3.5);
	\draw [very thick, sea] (1,2) to [out=90, in=0] (-1.5,3.5) to [out=180, in=90] (-4,2) to (-4,-3.5);
	\draw [very thick, tomato] (1,-2) to [out=270, in=180] (1.5,-2.5) to [out=0, in=270] (2,-2) to (2,3.5);
	\draw [very thick, tomato] (-1,-2) to [out=270, in=180] (1.5,-3.5) to [out=0, in=270] (4,-2) to (4,3.5);
	\draw [thick] (-1.5,-1.25) rectangle (1.5,1.25);
	\node at (-4,-3.85) {\tiny $\dicatsgen$};
	\node at (-2,-3.85) {\tiny $\dicatsgen$};
	\node at (4,3.85) {\tiny $\dicattgen$};
	\node at (2,3.85) {\tiny $\dicattgen$};
	\node at (3,3.85) {\tiny $\cdots$};
	\node at (-3,-3.85) {\tiny $\cdots$};
	\node at (0,-1.75) {\tiny $\cdots$};
	\node at (0,1.75) {\tiny $\cdots$};
	\node at (0,0) {\tiny $n$};
\end{tikzpicture}.
\end{equation}
\renewcommand{\theequation}{\arabic{section}.\arabic{equation}}

Moreover, all of the relations listed above, one- and two-colored, 
exist in two versions, i.e. the displayed ones 
and their color inverted ($\dicatsgen\rightleftarrows\dicattgen$) counterparts.
\end{definition}

The $2$-category $\dihedralcat{\infty}$ is defined similarly, but using only 
the generators and relations which are independent of $n$.

Let $\dihedralcatbig{}$ denote the $2$-category obtained 
from $\dihedralcat{}$ as explained in \fullref{subsec:basics}.
Since the relations \eqref{eq:far-comm} to \eqref{eq:twocolor2} 
are homogeneous with respect to the degrees of the Soergel diagrams, 
the $2$-category $\dihedralcatbig{}$ is additive, graded and $\zq$-linear.

\begin{example}\label{example:iso-rels}
Isotopy relations such as ``zigzag relations'' and other as e.g.
\begin{gather*}
\begin{tikzpicture}[anchorbase, scale=.3]
	\draw [very thick, sea] (0,-2) to (0,0) to [out=90, in=180] (1,1.5) to [out=0, in=90] (2,0);
	\draw [very thick, sea] (2,0) to [out=270, in=180] (3,-1.5) to [out=0, in=270] (4,0) to (4,2);
	\node at (0,-2.35) {\tiny $\dicatsgen$};
	\node at (4,2.35) {\tiny $\dicatsgen$};
\end{tikzpicture}
=
\begin{tikzpicture}[anchorbase, scale=.3]
	\draw [very thick, sea] (0,-2) to (0,2);
	\node at (0,-2.35) {\tiny $\dicatsgen$};
	\node at (0,2.35) {\tiny $\dicatsgen$};
\end{tikzpicture}
=
\begin{tikzpicture}[anchorbase, scale=.3]
	\draw [very thick, sea] (0,2) to (0,0) to [out=270, in=180] (1,-1.5) to [out=0, in=270] (2,0);
	\draw [very thick, sea] (2,0) to [out=90, in=180] (3,1.5) to [out=0, in=90] (4,0) to (4,-2);
	\node at (0,2.35) {\tiny $\dicatsgen$};
	\node at (4,-2.35) {\tiny $\dicatsgen$};
\end{tikzpicture}
\;,\;
\begin{tikzpicture}[anchorbase, scale=.3]
	\draw [very thick, sea] (0,-2) to (0,0) to [out=90, in=180] (1,2) to [out=0, in=90] (2,0);
	\node at (0,-2.35) {\tiny $\dicatsgen$};
	\node at (0,2.35) {\tiny $\phantom{a}$};
	\node at (2,0) {\Large $\bullets$};
\end{tikzpicture}
=
\begin{tikzpicture}[anchorbase, scale=.3]
	\draw [very thick, sea] (0,-2) to (0,0);
	\node at (0,-2.35) {\tiny $\dicatsgen$};
	\node at (0,2.35) {\tiny $\phantom{a}$};
	\node at (0,0) {\Large $\bullets$};
\end{tikzpicture}
=
\begin{tikzpicture}[anchorbase, scale=.3]
	\draw [very thick, sea] (0,-2) to (0,0) to [out=90, in=0] (-1,2) to [out=180, in=90] (-2,0);
	\node at (0,-2.35) {\tiny $\dicatsgen$};
	\node at (0,2.35) {\tiny $\phantom{a}$};
	\node at (-2,0) {\Large $\bullets$};
\end{tikzpicture}
\;,\;
\begin{tikzpicture}[anchorbase, scale=.3]
	\draw [very thick, sea] (-1,-2) to (-1,2);
	\draw [very thick, sea] (-1,0) to [out=315, in=180] (.25,-1) to [out=0, in=270] (1,0) to (1,2);
	\node at (-1,-2.35) {\tiny $\dicatsgen$};
	\node at (-1,2.35) {\tiny $\dicatsgen$};
	\node at (1,2.35) {\tiny $\dicatsgen$};
\end{tikzpicture}
=
\begin{tikzpicture}[anchorbase, scale=.3]
	\draw [very thick, sea] (0,-2) to (0,0) to (2,2);
	\draw [very thick, sea] (0,0) to (-2,2);
	\node at (0,-2.35) {\tiny $\dicatsgen$};
	\node at (-2,2.35) {\tiny $\dicatsgen$};
	\node at (2,2.35) {\tiny $\dicatsgen$};
\end{tikzpicture}
=
\begin{tikzpicture}[anchorbase, scale=.3]
	\draw [very thick, sea] (1,-2) to (1,2);
	\draw [very thick, sea] (1,0) to [out=225, in=0] (-.25,-1) to [out=180, in=270] (-1,0) to (-1,2);
	\node at (1,-2.35) {\tiny $\dicatsgen$};
	\node at (1,2.35) {\tiny $\dicatsgen$};
	\node at (-1,2.35) {\tiny $\dicatsgen$};
\end{tikzpicture}
\end{gather*}
are consequences of the Frobenius relations from \eqref{eq:frob1} 
and \eqref{eq:frob2}. However, the isotopy 
relations in \eqref{eq:twocolor3} 
are not consequences of the Frobenius relations.
\end{example}

\begin{remark}\label{remark:bf-equi-a}
When $2$ is invertible it is not hard to show that  
\eqref{eq:barb2-prime} can actually be replaced by
\renewcommand{\theequation}{\arabic{section}.BF2}
\begin{gather}\label{eq:barb2}
2\cdot
\left(
\begin{tikzpicture}[anchorbase, scale=.3]
	\draw [very thick, sea] (0,-2) to (0,2);
	\draw [very thick, tomato] (-2,-1) to (-2,1);
	\node at (0,-2.35) {\tiny $\dicatsgen$};
	\node at (0,2.35) {\tiny $\dicatsgen$};
	\node at (-2,-1) {\Large $\bullett$};
	\node at (-2,1) {\Large $\bullett$};
\end{tikzpicture}
-
\begin{tikzpicture}[anchorbase, scale=.3]
	\draw [very thick, sea] (0,-2) to (0,2);
	\draw [very thick, tomato] (2,-1) to (2,1);
	\node at (0,-2.35) {\tiny $\dicatsgen$};
	\node at (0,2.35) {\tiny $\dicatsgen$};
	\node at (2,-1) {\Large $\bullett$};
	\node at (2,1) {\Large $\bullett$};
\end{tikzpicture}
\right)
=
-[2]_{\qpar}\cdot
\left(
\begin{tikzpicture}[anchorbase, scale=.3]
	\draw [very thick, sea] (0,-2) to (0,2);
	\draw [very thick, sea] (-2,-1) to (-2,1);
	\node at (0,-2.35) {\tiny $\dicatsgen$};
	\node at (0,2.35) {\tiny $\dicatsgen$};
	\node at (-2,-1) {\Large $\bullets$};
	\node at (-2,1) {\Large $\bullets$};
\end{tikzpicture}
-
\begin{tikzpicture}[anchorbase, scale=.3]
	\draw [very thick, sea] (0,-2) to (0,2);
	\draw [very thick, sea] (2,-1) to (2,1);
	\node at (0,-2.35) {\tiny $\dicatsgen$};
	\node at (0,2.35) {\tiny $\dicatsgen$};
	\node at (2,-1) {\Large $\bullets$};
	\node at (2,1) {\Large $\bullets$};
\end{tikzpicture}
\right).
\end{gather}
\renewcommand{\theequation}{\arabic{section}.\arabic{equation}}
We use this equivalent relation later in \fullref{subsub:functor-infty-2}.
\end{remark}

The following is a direct consequence 
of \cite[Theorems 5.29 and 6.24]{El1}. To be consistent 
with our conventions from \fullref{sec:2reps}, we switch to a field 
$\someR$ containing $\zq$.

\begin{propositionn}\label{proposition:locally-finitary}
The $2$-categories 
$\Kar(\dihedralcat{n})^{\star}$ and $\Kar(\dihedralcat{\infty})^{\star}$ 
are graded finitary and 
graded locally finitary, respectively.\qedmake
\end{propositionn}
\subsection{Its \texorpdfstring{$2$}{2}-representations coming from bipartite graphs}\label{subsec:diacataction}

The purpose of the present subsection is to define the weak $2$-functors
\[
\functorGno\colon\dihedralcatbig{}\to\Endff(\cellcatG)
\]
(in our usual convention, we write $\functorGno$ for either 
$\functorG$ or $\functorADE$)
from \fullref{theorem:main-theorem}, for any given 
bipartite graph $\g$. As we will see in \fullref{subsec:infinite-case}, 
the $2$-functor $\functorGno$ is well-defined only for certain values of $\qpar\in\C-\{0\}$, which depend on $\g$.

Now, $\functorGno$ sends the unique object of $\dihedralcatbig{}$ to $\cellcatG$. In \fullref{subsec:quiver-stuff} 
we already defined $\functorGno$ on $1$-morphisms:
\[
\functorGno(\dicatsgen)=\disfun,
\quad\functorGno(\dicattgen)=\ditfun.
\]
Next, we define $\functorGno$ on 
$2$-morphisms, which we first 
do for $\dihedralcatbig{\infty}$ and then for $\dihedralcatbig{n}$. 

Moreover, we start with $\g$'s of $\ADE$ type (our main interest), and then discuss the 
generalization to arbitrary bipartite graphs.

\subsubsection{Assignment in the infinite case}\label{subsub:functor-infty}
Fix $\qpar\in\C-\{0\}$.
We have to assign natural transformations to the generators  
from \eqref{eq:dicat-gen1}. Recall that the functors $\functors$ and 
$\functort$ are given by tensoring with the sum 
of the $\algG$-bimodules 
$P_{\bbk{i}}\{\shiftme{-1}\}\otimes{}_{\bbk{i}}P$ over all 
{\color{sea}sea-green} $\bbii{i}$ and over all {\color{tomato}tomato} $\bbjj{j}$ colored 
vertices of $\g$, respectively. Thus, $\algG$-bimodule maps between tensor products of 
the $\algG$-bimodules $P_{\bbk{i}}\{\shiftme{-1}\}\otimes{}_{\bbk{i}}P$ induce 
natural transformations between the corresponding composites of 
$\functors$ and $\functort$. (This assignment is not 
strict, cf. \fullref{example:main-2-cats-first}.) Hence, we first specify a 
$\algG$-bimodule map for 
each of the generating Soergel diagrams and then check that our assignment preserves the 
relations \eqref{eq:far-comm} to \eqref{eq:barb2} of the two-color Soergel calculus. 

To understand our assignment below recall that, 
for all vertices $\bbk{i}\in\g$, there are 
free $\zq$-modules ${}_{\bbk{i}}P_{\bbk{i}}$ 
given by
\begin{gather}\label{eq:bimod-recall-first}
{}_{\bbk{i}}P_{\bbk{i}}
=
{}_{\bbk{i}}P\qgtimes P_{\bbk{i}}\cong \zq(\bbk{i})\oplus\zq(\bbk{i}|\bbk{i}),
\end{gather}
where the isomorphism, which we fix, is given by $y\qgtimes x\mapsto y\pathm x$. 
We will use \eqref{eq:bimod-recall-first} strategically below. 
Moreover, there are graded
$\algG$-bimodules
\begin{align*}
P_{\bbk{i}}\{\shiftme{-1}\}\otimes{}_{\bbk{i}}P
\cong
&
\left(
\zq(\bbk{i})\oplus\zq(\bbk{i}|\bbk{i})\oplus
{\textstyle\bigoplus_{\bbk{i}\gconnect{}\bbk{j}}}\,
\zq(\bbk{j}|\bbk{i})
\right)
\{\shiftme{-1}\}
\\
&\otimes
\left(
\zq(\bbk{i})
\oplus
\zq(\bbk{i}|\bbk{i})
\oplus
{\textstyle\bigoplus_{\bbk{i}\gconnect{}\bbk{j}}}\,
\zq(\bbk{i}|\bbk{j})
\right),
\end{align*}
\begin{gather*}
{}_{\bbk{j}}P_{\bbk{i}}
=
{}_{\bbk{j}}P\qgtimes P_{\bbk{i}}\cong\zq(\bbk{j}|\bbk{i}),
\end{gather*}
which easily follows from \eqref{eq:projectives} and the same isomorphism as before, respectively.
Hereby we recall that the left action is given 
by post-composition, and the right action by pre-composition of paths.

In the following we only give some of the $\algG$-bimodule maps. 
The other ones can be obtained from these via color inversion 
($\dicatsgen\rightleftarrows\dicattgen$) and 
interchanging $\graphs$ and $\grapht$.

Some of the maps below are weighted sums. 
The \textit{weights} $\lscalar_{\bbk{i}}$ are invertible elements 
of $\zq$, which depend on $\g$, and will be 
defined in \fullref{definition:these-numbers}.

In the assignment below, we fix $\bbii{i}\in\g$, 
while $\bbjj{j}$ always means a vertex of 
$\g$ connected to $\bbii{i}$. We then give each 
$\algG$-bimodule map only on certain basis elements. The rest of 
the map is determined by the 
fact that it is supposed to be a $\algG$-bimodule map. 
(We will check in \fullref{lemma:bimodule-maps} that our assignments are indeed $\algG$-bimodule maps.) 

In order to get started,
let $x_{\bbii{i}}\in P_{\bbii{i}}$ 
and ${}_{\bbii{i}}y\in {}_{\bbii{i}}P$, and
let us write $\bigoplus_{\bbii{i}}=\bigoplus_{\bbii{i}\in\g}$ etc. for short, 
and notations of the form $x_{\bbii{i}} \otimes {}_{\bbii{i}}y$ indicate 
that we take the corresponding entries from the direct sums. 
(We extend the below $\zq$-linearly.)
\medskip

\noindent \textit{Identity generators.}
To these we 
assign the corresponding identity maps.
\medskip

\noindent \textit{Dots.} We choose the following $\algG$-bimodule maps.
\renewcommand{\theequation}{\arabic{section}.d1}
\begin{align}\label{eq:assignment-dotsA-a}
\begin{tikzpicture}[anchorbase, scale=.3]
	\draw [very thick, sea] (0,-2) to (0,0);
	\node at (0,-2.35) {\tiny $\dicatsgen$};
	\node at (0,2.35) {\tiny $\phantom{a}$};
	\node at (0,0) {\Large $\bullets$};
\end{tikzpicture}
&\stackrel{\functorGno}{\rightsquigarrow}
\begin{cases}
{\textstyle\bigoplus_{\bbii{i}}}\, P_{\bbii{i}}\{\shiftme{-1}\}\otimes {}_{\bbii{i}} P & \to \algG,\\
x_{\bbii{i}} \otimes {}_{\bbii{i}}y & \mapsto 
x_{\bbii{i}}\pathm {}_{\bbii{i}}y,
\end{cases}
\end{align}
\renewcommand{\theequation}{\arabic{section}.d2}
\begin{align}\label{eq:assignment-dotsA-b}
\begin{tikzpicture}[anchorbase, scale=.3]
	\draw [very thick, sea] (0,2) to (0,0);
	\node at (0,2.35) {\tiny $\dicatsgen$};
	\node at (0,-2.35) {\tiny $\phantom{a}$};
	\node at (0,0) {\Large $\bullets$};
\end{tikzpicture}
&\stackrel{\functorGno}{\rightsquigarrow}
\begin{cases}
\algG&\to {\textstyle\bigoplus_{\bbii{i}}}\, P_{\bbii{i}}\{\shiftme{-1}\}\otimes {}_{\bbii{i}} P,\\
\bbii{i}&\mapsto \lscalar_{\bbii{i}}\cdot (\bbii{i} \otimes \bbii{i}\vert \bbii{i}+\bbii{i}\vert \bbii{i} \otimes \bbii{i}),\\
\bbjj{j}&\mapsto
{\textstyle\sum_{\bbii{i}\gconnect{}\bbjj{j}}}\,(\lscalar_{\bbii{i}}\cdot \bbjj{j}|\bbii{i}\otimes\bbii{i}|\bbjj{j}).
\end{cases}
\end{align}
\medskip

\noindent \textit{Trivalent vertices.}
Using the identification 
from \eqref{eq:bimod-recall-first}, we pick:
\renewcommand{\theequation}{\arabic{section}.t1}
\begin{gather}\label{eq:assignment-triA-a}
\begin{tikzpicture}[anchorbase, scale=.3]
	\draw [very thick, sea] (0,-2) to (0,0) to (-2,2);
	\draw [very thick, sea] (0,0) to (2,2);
	\node at (0,-2.35) {\tiny $\dicatsgen$};
	\node at (-2,2.35) {\tiny $\dicatsgen$};
	\node at (2,2.35) {\tiny $\dicatsgen$};
%
\end{tikzpicture}
\stackrel{\functorGno}{\rightsquigarrow}
\begin{cases}
{\textstyle\bigoplus_{\bbii{i}}}\, P_{\bbii{i}}\{\shiftme{-1}\}\otimes {}_{\bbii{i}}P
&\to 
{\textstyle\bigoplus_{\bbii{i}}}\, P_{\bbii{i}}\{\shiftme{-1}\}\otimes {}_{\bbii{i}}P_{\bbii{i}}\{\shiftme{-1}\}\otimes {}_{\bbii{i}}P,\\
x_{\bbii{i}}\otimes {}_{\bbii{i}}y&\mapsto  
x_{\bbii{i}}\otimes \bbii{i}\otimes {}_{\bbii{i}}y,
\end{cases}
\end{gather}
\renewcommand{\theequation}{\arabic{section}.t2}
\begin{gather}\label{eq:assignment-triA-b}
\begin{tikzpicture}[anchorbase, scale=.3]
	\draw [very thick, sea] (0,2) to (0,0) to (-2,-2);
	\draw [very thick, sea] (0,0) to (2,-2);
	\node at (0,2.35) {\tiny $\dicatsgen$};
	\node at (-2,-2.35) {\tiny $\dicatsgen$};
	\node at (2,-2.35) {\tiny $\dicatsgen$};
%
\end{tikzpicture}
\stackrel{\functorGno}{\rightsquigarrow}
\begin{cases}
{\textstyle\bigoplus_{\bbii{i}}}\, P_{\bbii{i}}\{\shiftme{-1}\}\otimes {}_{\bbii{i}}
P_{\bbii{i}}\{\shiftme{-1}\}\otimes {}_{\bbii{i}}P
&\to {\textstyle\bigoplus_{\bbii{i}}}\, P_{\bbii{i}}\{\shiftme{-1}\}\otimes {}_{\bbii{i}}P,\\
x_{\bbii{i}}\otimes \bbii{i} \otimes {}_{\bbii{i}}y & \mapsto 0,\\
x_{\bbii{i}}\otimes \bbii{i}|\bbii{i} \otimes {}_{\bbii{i}}y & \mapsto \lscalar_{\bbii{i}}^{-1}\cdot (x_{\bbii{i}}\otimes {}_{\bbii{i}}y).
\end{cases}
\end{gather}\renewcommand{\theequation}{\arabic{section}.\arabic{equation}}
\medskip

\noindent \textit{{\color{tomato}Tomato} $\dihtgen$ colored diagrams.}
We inverted the colors ($\dicatsgen\rightleftarrows\dicattgen$), including $\lscalar_{\bbii{i}}\rightleftarrows\lscalar_{\bbjj{j}}$.
\medskip

Our assignments above extend to any $2$-morphism in $\dihedralcatbig{\infty}$ which is written as 
a horizontal and vertical composite of the generators (using \eqref{eq:Godement}). 

\begin{example}\label{example:2functor-assignment}
Let us denote by $\nattrafo{d}$ the natural transformation 
induced by the $\algG$-bimodule 
map from \eqref{eq:assignment-dotsA-a}. 
Consider 
the following two natural transformations. 
\[
\begin{tikzpicture}[anchorbase, scale=.3]
	\draw [very thick, sea] (-2,-2) to (-2,2);
	\draw [very thick, sea] (0,-2) to (0,0);
	\node at (0,-2.35) {\tiny $\dicatsgen$};
	\node at (0,2.35) {\tiny $\phantom{a}$};
	\node at (-2,-2.35) {\tiny $\dicatsgen$};
	\node at (-2,2.35) {\tiny $\dicatsgen$};
	\node at (0,0) {\Large $\bullets$};
\end{tikzpicture}
\stackrel{\functorGno}{\rightsquigarrow}
\nattrafo{id}\hcomp\nattrafo{d}\colon
\functors\functors\Rightarrow\functors,
\quad\quad
\begin{tikzpicture}[anchorbase, scale=.3]
	\draw [very thick, sea] (2,-2) to (2,2);
	\draw [very thick, sea] (0,-2) to (0,0);
	\node at (0,-2.35) {\tiny $\dicatsgen$};
	\node at (0,2.35) {\tiny $\phantom{a}$};
	\node at (2,-2.35) {\tiny $\dicatsgen$};
	\node at (2,2.35) {\tiny $\dicatsgen$};
	\node at (0,0) {\Large $\bullets$};
\end{tikzpicture}
\stackrel{\functorGno}{\rightsquigarrow}
\nattrafo{d}\hcomp\nattrafo{id}\colon
\functors\functors\Rightarrow\functors.
\]
Now -- by the above -- we have 
using \eqref{eq:bimod-recall-first} (which we will do silently from now on)
\[
\nattrafo{id}\hcomp\nattrafo{d}
\leftrightsquigarrow
\begin{cases}
{\textstyle\bigoplus_{\bbii{i}}}\,
P_{\bbii{i}}\{\shiftme{-1}\}\otimes {}_{\bbii{i}}P_{\bbii{i}}\{\shiftme{-1}\}\otimes {}_{\bbii{i}}P
&
\to 
{\textstyle\bigoplus_{\bbii{i}}}\,P_{\bbii{i}}\{\shiftme{-1}\}\otimes {}_{\bbii{i}}P,\\
x_{\bbii{i}}\otimes {}_{\bbii{i}}y^{\prime}
\pathm x^{\prime}_{\bbii{i}}\otimes {}_{\bbii{i}}y
&\mapsto
x_{\bbii{i}}\otimes \left({}_{\bbii{i}}y^{\prime}
\pathm x^{\prime}_{\bbii{i}}\right)\pathm {}_{\bbii{i}}y.
\end{cases}
\]
\[
\nattrafo{d}\hcomp\nattrafo{id}
\leftrightsquigarrow
\begin{cases}
{\textstyle\bigoplus_{\bbii{i}}}\, P_{\bbii{i}}\{\shiftme{-1}\}\otimes {}_{\bbii{i}}P_{\bbii{i}}\{\shiftme{-1}\}\otimes {}_{\bbii{i}}P
&\to {\textstyle\bigoplus_{\bbii{i}}}\, P_{\bbii{i}}\{\shiftme{-1}\}\otimes {}_{\bbii{i}}P,\\
x_{\bbii{i}}\otimes {}_{\bbii{i}}y^{\prime}\pathm x^{\prime}_{\bbii{i}}\otimes {}_{\bbii{i}}y
&\mapsto
x_{\bbii{i}}\pathm\left({}_{\bbii{i}}y^{\prime}\pathm x^{\prime}_{\bbii{i}}\right)
\otimes {}_{\bbii{i}}y.
\end{cases}
\]
Hereby $x_{\bbii{i}},x^{\prime}_{\bbii{i}}\in P_{\bbii{i}}$ 
and ${}_{\bbii{i}}y,{}_{\bbii{i}}y^{\prime}\in {}_{\bbii{i}}P$, as before.
\end{example}

\subsubsection{Assignment in the finite case}\label{subsub:functor-finite}

Let $\g$ be of $\ADE$ type and $\qpar$ be -- as usual -- a complex, primitive 
$2n$th root of unity. 
We continue to write $\functorGno$, whose definition we need to complete on the $2n$-valent vertices.
To this end, recall $\Functor{\Theta}_{\overline{\dihsgen_{n}}}$ 
and $\Functor{\Theta}_{\overline{\dihtgen_{n}}}$ from \fullref{subsec:dihedral-stuff}.
\medskip

\textit{$2n$-valent vertices.} We assign the zero maps:
\renewcommand{\theequation}{\arabic{section}.$2n$v}
\begin{gather}\label{eq:assignment-2nvertex}
\begin{tikzpicture}[anchorbase, scale=.3]
	\draw [very thick, orchid] (-1,-2) to (-1,-1.5);
	\draw [very thick, sea] (1,-2) to (1,-1.5);
	\draw [very thick, orchid] (-1,2) to (-1,1.5);
	\draw [very thick, tomato] (1,2) to (1,1.5);
	\draw [thick] (-1.5,-1.5) rectangle (1.5,1.5);
	\node at (-1,2.35) {\tiny $\dicattsgen$};
	\node at (1,2.35) {\tiny $\dicattgen$};
	\node at (-1,-2.35) {\tiny $\dicatstgen$};
	\node at (1,-2.35) {\tiny $\dicatsgen$};
	\node at (0,-2.35) {\tiny $\,\cdots$};
	\node at (0,2.35) {\tiny $\,\cdots$};
	\node at (0,0) {\tiny $n$};
\end{tikzpicture}
\stackrel{\functorGno}{\rightsquigarrow} 0,
\quad\quad
\begin{tikzpicture}[anchorbase, scale=.3]
	\draw [very thick, orchid] (-1,-2) to (-1,-1.5);
	\draw [very thick, tomato] (1,-2) to (1,-1.5);
	\draw [very thick, orchid] (-1,2) to (-1,1.5);
	\draw [very thick, sea] (1,2) to (1,1.5);
	\draw [thick] (-1.5,-1.5) rectangle (1.5,1.5);
	\node at (-1,2.35) {\tiny $\dicattsgen$};
	\node at (1,2.35) {\tiny $\dicatsgen$};
	\node at (-1,-2.35) {\tiny $\dicatstgen$};
	\node at (1,-2.35) {\tiny $\dicattgen$};
	\node at (0,-2.35) {\tiny $\,\cdots$};
	\node at (0,2.35) {\tiny $\,\cdots$};
	\node at (0,0) {\tiny $n$};
\end{tikzpicture}
\stackrel{\functorGno}{\rightsquigarrow} 0.
\end{gather}
\renewcommand{\theequation}{\arabic{section}.\arabic{equation}}
These give rise to the zero natural transformations between 
$\Functor{\Theta}_{\overline{\dihsgen_{n}}}$ 
and $\Functor{\Theta}_{\overline{\dihtgen_{n}}}$, and 
between $\Functor{\Theta}_{\overline{\dihtgen_{n}}}$ 
and $\Functor{\Theta}_{\overline{\dihsgen_{n}}}$, respectively
Again, we extend everything horizontally (using \eqref{eq:Godement}).

\subsubsection{The weighting}\label{subsub:functor-infty-2}
Fix any $\qpar\in\C-\{0\}$. Recall that $V=\graphs\,{\textstyle\coprod}\,\grapht$ 
denotes the two-colored edge set of $\g$.

\begin{definition}\label{definition:these-numbers-2}
A \textit{weighting $\llscalar$
of $\g$} is an assignment
\[
V\to\C-\{0\},\quad
\bbii{i}\mapsto\lscalar_{\bbii{i}},\;
\bbjj{j}\mapsto\lscalar_{\bbjj{j}}.
\]
The scalars $\lscalar_{\bbk{i}}$ are called \textit{weights}.

Fixing an ordering of the vertices of $\g$ as in 
\fullref{subsub:basic-graphs-spectra} 
we can write
\[
\llscalar=
(\lscalar_{\bbii{i}_1},\dots,\lscalar_{\bbii{i}_{|\graphs|}},
\lscalar_{\bbjj{j}_1},\dots,\lscalar_{\bbjj{j}_{|\grapht|}}).
\]
Then the triple $(\g,\llscalar,\qpar)$ is 
called a \textit{\eqref{eq:barb2}-weighting} 
if
\begin{gather}\label{eq:eigenvector-problem}
A(\g)\llscalar=-[2]_{\qpar}\cdot\llscalar, 
\end{gather}
where $A(\g)$ is the adjacency matrix of $\g$ (in the evident ordering).
\end{definition}

Note that \eqref{eq:eigenvector-problem} is equivalent to 
\begin{gather}\label{eq:eigenvector-problem-2}
-[2]_{\qpar}\cdot\lscalar_{\bbk{i}}=
{\textstyle\sum_{\bbk{i}\gconnect{}\bbk{j}}}
\lscalar_{\bbk{j}},\;\text{for all }\bbk{i}\in\g.
\end{gather}

\begin{remark}\label{remark:bf-equi}
Let us explain where the 
condition \eqref{eq:eigenvector-problem-2} comes from. Recall that we 
assume that $2$ is invertible, and 
we can replace \eqref{eq:barb2-prime} by \eqref{eq:barb2}, cf. \fullref{remark:bf-equi-a}. 

Thus, in order for the corresponding 
sums of \eqref{eq:floating-barbell1} and \eqref{eq:floating-barbell2} 
to work out 
(these appear in the proof that 
\eqref{eq:barb2} holds in the 
$2$-representation), we need precisely 
condition \eqref{eq:eigenvector-problem-2} 
to be satisfied. Hence, the name \eqref{eq:barb2}-weighting.
\end{remark}

\begin{definition}\label{definition:these-numbers}
Let $\g$ be of $\ADE$ type and $n$ be its Coxeter number, 
where we use the conventions from \ref{enum:typeA}, \ref{enum:typeD} 
and \ref{enum:typeE}. Let $\qpar$ be the usual (fixed) complex, primitive $2n$-root of unity.

We define \eqref{eq:barb2}-weightings for these $\g$'s as follows.
The scalars $\lscalar_{\bbk{i}}$ do not depend on the two-coloring, and 
are given next to the vertices.
\begin{subequations} 
\begingroup
\renewcommand{\theequation}{$\typea{m}$}
\begin{gather}\label{eq:crazy-numbers-an}
\begin{tikzpicture}[anchorbase, scale=1]
	\draw [thick] (0,0) to (2.825,0);
	\draw [thick] (3.375,0) to (6.2,0);
	\node at (0,-.01) {\Large $\bullet$};
	\node at (1,-.01) {\Large $\bullet$};
	\node at (2,-.01) {\Large $\bullet$};
	\node at (3.1,-.01) {\Large $\,\cdots$};
	\node at (4.2,-.01) {\Large $\bullet$};
	\node at (5.2,-.01) {\Large $\bullet$};
	\node at (6.2,-.01) {\Large $\bullet$};
	\node at (0,-.3) {\tiny ${+}[1]_{\qpar}$};
	\node at (1,-.3) {\tiny ${-}[2]_{\qpar}$};
	\node at (2,-.3) {\tiny ${+}[3]_{\qpar}$};
	\node at (4.15,-.3) {\tiny ${\mp}[m{-}2]_{\qpar}$};
	\node at (5.25,-.3) {\tiny ${\pm}[m{-}1]_{\qpar}$};
	\node at (6.2,-.3) {\tiny ${\mp}[m]_{\qpar}$};
\end{tikzpicture},
\end{gather}
\endgroup
\end{subequations}
\addtocounter{equation}{-1}
\begin{subequations}
\begingroup
\renewcommand{\theequation}{$\typed{m}$}
\begin{gather}\label{eq:crazy-numbers-dn}
\begin{tikzpicture}[anchorbase, scale=1]
	\draw [thick] (1.55,0) to (3.375,0);
	\draw [thick] (3.925,0) to (5.8,0);
	\draw [thick] (6.3,-.866) to (5.8,0) to (6.3,.866);
	\node at (1.55,-.01) {\Large $\bullet$};
	\node at (2.55,-.01) {\Large $\bullet$};
	\node at (3.65,-.01) {\Large $\,\cdots$};
	\node at (4.8,0) {\Large $\bullet$};
	\node at (5.8,0) {\Large $\bullet$};
	\node at (6.3,.866) {\Large $\bullet$};
	\node at (6.3,-.866) {\Large $\bullet$};
	\node at (1.55,-.3) {\tiny ${+}[1]_{\qpar}$};
	\node at (2.55,-.3) {\tiny ${-}[2]_{\qpar}$};
	\node at (4.8,-.3) {\tiny ${\mp}[m{-}3]_{\qpar}$};
	\node at (6.5,0) {\tiny ${\pm}[m{-}2]_{\qpar}$};
	\node at (7.0,.866) {\tiny ${\mp}\nicefrac{$[m{-}1]$}{$2$}$};
	\node at (7.0,-.866) {\tiny ${\mp}\nicefrac{$[m{-}1]$}{$2$}$};
\end{tikzpicture},
\end{gather}
\endgroup
\end{subequations}
\addtocounter{equation}{-1}
\begin{subequations}
\begingroup
\renewcommand{\theequation}{$\typee{6}$}
\begin{gather}\label{eq:crazy-numbers-e6}
\begin{tikzpicture}[anchorbase, scale=1]
	\draw [thick] (0,0) to (4,0);
	\draw [thick] (2,0) to (2,1);
	\node at (0,-.01) {\Large $\bullet$};
	\node at (1,-.01) {\Large $\bullet$};
	\node at (2,-.01) {\Large $\bullet$};
	\node at (3,-.01) {\Large $\bullet$};
	\node at (4,-.01) {\Large $\bullet$};
	\node at (2,.99) {\Large $\bullet$};
	\node at (0,-.3) {\tiny ${+}[1]_{\qpar}$};
	\node at (1,-.3) {\tiny ${-}[2]_{\qpar}$};
	\node at (2,1.3) {\tiny ${-}\nicefrac{$[3]_{\qpar}$}{$[2]_{\qpar}$}$};
	\node at (2,-.3) {\tiny ${+}[3]_{\qpar}$};
	\node at (3,-.3) {\tiny ${-}[2]_{\qpar}$};
	\node at (4,-.3) {\tiny ${+}[1]_{\qpar}$};
\end{tikzpicture},
\end{gather}
\endgroup
\end{subequations}
\addtocounter{equation}{-1}
\begin{subequations}
\begingroup
\renewcommand{\theequation}{$\typee{7}$}
\begin{gather}\label{eq:crazy-numbers-e7}
\begin{tikzpicture}[anchorbase, scale=1]
	\draw [thick] (-1,0) to (4,0);
	\draw [thick] (2,0) to (2,1);
	\node at (-1,-.01) {\Large $\bullet$};
	\node at (0,-.01) {\Large $\bullet$};
	\node at (1,-.01) {\Large $\bullet$};
	\node at (2,-.01) {\Large $\bullet$};
	\node at (3,-.01) {\Large $\bullet$};
	\node at (4,-.01) {\Large $\bullet$};
	\node at (2,.99) {\Large $\bullet$};
	\node at (-1,-.3) {\tiny ${+}[1]_{\qpar}$};
	\node at (0,-.3) {\tiny ${-}[2]_{\qpar}$};
	\node at (1,-.3) {\tiny ${+}[3]_{\qpar}$};
	\node at (2,1.3) {\tiny ${+}\nicefrac{$[4]_{\qpar}$}{$[2]_{\qpar}$}$};
	\node at (2,-.3) {\tiny ${-}[4]_{\qpar}$};
	\node at (2.925,-.3) {\tiny ${+}\nicefrac{$[6]_{\qpar}$}{$[2]_{\qpar}$}$};
	\node at (4.075,-.3) {\tiny ${-}\nicefrac{$[4]_{\qpar}$}{$[3]_{\qpar}$}$};
\end{tikzpicture},
\end{gather}
\endgroup
\end{subequations}
\addtocounter{equation}{-1}
\begin{subequations}
\begingroup
\renewcommand{\theequation}{$\typee{8}$}
\begin{gather}\label{eq:crazy-numbers-e8}
\begin{tikzpicture}[anchorbase, scale=1]
	\draw [thick] (-2,0) to (4,0);
	\draw [thick] (2,0) to (2,1);
	\node at (-2,-.01) {\Large $\bullet$};
	\node at (-1,-.01) {\Large $\bullet$};
	\node at (0,-.01) {\Large $\bullet$};
	\node at (1,-.01) {\Large $\bullet$};
	\node at (2,-.01) {\Large $\bullet$};
	\node at (3,-.01) {\Large $\bullet$};
	\node at (4,-.01) {\Large $\bullet$};
	\node at (2,.99) {\Large $\bullet$};
	\node at (-2,-.3) {\tiny ${+}[1]_{\qpar}$};
	\node at (-1,-.3) {\tiny ${-}[2]_{\qpar}$};
	\node at (0,-.3) {\tiny ${+}[3]_{\qpar}$};
	\node at (1,-.3) {\tiny ${-}[4]_{\qpar}$};
	\node at (2,1.3) {\tiny ${-}\nicefrac{$[5]_{\qpar}$}{$[2]_{\qpar}$}$};
	\node at (2,-.3) {\tiny ${+}[5]_{\qpar}$};
	\node at (2.925,-.3) {\tiny ${-}\nicefrac{$[7]_{\qpar}$}{$[2]_{\qpar}$}$};
	\node at (4.075,-.3) {\tiny ${+}\nicefrac{$[5]_{\qpar}$}{$[3]_{\qpar}$}$};
\end{tikzpicture}.
\end{gather}
\endgroup
\end{subequations}
\addtocounter{equation}{-1}
This defines the triple $(\g,\llscalar,\qpar)$.
\end{definition}

\begin{remark}\label{remark:these-numbers}
These $\lscalar_{\bbk{i}}$'s can be found 
solving the eigenvector problem for the 
adjacency matrix of $\g$, cf. \eqref{eq:eigenvector-problem}. 
The weights are actually 
the entries of the Perron--Frobenius eigenvector 
-- normalized to have the entry $1$ at a fixed vertex --
for the smallest eigenvalue. (Recall hereby that the spectrum of a 
bipartite graph is a symmetric set. Hence, there is a unique smallest eigenvalue. 
See also in the proof of \fullref{lemma:FP-in-action}.)

The reader familiar with \cite{KO1} might also note that the weights above 
are equal to the quantum numbers associated 
to the vertices of $\g$ in \cite[Section 6]{KO1}, after division by 
the quantum number associated to the left-most vertex 
(and adding signs to match our conventions).
\end{remark}

It remains 
to prove well-definedness, 
which we will do in \fullref{sec:proofs}, i.e. we will 
show that the $2$-functor $\functorGno$ is well-defined if and only if 
the chosen weighting is \eqref{eq:barb2}.

\subsubsection{``Uniqueness'' of the \texorpdfstring{$2$}{2}-functor \texorpdfstring{$\functorGno$}{G}}\label{subsub:functor-unique}

For $\g$ of $\ADE$ type, the assignment from above is essentially unique:

\begin{lemma}\label{lemma:unique-maps}
Let $\g$ be of $\ADE$ type and fix $\qpar$ 
to be a complex, primitive $2n$th root of unity. 
For any additive, degree-preserving, $\C$-linear, weak $2$-functor
\[
\functorGs\colon\dihedralcatbig{\infty}\to\Endff(\cellcatG)
\]
which agrees with $\functorGno$ on $1$-morphisms, there exist  
scalars $\scalargt,\scalargtt\in\C-\{0\}$ 
such that
\[
\functorGs(f)=
\scalargt^{d_{\mathrm{st}}-t_{\mathrm{me}}}\cdot
\scalargtt^{d_{\mathrm{end}}-t_{\mathrm{sp}}}\cdot
\functorGno(f),
\]
for any homogeneous $2$-morphism $f$ built from 
$d_{\mathrm{st}}$ start dots, $t_{\mathrm{me}}$ merges, $d_{\mathrm{end}}$ end dots 
and $t_{\mathrm{sp}}$ splits.

In particular, $\functorGs$ is equivalent to $\functorGno$.
\end{lemma}

\begin{lemma}\label{lemma:unique-maps2}
Any additive, degree-preserving, $\C$-linear, weak $2$-functor
\[
\functorGs\colon\dihedralcatbig{n}\to\Endff(\cellcatG)
\]
which agrees with $\functorGno$ on $1$-morphisms sends 
the $2n$-valent vertices to zero.
\end{lemma}

\subsubsection{More general bipartite graphs}\label{subsub:functor-bi-graphs}
We will now briefly discuss \eqref{eq:barb2}-weightings for graphs which are not of $\ADE$ type. 

We first give two examples: 

\begin{example}\label{example:weighting-bi-graph}
Here are two examples for $\g$'s being not of $\ADE$ type:
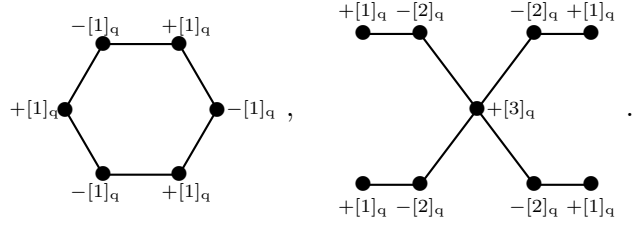
\[
\begin{tikzpicture}[anchorbase, scale=1]
	\draw [thick] (-1,0) to (-.5,.866) to (.5,.866) to (1,0) to (.5,-.866) to (-.5,-.866) to (-1,0);
	\node at (-1,-.01) {\Large $\bullet$};
	\node at (-.5,.866) {\Large $\bullet$};
	\node at (.5,.866) {\Large $\bullet$};
	\node at (1,-.01) {\Large $\bullet$};
	\node at (.5,-.866) {\Large $\bullet$};
	\node at (-.5,-.866) {\Large $\bullet$};
	\node at (-1.4,0) {\tiny ${+}[1]_{\qpar}$};
	\node at (-.6,1.1) {\tiny ${-}[1]_{\qpar}$};
	\node at (.6,1.1) {\tiny ${+}[1]_{\qpar}$};
	\node at (1.45,0) {\tiny ${-}[1]_{\qpar}$};
	\node at (.6,-1.1) {\tiny ${+}[1]_{\qpar}$};
	\node at (-.6,-1.1) {\tiny ${-}[1]_{\qpar}$};
\end{tikzpicture}
,\quad
\begin{tikzpicture}[anchorbase, scale=1]
	\draw[thick] (1.5,1) to (.75,1) to (0,0);
	\draw[thick] (-1.5,1) to (-.75,1) to (0,0);
	\draw[thick] (1.5,-1) to (.75,-1) to (0,0);
	\draw[thick] (-1.5,-1) to (-.75,-1) to (0,0);
	\node at (0,0) {\Large $\bullet$};
	\node at (.75,1) {\Large $\bullet$};
	\node at (-.75,1) {\Large $\bullet$};
	\node at (.75,-1) {\Large $\bullet$};
	\node at (-.75,-1) {\Large $\bullet$};
	\node at (1.5,1) {\Large $\bullet$};
	\node at (-1.5,1) {\Large $\bullet$};
	\node at (1.5,-1) {\Large $\bullet$};
	\node at (-1.5,-1) {\Large $\bullet$};
	\node at (1.5,1.3) {\tiny ${+}[1]_{\qpar}$};
	\node at (-1.5,1.3) {\tiny ${+}[1]_{\qpar}$};
	\node at (1.5,-1.3) {\tiny ${+}[1]_{\qpar}$};
	\node at (-1.5,-1.3) {\tiny ${+}[1]_{\qpar}$};
	\node at (.75,1.3) {\tiny ${-}[2]_{\qpar}$};
	\node at (-.75,1.3) {\tiny ${-}[2]_{\qpar}$};
	\node at (.75,-1.3) {\tiny ${-}[2]_{\qpar}$};
	\node at (-.75,-1.3) {\tiny ${-}[2]_{\qpar}$};
	\node at (.45,0) {\tiny ${+}[3]_{\qpar}$};
\end{tikzpicture}.
\]
The first weighting is \eqref{eq:barb2} if and only 
if $-[2]_{\qpar}\cdot\pm[1]_{\qpar}=2\cdot\mp[1]_{\qpar}$, i.e. if and only if 
$q=1$. The second weighting is \eqref{eq:barb2} if and only if $-[4]_{\qpar}+3\cdot[2]_{\qpar}=0$ 
and all involved weights are non-zero. This happens if 
and only if 
$\qpar\in\{\scalebox{.9}{$\nicefrac{$1$}{$2$}$}\cdot(1\pm\sqrt{5}),
-\scalebox{.9}{$\nicefrac{$1$}{$2$}$}\cdot(1\pm\sqrt{5})\}$.  
\end{example}

In general, we will prove the following result. 

\begin{lemma}\label{lemma:FP-in-action}
For any bipartite graph $\g$, there 
is at least one value of 
$\qpar\in\C-\{0\}$ such that a \eqref{eq:barb2}-weighting exists.
\end{lemma}

Note that \fullref{lemma:FP-in-action} is not immediate, since 
we need all weights to be invertible.
\section{Proofs}\label{sec:proofs}
Finally, we give the proofs of all statements.
\subsection{The uncategorified story}\label{subsec:GG-story}

Fix a bipartite graph $\g$. The following lemma follows from 
the fact that the $P_i\{\shiftme{-1}\}\otimes {}_i P$'s are graded biprojective 
$\algG$-bimodules. 

\begin{lemman}\label{lemma:projective}
The functors $\functors$ and $\functort$ are degree-preserving and biprojective.\qedmake
\end{lemman}

\begin{lemma}\label{lemma:graded-adjoint-infty}
The functors $\functors$ and $\functort$ are additive, $\zq$-linear 
and self-adjoint.
\end{lemma}

\begin{proof}
We only treat the case of $\functors$ here, the other is similar and omitted.

Up to the self-adjointness of $\functors$ 
the statement is clear. 
Moreover, we claim that 
\[
\begin{tikzpicture}[anchorbase, scale=.3]
	\draw [very thick, sea] (0,2) to [out=270, in=180] (1,0) to [out=0, in=270] (2,2);
	\node at (0,2.35) {\tiny $\dicatsgen$};
	\node at (2,2.35) {\tiny $\dicatsgen$};
	\node at (0,-2.35) {\tiny $\phantom{a}$};
\end{tikzpicture}
\stackrel{\functorGno}{\rightsquigarrow}
\nattrafo{i}\colon\functor{ID}\Rightarrow\functors\functors,
\quad\quad
\begin{tikzpicture}[anchorbase, scale=.3]
	\draw [very thick, sea] (0,-2) to [out=90, in=180] (1,0) to [out=0, in=90] (2,-2);
	\node at (0,-2.35) {\tiny $\dicatsgen$};
	\node at (2,-2.35) {\tiny $\dicatsgen$};
	\node at (0,2.35) {\tiny $\phantom{a}$};
\end{tikzpicture}
\stackrel{\functorGno}{\rightsquigarrow}
\nattrafo{e}\colon\functors\functors\Rightarrow\functor{ID},
\]
form the unit respectively counit of the self-adjunction for $\functors$. 
Next, there are two things to check. Namely that $\functors$ 
is left adjoint to itself, and 
that it is right adjoint to itself, i.e. we have to show that
\begin{gather*}
\nattrafo{id}_{\functors}
=\nattrafo{e}\functors
\vcomp\functors\nattrafo{i},
\quad\quad
\nattrafo{id}_{\functors}
=\functors\nattrafo{e}
\vcomp\nattrafo{i}\functors.
\end{gather*}
Recalling the definitions 
of cup and cap from \fullref{example:pitchfork}, we see that graphically this is 
the ``zigzag relation'' from \fullref{example:iso-rels}.
Hence, the claim of self-adjointness follows from 
the well-definedness of $\functorGno$, which we show later in \fullref{subsec:infinite-case}. 
\end{proof}

Alternatively, after checking that 
the underlying quiver algebra is
weakly symmetric and self-injective 
-- which can be done by e.g. 
copying \cite[Proposition 1]{HK1} -- one could also 
use \cite[Section 7.3]{MM1} 
to prove \fullref{lemma:graded-adjoint-infty} 
abstractly. In contrast, our proof above fixes the natural 
transformations realizing the adjunctions.

\begin{proof}[Proof of \fullref{proposition:action-infty}]
First, by using \fullref{lemma:graded-adjoint-infty}, we see that the two functors $\functors$ 
and $\functort$ are exact and descent to the Grothendieck group.

Next -- using \eqref{eq:functor-projectives-infty} -- we get for $\bbii{i}\in\graphs$ that
\begin{gather*}
\functors\functors(P_{\bbii{i}})\cong P_{\bbii{i}}\{\shiftme{-2}\}\oplus P_{\bbii{i}}\oplus P_{\bbii{i}}\oplus P_{\bbii{i}}\{\shiftme{+2}\}
\cong\functors(P_{\bbii{i}})\{\shiftme{-1}\}\oplus\functors(P_{\bbii{i}})\{\shiftme{+1}\},\\
\functort\functort(P_{\bbii{i}})\cong{\textstyle\bigoplus_{\bbii{i}\gconnect{}\bbjj{j}}}\, P_{\bbjj{j}}\{\shiftme{-1}\}\oplus P_{\bbjj{j}}\{\shiftme{+1}\}
\cong\functort(P_{\bbii{i}})\{\shiftme{-1}\}\oplus\functort(P_{\bbii{i}})\{\shiftme{+1}\}.
\end{gather*}
We get a similar result for $P_{\bbjj{j}}$, with $\bbjj{j}\in\grapht$. Therefore, we get the following natural isomorphisms 
of degree-preserving functors:
\begin{gather*}
\functors\functors\cong\functors\{\shiftme{-1}\}\oplus\functors\{\shiftme{+1}\},\quad\quad
\functort\functort\cong\functort\{\shiftme{-1}\}\oplus\functort\{\shiftme{+1}\}.
\end{gather*}
Since these are the defining relations 
of $\mathrm{H}_{\infty}$ from \eqref{eq:def-relationsa}, we see that $\GG{\cellcatG}$ 
is indeed an $\mathrm{H}_{\infty}$-module. Next, 
choosing the evident basis of $\GG{\cellcatG}$ 
given by the $[P_{\bbk{i}}]$'s and 
using \eqref{eq:functor-projectives-infty}, one can see that  
$[\functors]$ and $[\functort]$ act as in \eqref{eq:action-downstairs}. 
This shows that $\zeta_{\mathrm{G}}$ is 
an $\mathrm{H}_{\infty}$-homomorphism. That $\zeta_{\mathrm{G}}$ is bijective is clear.
\end{proof}

\begin{proof}[Proof of \fullref{proposition:action-finite}]
The claim basically follows
by observing that the scalars from \eqref{eq:numbers} for the defining relations 
of the $\disgen$ and $\ditgen$ generators of $\mathrm{H}$ are 
the coefficients of the (normalized) 
Chebyshev polynomials (of the second kind) given by
\[
\tilde{U}_{0}=1,\quad\tilde{U}_1=X,
\quad\tilde{U}_{k+1}=X\,\tilde{U}_k-\tilde{U}_{k-1}\text{ for }k\in\Z_{\geq 1}. 
\]
To give a few more details: 
Given a polynomial in $X$, one can obtain 
a non-commutative polynomial in two variables -- 
say $\disgen$ and $\ditgen$ -- by 
replacing $X^k$ with an alternating string $\dots\disgen\ditgen\disgen$ 
of length $k$ (always having $\disgen$ to the right).
We write $\tilde{U}_{k}(\disgen,\ditgen)$ for the non-commutative polynomial obtained 
from $\tilde{U}_{k}$ in this way.
Then \eqref{eq:numbers} 
implies that $\klbasis_{\overline{\dihsgen_k}}$ is $\tilde{U}_{k}(\disgen,\ditgen)$. 

Now, if a representation of $\mathrm{H}_{\infty}$ factors through 
$\mathrm{H}_{n}$ and is annihilated by $\klbasis_{\wnull}$, 
then 
\[
\tilde{U}_{n{-}1}(\disgen,\ditgen)=0=\tilde{U}_{n{-}1}(\ditgen,\disgen). 
\]
It is not hard to 
deduce from this that the eigenvalues of the matrices associated to 
$\disgen$ and $\ditgen$ in the representations from 
\fullref{definition:module-for-G}
(these are -- up to base change -- of the form
$\begin{psmallmatrix}
[2]_{\vpar}^{|\graphs|} & A\\
0 & 0
\end{psmallmatrix}$ respectively $\begin{psmallmatrix}
0 & 0\\
A^{\mathrm{T}} & [2]_{\vpar}^{|\grapht|}
\end{psmallmatrix}$) have to 
be (multi)subsets of the (multi)subset of eigenvalues of $X\,\tilde{U}_{n{-}1}$.
This follows from \eqref{eq:eigenvalue} and the observation stated after it.

The (normalized) 
Chebyshev polynomials have only real roots which are given by 
$S_{\typea{m}}\subset\; ]-2,2[$ in \eqref{eq:spectrum} below.
Now, it follows 
from \cite[Theorem 2]{Smi1} that 
the largest eigenvalue of $AA^{\mathrm{T}}$ or 
$A^{\mathrm{T}}A$ is strictly 
less than $4$ if and only if 
$\g$ is as in \ref{enum:typeA}, \ref{enum:typeD} 
or \ref{enum:typeE}. (See also \cite[Theorem 3.1.3]{BH1} for a more recent proof 
using Perron--Frobenius theory.) 
This shows -- keeping \eqref{eq:eigenvalue} in mind -- that only a 
$\g$ as in \ref{enum:typeA}, \ref{enum:typeD} 
or \ref{enum:typeE} (if non-trivial) can give a well-defined action of $\mathrm{H}_n$.

Moreover, we have the following spectra:
\begin{gather}\label{eq:spectrum}
\begin{aligned}
S_{\typea{m}}&=\left\{2\cos\left(
\scalebox{.9}{$\nicefrac{$k\pi$}{$m{+}1$}$}
\right)
\middle|k=1,\dots,m
\right\}
,\\
S_{\typed{m}}&=\left\{2\cos\left(
\scalebox{.9}{$\nicefrac{$k\pi$}{$2m{-}2$}$}
\right)
\middle|k=1,3,5,\dots,2m-7,2m-5,2m-3
\right\}\cup\{0\},\\
S_{\typee{6}}&=\left\{2\cos\left(
\scalebox{.9}{$\nicefrac{$k\pi$}{$12$}$}
\right)
\middle|k=1,4,5,7,8,11
\right\}
,\\
S_{\typee{7}}&=\left\{2\cos\left(
\scalebox{.9}{$\nicefrac{$k\pi$}{$18$}$}
\right)
\middle|k=1,5,7,9,11,13,17
\right\},\\
S_{\typee{8}}&=\left\{2\cos\left(
\scalebox{.9}{$\nicefrac{$k\pi$}{$30$}$}
\right)
\middle|k=1,7,11,13,17,19,23,29
\right\}.
\end{aligned}
\end{gather}
This list is known, see e.g. \cite[Section 3.1.1]{BH1}. 
The type $\typea{m}$ spectrum is the (multi)set of roots of 
the polynomial $\tilde{U}_{m}$.
The fact that all $\ADE$ type graphs give 
well-defined actions of $\mathrm{H}_n$ then follows easily 
using these spectra.

In case 
$\klbasis_{\wnull}$ does not act as zero the explicit 
form of the matrices associated to 
$\disgen$ and $\ditgen$ immediately shows that any well-defined 
action of $\mathrm{H}_n$ has to be trivial, i.e. both -- $\disgen$ 
and $\ditgen$ --
have to act as zero.
\end{proof}

\begin{remark}\label{remark-affine-radius}
In \fullref{example:typeA-calc1} we have seen that 
the largest eigenvalue of $AA^{\mathrm{T}}$ and 
$A^{\mathrm{T}}A$ in type $\typeat{3}$ was $4$. This is true for all 
affine types $\typeAt$, $\typeDt$ and $\typeEt$ and these 
are also the only graphs with this property, see 
e.g. \cite[Section 3.1.1]{BH1}.
\end{remark}

The following is now a direct 
consequence of \fullref{proposition:action-finite}.
(Hereby $[\Functor{\Theta}_{\wnull}]$ corresponds to $\klbasis_{\wnull}$ under $\zeta_{\G}$ 
from \eqref{eq:matching-actions}.)

\begin{corollaryn}\label{corollary:w0-is-zero}
In the setup from \fullref{subsec:quiver-stuff}: 
$[\Functor{\Theta}_{\wnull}]=0$ 
if and only if $\g$ is as 
in \ref{enum:typeA}, \ref{enum:typeD} 
or \ref{enum:typeE} (if non-trivial).\qedmake 
\end{corollaryn}

We finish this section with the proof of \fullref{lemma:FP-in-action}, because 
its proof is very much in the spirit of the proof of 
\fullref{proposition:action-finite} above.

\begin{proof}[Proof of \fullref{lemma:FP-in-action}]
First note that Perron--Frobenius theory guarantees an 
eigenvector $\llscalar_{\alpha}$ of $A(\g)$ with strictly positive entries. 
Moreover, the corresponding 
(so-called) Perron--Frobenius eigenvalue $\alpha$ is strictly positive.

Thus, after letting $\qpar$ to be a solution of the equation 
$-[2]_{\qpar}=\alpha$, we get a 
solution to \eqref{eq:eigenvector-problem}, 
i.e. a $\llscalar_{\alpha}$ without zero entries and a $\qpar\in\C-\{0\}$ 
such that \eqref{eq:eigenvector-problem} holds.
\end{proof}

\begin{remark}\label{remark:minus-signs}
Note that the proof of \fullref{lemma:FP-in-action} is constructive 
and can be used to produce \eqref{eq:barb2}-weightings for 
any given bipartite graph.

However, for $\ADE$ type graphs the Perron--Frobenius eigenvalue 
does not give $\qpar=\exp(\scalebox{.9}{$\nicefrac{$\pi i$}{$n$}$})$, but rather 
$\qpar=-\exp(\scalebox{.9}{$\nicefrac{$\pi i$}{$n$}$})$. 
For example, for $n=3$ and type $\typea{2}$ the Perron--Frobenius eigenvalue gives 
$\qpar=-\exp(\scalebox{.9}{$\nicefrac{$\pi i$}{$3$}$})$, for which $-[2]_{\qpar}=1$, 
rather than $\qpar=\exp(\scalebox{.9}{$\nicefrac{$\pi i$}{$3$}$})$, for which $-[2]_{\qpar}=-1$.
\end{remark}

Since the Kazhdan--Lusztig basis elements of $\mathrm{H}_{\infty}$ are categorified by 
indecomposable $1$-morphisms in $\Kar(\dihedralcatbig{\infty})$, see \cite[Theorem 5.29]{El1},
we get a stronger version of \fullref{corollary:w0-is-zero} by 
\fullref{theorem:main-theorem} later on. Namely, 
the matrices associated to $[\Functor{\Theta}_{\word}]$ have non-negative 
entries for all $\word\in\mathrm{W}_{\infty}$ 
if and only if $\g$ is not of type $\ADE$.
The ``if'' part of this statement
follows from \fullref{corollary:w0-is-zero} and the Chebyshev recursion; the ``only if'' part
needs \fullref{theorem:main-theorem}. (Note hereby that for $\g$ being of type 
$\ADE$ we
need the quantum parameter to be a root of unity to have a 
well-defined $2$-functor and \cite[Theorem 5.29]{El1} does not apply anymore. Hence, there is no
contradiction to \fullref{theorem:main-theorem}.)
\subsection{The infinite case}\label{subsec:infinite-case}

First, we need to check that the maps 
specified in \fullref{subsec:diacataction} 
are actually $\algG$-bimodule maps and give rise to 
natural transformations.

\begin{lemma}\label{lemma:bimodule-maps}
The maps from \eqref{eq:assignment-dotsA-a} 
to \eqref{eq:assignment-triA-b} are $\algG$-bimodule maps.
\end{lemma}

\begin{proof}
The map 
from \eqref{eq:assignment-dotsA-a} is the (scaled) multiplication map 
and thus, a $\algG$-bimodule map. 
The two maps from \eqref{eq:assignment-triA-a} and \eqref{eq:assignment-triA-b} 
clearly intertwine the left and right action of $\algG$ since their definition only involves 
the middle tensor factors in a non-trivial way. That the map 
from \eqref{eq:assignment-dotsA-b} is a $\algG$-bimodule map can be checked
via a case-by-case calculation. We illustrate this in an example. To this end, let us denote 
the {\color{sea}sea-green} version of it by $\nattrafo{d}^*$. 
Then, recalling \eqref{eq:partner}, we get
\begin{gather*}
\nattrafo{d}^*(\bbjj{j}|\bbii{i})
=
\nattrafo{d}^*(\bbjj{j}|\bbii{i}\pathm\bbii{i})
=
\bbjj{j}|\bbii{i}\pathm\nattrafo{d}^*(\bbii{i})
=
\lscalar_{\bbii{i}}\cdot
\bbjj{j}|\bbii{i}\pathm 
(\bbii{i}\otimes\bbii{i}|\bbii{i}+\bbii{i}|\bbii{i} \otimes \bbii{i})
=
\lscalar_{\bbii{i}}\cdot
\bbjj{j}|\bbii{i}
\otimes
\bbii{i}|\bbii{i},
\\
\nattrafo{d}^*(\bbjj{j}|\bbii{i})
=
\nattrafo{d}^*(\bbjj{j}\pathm\bbjj{j}|\bbii{i})
=
\nattrafo{d}^*(\bbjj{j})\pathm\bbjj{j}|\bbii{i}
=
{\textstyle\oplus_{\bbii{i}\gconnect{}\bbjj{j}^{\prime}}}\,\lscalar_{\bbii{i}}\cdot (\bbjj{j}^{\prime}|\bbii{i}\otimes\bbii{i}|\bbjj{j}^{\prime})
\pathm\bbjj{j}|\bbii{i}
\overset{\eqref{eq:partner}}{=}
\lscalar_{\bbii{i}}\cdot
\bbjj{j}|\bbii{i}
\otimes
\bbii{i}\vert\bbii{i}.
\end{gather*}
The remaining cases can be checked verbatim.
\end{proof}

The next lemmas follow directly from the definitions 
respectively via direct computation, and the proofs are omitted.

\begin{lemman}\label{lemma:nat-trafo}
The maps from \eqref{eq:assignment-dotsA-a} 
to \eqref{eq:assignment-triA-b} induce $2$-morphisms in $\Endff(\cellcatG)$. Moreover, 
the extension of the local assignment to arbitrary Soergel diagrams 
is consistent with the $2$-structure of $\Endff(\cellcatG)$.\qedmake
\end{lemman} 

\begin{lemman}\label{lemma:ADE-weights}
The weightings from \fullref{definition:these-numbers} 
are \eqref{eq:barb2}-weightings.\qedmake
\end{lemman} 

\begin{proof}[Proof of \fullref{theorem:main-theorem}, part (a)]
We first note that, by 
\fullref{lemma:projective}, \fullref{lemma:bimodule-maps} and \fullref{lemma:nat-trafo}, 
the weak $2$-functor $\functorGno$, if well-defined, is between 
the stated $2$-categories.

Furthermore, if $\functorGno$ is well-defined, then 
it extends uniquely to the Karoubi envelope by 
e.g. \cite[Proposition 6.5.9]{Bor1}, and the corresponding 
diagram will commute by \fullref{proposition:action-infty} 
(note that $\GG{\cellcatG}\cong\cellrepgrg$).
Moreover, $\functorGno$ -- if well-defined -- clearly preserves all additional structures. 
Thus, it remains to check that $\functorGno$ is well-defined which amounts 
to checking that $\functorGno$ preserves the relations of $\dihedralcatbig{}$.

We start with the far-commutativity \eqref{eq:far-comm}, which is 
just the interchange law (il) in $\Endff(\cellcatG)$. Denote 
by $\natid$, $\nattrafo{f}$ and $\nattrafo{g}$ the images 
under $\functorGno$
of the Soergel diagrams in question. Then (by our conventions 
how to apply $\functorGno$ to arbitrary Soergel diagrams):
\begin{gather*}
\functorGno\left(
\begin{tikzpicture}[anchorbase, scale=.3]
	\draw [very thick, orchid] (-1,-2) to (-1,-1.25);
	\draw [very thick, orchid] (-1,1.25) to (-1,5);
	\draw [very thick, orchid] (-3,-2) to (-3,-1.25);
	\draw [very thick, orchid] (-3,1.25) to (-3,5);
	\draw [very thick, orchid] (1,-2) to (1,1.75);
	\draw [very thick, orchid] (1,4.25) to (1,5);
	\draw [very thick, orchid] (3,-2) to (3,1.75);
	\draw [very thick, orchid] (3,4.25) to (3,5);
	\draw [thick] (-3.25,-1.25) rectangle (-.75,1.25);
	\draw [thick] (.75,1.75) rectangle (3.25,4.25);
	\node at (-3,-2.35) {\tiny $\word_{l}$};
	\node at (-3,5.55) {\tiny $\word_{l^{\prime}}^{\prime}$};
	\node at (-1,-2.35) {\tiny $\word_{k{+}1}$};
	\node at (-1,5.55) {\tiny $\word_{k^{\prime}{+}1}^{\prime}$};
	\node at (1,-2.35) {\tiny $\word_{k}$};
	\node at (1,5.55) {\tiny $\word_{k^{\prime}}^{\prime}$};
	\node at (3,-2.35) {\tiny $\word_{1}$};
	\node at (3,5.55) {\tiny $\word_{1}^{\prime}$};
	\node at (-2,-1.625) {\,$\cdots$};
	\node at (2,-1.625) {\,$\cdots$};
	\node at (-2,4.625) {\,$\cdots$};
	\node at (2,4.625) {\,$\cdots$};
	\draw [dotted] (-3.25,1.5) to (3.25,1.5);
	\draw [dotted] (0,5) to (0,-2);
	\fill [white] (0,1.5) circle (.35cm);
	\fill [white] (0,3) circle (.35cm);
	\fill [white] (0,0) circle (.35cm);
	\node at (-2,0) {\tiny $g$};
	\node at (2,3) {\tiny $f$};
	\node at (-2,3) {\tiny $\iddia$};
	\node at (2,0) {\tiny $\iddia$};
	\node at (0,0) {\tiny $\hcomp$};
	\node at (0,3) {\tiny $\hcomp$};
	\node at (0,1.5) {\tiny $\vcomp$};
\end{tikzpicture}
\right)
=
(\natid\hcomp\nattrafo{f})\circ(\nattrafo{g}\hcomp\natid)
\stackrel{\mathrm{il}}{=}
(\natid\circ\nattrafo{g})\hcomp(\nattrafo{f}\circ\natid)\\
=(\nattrafo{g}\circ\natid)\hcomp(\natid\circ\nattrafo{f})
\stackrel{\mathrm{il}}{=}(\nattrafo{g}\hcomp\natid)\circ(\natid\hcomp\nattrafo{f})
=
\functorGno\left(
\begin{tikzpicture}[anchorbase, scale=.3]
	\draw [very thick, orchid] (-1,-2) to (-1,1.75);
	\draw [very thick, orchid] (-1,4.25) to (-1,5);
	\draw [very thick, orchid] (-3,-2) to (-3,1.75);
	\draw [very thick, orchid] (-3,4.25) to (-3,5);
	\draw [very thick, orchid] (1,-2) to (1,-1.25);
	\draw [very thick, orchid] (1,1.25) to (1,5);
	\draw [very thick, orchid] (3,-2) to (3,-1.25);
	\draw [very thick, orchid] (3,1.25) to (3,5);
	\draw [thick] (-3.25,1.75) rectangle (-.75,4.25);
	\draw [thick] (.75,-1.25) rectangle (3.25,1.25);
	\node at (-3,-2.35) {\tiny $\word_{l}$};
	\node at (-3,5.55) {\tiny $\word_{l^{\prime}}^{\prime}$};
	\node at (-1,-2.35) {\tiny $\word_{k{+}1}$};
	\node at (-1,5.55) {\tiny $\word_{k^{\prime}{+}1}^{\prime}$};
	\node at (1,-2.35) {\tiny $\word_{k}$};
	\node at (1,5.55) {\tiny $\word_{k^{\prime}}^{\prime}$};
	\node at (3,-2.35) {\tiny $\word_{1}$};
	\node at (3,5.55) {\tiny $\word_{1}^{\prime}$};
	\node at (-2,-1.625) {\,$\cdots$};
	\node at (2,-1.625) {\,$\cdots$};
	\node at (-2,4.625) {\,$\cdots$};
	\node at (2,4.625) {\,$\cdots$};
	\draw [dotted] (-3.25,1.5) to (3.25,1.5);
	\draw [dotted] (0,5) to (0,-2);
	\fill [white] (0,1.5) circle (.35cm);
	\fill [white] (0,3) circle (.35cm);
	\fill [white] (0,0) circle (.35cm);
	\node at (-2,3) {\tiny $g$};
	\node at (2,0) {\tiny $f$};
	\node at (-2,0) {\tiny $\iddia$};
	\node at (2,3) {\tiny $\iddia$};
	\node at (0,0) {\tiny $\hcomp$};
	\node at (0,3) {\tiny $\hcomp$};
	\node at (0,1.5) {\tiny $\vcomp$};
\end{tikzpicture}
\right).
\end{gather*}

Next, we check that the other relations hold. 
There are several cases depending on whether $\bbk{i}$ is in 
$\graphs$ or $\grapht$, as well as on the color of the involved 
Soergel diagrams. We do some of them and leave the other 
(completely similar) cases to the reader. We also omit 
to indicate the sources and targets of the maps. 
As usual, we write 
$x_{\bbii{i}}\in P_\bbii{i}$, ${}_\bbii{i}y\in{}_\bbii{i}P$ and 
${}_\bbii{i}z_\bbii{i}\in{}_\bbii{i}P_\bbii{i}$. We also 
calculate the assignments on the corresponding direct summands only.
\medskip

\noindent \textit{The Frobenius relation \eqref{eq:frob1}.}  
We get the $\algG$-bimodule maps
\begin{gather*}
\begin{aligned}
\begin{tikzpicture}[anchorbase, scale=.3]
	\draw [very thick, sea] (0,-2) to (2,-1) to (4,-2);
	\draw [very thick, sea] (0,2) to (2,1) to (4,2);
	\draw [very thick, sea] (2,-1) to (2,1);
	\node at (0,-2.35) {\tiny $\dicatsgen$};
	\node at (0,2.35) {\tiny $\dicatsgen$};
	\node at (4,-2.35) {\tiny $\dicatsgen$};
	\node at (4,2.35) {\tiny $\dicatsgen$};
\end{tikzpicture}
&\stackrel{\functorGno}{\rightsquigarrow}
\begin{cases}
x_{\bbii{i}}\otimes \bbii{i} \otimes {}_\bbii{i}y & \mapsto 
0,\\
x_{\bbii{i}}\otimes \bbii{i}|\bbii{i} \otimes {}_\bbii{i}y & \mapsto
\lscalar_{\bbii{i}}^{-1}\cdot 
x_{\bbii{i}}\otimes \bbii{i} \otimes {}_\bbii{i}y,
\end{cases}
\end{aligned}
\end{gather*}
\begin{gather*}
\begin{aligned}
\begin{tikzpicture}[anchorbase, scale=.3]
	\draw [very thick, sea] (0,-2) to (0,0) to (-2,2);
	\draw [very thick, sea] (0,0) to (2,2);
	\draw [very thick, sea] (4,-2) to (4,2);
	\node at (0,-2.35) {\tiny $\dicatsgen$};
	\node at (-2,2.35) {\tiny $\dicatsgen$};
	\node at (2,2.35) {\tiny $\dicatsgen$};
	\node at (4,-2.35) {\tiny $\dicatsgen$};
	\node at (4,2.35) {\tiny $\dicatsgen$};
\end{tikzpicture}
&\stackrel{\functorGno}{\rightsquigarrow}
x_\bbii{i}\otimes {}_\bbii{i}z_\bbii{i}\otimes {}_\bbii{i}y\mapsto
x_\bbii{i}\otimes \bbii{i}\otimes {}_\bbii{i}z_\bbii{i}\otimes {}_\bbii{i}y,
\end{aligned}
\end{gather*}
\begin{gather*}
\begin{aligned}
\begin{tikzpicture}[anchorbase, scale=.3]
	\draw [very thick, sea] (0,2) to (0,0) to (-2,-2);
	\draw [very thick, sea] (0,0) to (2,-2);
	\draw [very thick, sea] (-4,-2) to (-4,2);
	\node at (0,2.35) {\tiny $\dicatsgen$};
	\node at (-2,-2.35) {\tiny $\dicatsgen$};
	\node at (2,-2.35) {\tiny $\dicatsgen$};
	\node at (-4,-2.35) {\tiny $\dicatsgen$};
	\node at (-4,2.35) {\tiny $\dicatsgen$};
\end{tikzpicture}
&\stackrel{\functorGno}{\rightsquigarrow}
\begin{cases}
x_\bbii{i}\otimes {}_\bbii{i}z_\bbii{i}\otimes \bbii{i}\otimes {}_\bbii{i}y
&\mapsto 
0,\\
x_\bbii{i}\otimes {}_\bbii{i}z_\bbii{i}\otimes \bbii{i}|\bbii{i}\otimes {}_\bbii{i}y
&\mapsto 
\lscalar_{\bbii{i}}^{-1}\cdot
x_\bbii{i}\otimes {}_\bbii{i}z_\bbii{i} \otimes {}_\bbii{i}y,
\end{cases}
\end{aligned}
\end{gather*}
\begin{gather*}
\begin{aligned}
\begin{tikzpicture}[anchorbase, scale=.3]
	\draw [very thick, sea] (0,-2) to (0,0) to (-2,2);
	\draw [very thick, sea] (0,0) to (2,2);
	\draw [very thick, sea] (-4,-2) to (-4,2);
	\node at (0,-2.35) {\tiny $\dicatsgen$};
	\node at (-2,2.35) {\tiny $\dicatsgen$};
	\node at (2,2.35) {\tiny $\dicatsgen$};
	\node at (-4,-2.35) {\tiny $\dicatsgen$};
	\node at (-4,2.35) {\tiny $\dicatsgen$};
\end{tikzpicture}
&\stackrel{\functorGno}{\rightsquigarrow}
x_\bbii{i}\otimes {}_\bbii{i}z_\bbii{i}\otimes {}_\bbii{i}y\mapsto
x_\bbii{i}\otimes {}_\bbii{i}z_\bbii{i}\otimes \bbii{i}\otimes {}_\bbii{i}y,
\end{aligned}
\end{gather*}
\begin{gather*}
\begin{aligned}
\begin{tikzpicture}[anchorbase, scale=.3]
	\draw [very thick, sea] (0,2) to (0,0) to (-2,-2);
	\draw [very thick, sea] (0,0) to (2,-2);
	\draw [very thick, sea] (4,-2) to (4,2);
	\node at (0,2.35) {\tiny $\dicatsgen$};
	\node at (-2,-2.35) {\tiny $\dicatsgen$};
	\node at (2,-2.35) {\tiny $\dicatsgen$};
	\node at (4,-2.35) {\tiny $\dicatsgen$};
	\node at (4,2.35) {\tiny $\dicatsgen$};
\end{tikzpicture}
&\stackrel{\functorGno}{\rightsquigarrow}
\begin{cases}
x_\bbii{i}\otimes \bbii{i}\otimes {}_\bbii{i}z_\bbii{i}\otimes {}_\bbii{i}y
&\mapsto 
0,\\
x_\bbii{i}\otimes \bbii{i}|\bbii{i}\otimes {}_\bbii{i}z_\bbii{i}\otimes {}_\bbii{i}y
&\mapsto 
\lscalar_{\bbii{i}}^{-1}\cdot
x_\bbii{i}\otimes {}_\bbii{i}z_\bbii{i} \otimes {}_\bbii{i}y.
\end{cases}
\end{aligned}
\end{gather*}
These compose as claimed.
\medskip

\noindent \textit{The Frobenius relation \eqref{eq:frob2}.}
We get (keeping \eqref{eq:partner} in mind which kills a lot of terms):
\begin{gather*}
\begin{tikzpicture}[anchorbase, scale=.3]
	\draw [very thick, sea] (-2,2) to (-2,-2);
	\draw [very thick, sea] (0,2) to (0,0);
	\node at (0,2.35) {\tiny $\dicatsgen$};
	\node at (0,-2.35) {\tiny $\phantom{a}$};
	\node at (-2,2.35) {\tiny $\dicatsgen$};
	\node at (-2,-2.35) {\tiny $\dicatsgen$};
	\node at (0,0) {\Large $\bullets$};
\end{tikzpicture}
\stackrel{\functorGno}{\rightsquigarrow}
x_\bbii{i}\otimes {}_\bbii{i}y
\mapsto
\begin{aligned}
& \lscalar_{\bbii{i}}\cdot
\left(x_\bbii{i}\otimes {}_\bbii{i}y\pathm \bbii{i}\otimes \bbii{i}|\bbii{i}
+ x_\bbii{i}\otimes {}_\bbii{i}y\pathm \bbii{i}|\bbii{i}\otimes \bbii{i}\right)
\\
&+{\textstyle\sum_{\bbii{i}\gconnect{}\bbjj{j}}}\,
(\lscalar_{\bbii{i}}\cdot
x_\bbii{i}\otimes {}_\bbii{i}y\pathm \bbjj{j}|\bbii{i}\otimes\bbii{i}|\bbjj{j}),
\end{aligned}
\\
\begin{tikzpicture}[anchorbase, scale=.3]
	\draw [very thick, sea] (2,2) to (2,-2);
	\draw [very thick, sea] (0,2) to (0,0);
	\node at (0,2.35) {\tiny $\dicatsgen$};
	\node at (0,-2.35) {\tiny $\phantom{a}$};
	\node at (2,2.35) {\tiny $\dicatsgen$};
	\node at (2,-2.35) {\tiny $\dicatsgen$};
	\node at (0,0) {\Large $\bullets$};
\end{tikzpicture}
\stackrel{\functorGno}{\rightsquigarrow}
x_\bbii{i}\otimes {}_\bbii{i}y\mapsto 
\begin{aligned}
&\lscalar_{\bbii{i}}\cdot
\left(\bbii{i}\otimes \bbii{i}|\bbii{i}\pathm 
x_\bbii{i}\otimes {}_\bbii{i}y
+ \bbii{i}|\bbii{i}\otimes \bbii{i}\pathm x_\bbii{i}\otimes {}_\bbii{i}y\right)
\\
&+{\textstyle\sum_{\bbii{i}\gconnect{}\bbjj{j}}}\,
(\lscalar_{\bbii{i}}\cdot
\bbjj{j}|\bbii{i}\otimes\bbii{i}|\bbjj{j}\pathm x_\bbii{i}\otimes {}_\bbii{i}y).
\end{aligned}
\end{gather*}
The remaining two assignments were 
already calculated in \fullref{example:2functor-assignment}.
These compose with the assignments 
for the trivalent 
vertices from \eqref{eq:assignment-triA-a} and \eqref{eq:assignment-triA-b} 
as claimed. 
(This relies on \eqref{eq:partner}.) 
Let us stress that the \eqref{eq:barb2}-weighting scalars cancel, since 
the assignments from \eqref{eq:assignment-dotsA-b} 
and \eqref{eq:assignment-triA-b} are multiplied by inverse scalars.
\medskip

\noindent\textit{The needle relation \eqref{eq:needle}.} 
Directly by composing \eqref{eq:assignment-triA-b} and \eqref{eq:assignment-triA-a} 
(in this order).
\medskip

\noindent\textit{The first barbell forcing 
relation \eqref{eq:barb1}.} 
Using \eqref{eq:assignment-dotsA-a} and \eqref{eq:assignment-dotsA-b}, we get:
\begin{gather}\label{eq:floating-barbell1}
\begin{tikzpicture}[anchorbase, scale=.3]
	\draw [very thick, sea] (-2,-1) to (-2,1);
	\node at (-2,-1) {\Large $\bullets$};
	\node at (-2,1) {\Large $\bullets$};
\end{tikzpicture}
\stackrel{\functorGno}{\rightsquigarrow}
\begin{cases}
\bbii{i} \mapsto \lscalar_{\bbii{i}}\cdot 2\cdot
\bbii{i}|\bbii{i},
\\
\bbjj{j} \mapsto 
-[2]_{\qpar}\cdot\lscalar_{\bbjj{j}}\cdot
\bbjj{j}|\bbjj{j},
\\
\bbii{i}|\bbii{i}\text{ and }\bbii{i}|\bbjj{j}\text{ and }\bbjj{j}|\bbii{i}\text{ and }\bbjj{j}|\bbjj{j}
\mapsto 0,
\end{cases}
\end{gather}
where we used \eqref{eq:eigenvector-problem-2} to replace
${\textstyle\sum_{\bbii{i}\gconnect{}\bbjj{j}}}\,\lscalar_{\bbii{i}}$ 
by $-[2]_{\qpar}\cdot\lscalar_{\bbjj{j}}$.
Next, we have
\begin{gather}\label{eq:need-some-scalars-later1}
\begin{tikzpicture}[anchorbase, scale=.3]
	\draw [very thick, sea] (0,-2) to (0,-1);
	\draw [very thick, sea] (0,1) to (0,2);
	\node at (0,-2.35) {\tiny $\dicatsgen$};
	\node at (0,2.35) {\tiny $\dicatsgen$};
	\node at (0,-1) {\Large $\bullets$};
	\node at (0,1) {\Large $\bullets$};
\end{tikzpicture}
\stackrel{\functorGno}{\rightsquigarrow}
x_\bbii{i}\otimes {}_\bbii{i}y
\mapsto
\begin{aligned}
&\lscalar_{\bbii{i}}\cdot
\left(
(x_\bbii{i}\pathm\bbii{i})
\otimes
(\bbii{i}\vert\bbii{i}\pathm{}_\bbii{i}y)
+
(x_\bbii{i}\pathm\bbii{i}\vert\bbii{i})
\otimes
(\bbii{i}\pathm{}_\bbii{i}y)\right)
\\
&+{\textstyle\sum_{\bbii{i}\gconnect{}\bbjj{j}}}\,
(\lscalar_{\bbii{i}}\cdot
(x_\bbii{i}\pathm\bbjj{j}|\bbii{i})
\otimes
(\bbii{i}|\bbjj{j}\pathm{}_\bbii{i}y)).
\end{aligned}
\end{gather}
Therefore -- by \eqref{eq:floating-barbell1} -- we have
\begin{gather}\label{eq:need-some-scalars-later-missing}
\begin{aligned}
\begin{tikzpicture}[anchorbase, scale=.3]
	\draw [very thick, sea] (0,-2) to (0,2);
	\draw [very thick, sea] (-2,-1) to (-2,1);
	\node at (0,-2.35) {\tiny $\dicatsgen$};
	\node at (0,2.35) {\tiny $\dicatsgen$};
	\node at (-2,-1) {\Large $\bullets$};
	\node at (-2,1) {\Large $\bullets$};
\end{tikzpicture}
&\stackrel{\functorGno}{\rightsquigarrow}
x_\bbii{i}\otimes {}_\bbii{i}y
\mapsto
\begin{aligned}
& \lscalar_{\bbii{i}}\cdot
\left(2\cdot (\bbii{i}|\bbii{i}\pathm x_\bbii{i})\otimes {}_\bbii{i}y\right)
\\
&-[2]_{\qpar}\cdot
\lscalar_{\bbjj{j}}\cdot(
(\bbjj{j}|\bbjj{j}\pathm x_\bbii{i})\otimes {}_\bbii{i}y),
\end{aligned}
\\
\begin{tikzpicture}[anchorbase, scale=.3]
	\draw [very thick, sea] (0,-2) to (0,2);
	\draw [very thick, sea] (2,-1) to (2,1);
	\node at (0,-2.35) {\tiny $\dicatsgen$};
	\node at (0,2.35) {\tiny $\dicatsgen$};
	\node at (2,-1) {\Large $\bullets$};
	\node at (2,1) {\Large $\bullets$};
\end{tikzpicture}
&\stackrel{\functorGno}{\rightsquigarrow}
x_\bbii{i}\otimes {}_\bbii{i}y
\mapsto 
\begin{aligned}
&\lscalar_{\bbii{i}}\cdot\left(2\cdot x_\bbii{i}\otimes ({}_\bbii{i}y\pathm\bbii{i}|\bbii{i})\right)
\\
&-[2]_{\qpar}\cdot
\lscalar_{\bbjj{j}}\cdot(
x_\bbii{i}\otimes
({}_\bbii{i}y\pathm\bbjj{j}|\bbjj{j})).
\end{aligned}
\end{aligned}
\end{gather}
Note that, for all $x_\bbii{i}\in P_{\bbii{i}}$ 
and all ${}_\bbii{i}y\in {}_{\bbii{i}}P$, we have 
$\bbjj{j}|\bbjj{j}\pathm x_\bbii{i}=0={}_\bbii{i}y\pathm\bbjj{j}|\bbjj{j}$. This holds since 
all paths in $P_{\bbii{i}}$ start in 
$\bbii{i}$ and all paths in 
${}_{\bbii{i}}P$ end in $\bbii{i}$, and 
$\bbjj{j}|\bbjj{j}$ composes only 
with $\bbjj{j}$ to a non-zero element.
Moreover, we also have 
$x_\bbii{i}\pathm\bbjj{j}|\bbii{i}=0=\bbii{i}|\bbjj{j}\pathm{}_\bbii{i}y$. 
In total, all the terms involving {\color{tomato} tomato t} vanish.
For the rest one checks 
directly on all
possible $x_\bbii{i},{}_\bbii{i}y$ 
that the claimed equation holds. 
For example, in case $x_\bbii{i}=\bbii{i}$ and ${}_\bbii{i}y=\bbii{i}|\bbii{i}$ 
one gets $\lscalar_{\bbii{i}}\cdot 2\cdot\bbii{i}|\bbii{i}\otimes\bbii{i}|\bbii{i}$ 
from the sum of the two maps in 
\eqref{eq:need-some-scalars-later-missing}, and $\lscalar_{\bbii{i}}\cdot\bbii{i}|\bbii{i}\otimes\bbii{i}|\bbii{i}$ 
from the map in \eqref{eq:need-some-scalars-later1}.
\medskip

\noindent \textit{The second barbell forcing 
relation \eqref{eq:barb2}.}
Now the scaling will be crucial. 
  
Similarly to \eqref{eq:floating-barbell1} we get
\begin{gather}\label{eq:floating-barbell2}
\begin{tikzpicture}[anchorbase, scale=.3]
	\draw [very thick, tomato] (-2,-1) to (-2,1);
	\node at (-2,-1) {\Large $\bullett$};
	\node at (-2,1) {\Large $\bullett$};
\end{tikzpicture}
\stackrel{\functorGno}{\rightsquigarrow}
\begin{cases}
\bbjj{j} \mapsto \lscalar_{\bbjj{j}}\cdot 2\cdot
\bbjj{j}|\bbjj{j},
\\
\bbii{i} \mapsto -[2]_{\qpar}\cdot\lscalar_{\bbii{i}}\cdot
\bbii{i}|\bbii{i},
\\
\bbjj{j}|\bbjj{j}\text{ and }\bbjj{j}|\bbii{i}\text{ and }\bbii{i}|\bbjj{j}\text{ and }\bbii{i}|\bbii{i}
\mapsto 0,
\end{cases}
\end{gather}
The two not yet computed assignments are
\begin{gather}\label{eq:need-some-scalars-later2}
\begin{aligned}
\begin{tikzpicture}[anchorbase, scale=.3]
	\draw [very thick, sea] (0,-2) to (0,2);
	\draw [very thick, tomato] (-2,-1) to (-2,1);
	\node at (0,-2.35) {\tiny $\dicatsgen$};
	\node at (0,2.35) {\tiny $\dicatsgen$};
	\node at (-2,-1) {\Large $\bullett$};
	\node at (-2,1) {\Large $\bullett$};
\end{tikzpicture}
&\stackrel{\functorGno}{\rightsquigarrow}
x_\bbii{i}\otimes{}_\bbii{i}y
\mapsto
\begin{aligned}
& \lscalar_{\bbjj{j}}\cdot
\left(2\cdot (\bbjj{j}|\bbjj{j}\pathm x_\bbii{i})\otimes {}_\bbii{i}y\right)
\\
&-[2]_{\qpar}\cdot\lscalar_{\bbii{i}}\cdot(
(\bbii{i}|\bbii{i}\pathm x_\bbii{i})\otimes {}_\bbii{i}y),
\end{aligned}
\\
\begin{tikzpicture}[anchorbase, scale=.3]
	\draw [very thick, sea] (0,-2) to (0,2);
	\draw [very thick, tomato] (2,-1) to (2,1);
	\node at (0,-2.35) {\tiny $\dicatsgen$};
	\node at (0,2.35) {\tiny $\dicatsgen$};
	\node at (2,-1) {\Large $\bullett$};
	\node at (2,1) {\Large $\bullett$};
\end{tikzpicture}
&\stackrel{\functorGno}{\rightsquigarrow}
x_\bbii{i}\otimes{}_\bbii{i}y
\mapsto 
\begin{aligned}
&\lscalar_{\bbjj{j}}\cdot\left(
2\cdot x_\bbii{i}\otimes ({}_\bbii{i}y\pathm\bbjj{j}|\bbjj{j})\right)
\\
&-[2]_{\qpar}\cdot\lscalar_{\bbii{i}}\cdot(
x_\bbii{i}\otimes
({}_\bbii{i}y\pathm\bbii{i}|\bbii{i})).
\end{aligned}
\end{aligned}
\end{gather}
As before we have $\bbjj{j}|\bbjj{j}\pathm x_\bbii{i}=0={}_\bbii{i}y\pathm\bbjj{j}|\bbjj{j}$ 
and the factors with the $2$ die. 
That the claimed equation holds can be verified by a case-by-case check. 
(Hereby we stress that \fullref{lemma:ADE-weights} comes crucially into the game 
since it ensures that the scalars add up as they should.)

All together this shows that $\functorGno$ is well-defined.
\end{proof}

Note that in the proof above we never used that $\g$ is of $\ADE$ type, but 
rather that we have a \eqref{eq:barb2}-weighting. Thus, the above goes through -- 
mutatis mutandis -- for any bipartite graph with a \eqref{eq:barb2}-weighting.

Now we switch to $\C$ since we use some 
spectral theory below.

\begin{proof}[Proof of \fullref{lemma:unique-maps}]
Assume that we have fixed $\g$ of $\ADE$ type with 
the corresponding root of unity $\qpar$, its 
double-quiver algebra $\algG$ and its module category $\cellcatG$. 
Assume also that we have a (well-defined) weak $2$-functor 
$\functorGs\colon\dihedralcatbig{\infty}\to\Endff(\cellcatG)$ which on 
$1$-morphisms is equal to the weak $2$-functor 
$\functorGno$ from \fullref{subsec:diacataction}.

Next, recall the bases of the left 
and the right $\algG$-modules 
$P_{\bbk{i}}$ and ${}_{\bbk{i}}P$, as exemplified in \eqref{eq:projectives}. 
Using these bases, one easily sees that the image of the first 
of the two trivalent 
vertices has to be given by (using the notational conventions 
from above)
\[
\begin{tikzpicture}[anchorbase, scale=.3]
	\draw [very thick, sea] (0,-2) to (0,0) to (-2,2);
	\draw [very thick, sea] (0,0) to (2,2);
	\node at (0,-2.35) {\tiny $\dicatsgen$};
	\node at (-2,2.35) {\tiny $\dicatsgen$};
	\node at (2,2.35) {\tiny $\dicatsgen$};
%
\end{tikzpicture}
\stackrel{\functorGs}{\rightsquigarrow}
x_\bbii{i}\otimes {}_\bbii{i}y\mapsto  
\lscalar_\bbii{i}(\seasplit)\cdot x_\bbii{i}\otimes \bbii{i}\otimes {}_\bbii{i}y,
\]
for some scalar $\lscalar_\bbii{i}(\seasplit)\in\C$. This follows directly for  
degree reasons and the fact that the assignment should be a $\algG$-bimodule map.

For the same reasons, the first of the two dots has to be mapped to
\[
\begin{tikzpicture}[anchorbase, scale=.3]
	\draw [very thick, sea] (0,-2) to (0,0);
	\node at (0,-2.35) {\tiny $\dicatsgen$};
	\node at (0,2.35) {\tiny $\phantom{a}$};
	\node at (0,0) {\Large $\bullets$};
%
\end{tikzpicture}
\stackrel{\functorGs}{\rightsquigarrow}
x_\bbii{i} \otimes {}_\bbii{i}y
\mapsto
\lscalar_\bbii{i}(\seadotup)\cdot x_\bbii{i}\pathm {}_\bbii{i}y,
\]
for some scalar $\lscalar_\bbii{i}(\seadotup)\in\C$. Now, under the assumption that 
$\functorGs$ is well-defined, the 
second Frobenius relation \eqref{eq:frob2} holds, which implies that 
both scalars have to be invertible and satisfy 
$\lscalar_\bbii{i}(\seadotup)=(\lscalar_\bbii{i}(\seasplit))^{-1}$.

Using the same reasoning,  
(and \eqref{eq:needle}), we see that the image of the other two generators has to be given by
\begin{gather*}
\begin{tikzpicture}[anchorbase, scale=.3]
	\draw [very thick, sea] (0,2) to (0,0) to (-2,-2);
	\draw [very thick, sea] (0,0) to (2,-2);
	\node at (0,2.35) {\tiny $\dicatsgen$};
	\node at (-2,-2.35) {\tiny $\dicatsgen$};
	\node at (2,-2.35) {\tiny $\dicatsgen$};
%
\end{tikzpicture}
\stackrel{\functorGs}{\rightsquigarrow}
\begin{cases}
x_\bbii{i}\otimes \bbii{i} \otimes {}_\bbii{i}y & 
\mapsto 
0,\\
x_\bbii{i}\otimes \bbii{i}|\bbii{i} \otimes {}_\bbii{i}y & 
\mapsto 
\lscalar_\bbii{i}(\seamerge)\cdot x_\bbii{i}\otimes {}_\bbii{i}y.
\end{cases}
\\
\begin{tikzpicture}[anchorbase, scale=.3]
	\draw [very thick, sea] (0,2) to (0,0);
	\node at (0,2.35) {\tiny $\dicatsgen$};
	\node at (0,-2.35) {\tiny $\phantom{a}$};
	\node at (0,0) {\Large $\bullets$};
	\node at (1,0) {\tiny $\bbii{i}$};
\end{tikzpicture}
\stackrel{\functorGs}{\rightsquigarrow}
\begin{cases}
&\bbii{i}\mapsto \lscalar_\bbii{i}(\seadotdown)\cdot
\left(\bbii{i}\otimes\bbii{i}|\bbii{i}+\bbii{i}|\bbii{i} \otimes \bbii{i}\right),\\
&\bbjj{j}\mapsto
{\textstyle\sum_{\bbii{i}\gconnect{}\bbjj{j}}}\,
(\lscalar_\bbii{i}(\seadotdown)\cdot\bbjj{j}|\bbii{i}\otimes \bbii{i}|\bbjj{j}),
\end{cases}
\end{gather*}
for some invertible scalars 
$\lscalar_\bbii{i}(\seamerge),\lscalar_\bbii{i}(\seadotdown)\in\C$ satisfying 
$\lscalar_\bbii{i}(\seadotdown)=(\lscalar_\bbii{i}(\seamerge))^{-1}$.

Similarly for {\color{tomato}tomato} $\dihtgen$, where we get four 
invertible complex scalars 
which satisfy 
$\lscalar_\bbjj{j}(\todotup)=(\lscalar_\bbjj{j}(\tosplit))^{-1}$ 
and $\lscalar_\bbjj{j}(\todotdown)=(\lscalar_\bbjj{j}(\tomerge))^{-1}$.

It remains to check that there is no choice for the weighting.

To this end, we write $\lscalar_\bbii{i}(\seabarb)=
\lscalar_\bbii{i}(\seadotup)\cdot\lscalar_\bbii{i}(\seadotdown)$ 
and $\lscalar_\bbjj{j}(\tobarb)=
\lscalar_\bbjj{j}(\todotup)\cdot\lscalar_\bbjj{j}(\todotdown)$.
As in the proof of \fullref{theorem:main-theorem}, part (a),
we get
\begin{gather*}
\begin{aligned}
\begin{tikzpicture}[anchorbase, scale=.3]
	\draw [very thick, sea] (-2,-1) to (-2,1);
	\node at (-2,-1) {\Large $\bullets$};
	\node at (-2,1) {\Large $\bullets$};
%
\end{tikzpicture}
&\stackrel{\functorGs}{\rightsquigarrow}
\begin{cases}
\bbii{i} \mapsto \lscalar_\bbii{i}(\seabarb)\cdot 2\cdot
\bbii{i}|\bbii{i},
\\
\bbjj{j} \mapsto 
({\textstyle\sum_{\bbii{i}\gconnect{}\bbjj{j}}}\, \lscalar_\bbii{i}(\seabarb))\cdot
\bbjj{j}|\bbjj{j},
\\
\bbii{i}|\bbii{i}\text{ and }\bbii{i}|\bbjj{j}\text{ and }\bbjj{j}|\bbii{i}\text{ and }\bbjj{j}|\bbjj{j}
\mapsto 0,
\end{cases}
\\
\begin{tikzpicture}[anchorbase, scale=.3]
	\draw [very thick, tomato] (-2,-1) to (-2,1);
	\node at (-2,-1) {\Large $\bullett$};
	\node at (-2,1) {\Large $\bullett$};
%
\end{tikzpicture}
&\stackrel{\functorGs}{\rightsquigarrow}
\begin{cases}
\bbjj{j} \mapsto \lscalar_\bbjj{j}(\tobarb)\cdot 2\cdot
\bbjj{j}|\bbjj{j},
\\
\bbii{i} \mapsto 
({\textstyle\sum_{\bbjj{j}\gconnect{}\bbii{i}}}\,\lscalar_\bbjj{j}(\tobarb))\cdot
\bbii{i}|\bbii{i},
\\
\bbjj{j}|\bbjj{j}\text{ and }\bbjj{j}|\bbii{i}\text{ and }\bbii{i}|\bbjj{j}\text{ and }\bbii{i}|\bbii{i}
\mapsto 0.
\end{cases}
\end{aligned}
\end{gather*}
This gives the scaled versions of \eqref{eq:need-some-scalars-later1} 
and \eqref{eq:need-some-scalars-later2}.

Using the assumption that $\functorGs$ preserves the second  
barbell forcing relation \eqref{eq:barb2}, we can relate 
the $\lscalar_\bbii{i}(\seabarb)$'s and 
the $\lscalar_\bbjj{j}(\tobarb)$'s. Since relation \eqref{eq:barb2} holds, we see that 
an equation of the form \eqref{eq:eigenvector-problem-2} 
must be satisfied for the $\lscalar_\bbii{i}(\seabarb)$'s 
and the $\lscalar_\bbjj{j}(\tobarb)$'s.
Spectral theory shows that the eigenspace of $A(\g)$ for the eigenvalue $-[2]_{\qpar}$ 
is one-dimensional for $\ADE$ type graphs 
and the corresponding roots of unity, cf. \eqref{eq:spectrum}. 
This means that our chosen \eqref{eq:barb2}-weighting 
from \fullref{definition:these-numbers} is unique up to a scalar
$\scalargttt$ with
\[
\scalargttt
=
\lscalar_\bbii{i}(\seabarb)
\cdot\lscalar_\bbii{i}^{-1}
=
\lscalar_\bbjj{j}(\tobarb)
\cdot\lscalar_\bbjj{j}^{-1},
\quad
\text{for all }\bbii{i},\bbjj{j}\in\g.
\]
Observe that this amounts to saying that 
the corresponding products do not depend on the vertices.

The first part of the claim now follows by letting 
$\scalargt$ be $\scalargttt$ divided by the value of 
the end dot $\lscalar_\bbii{i}(\seadotup)$ and 
$\scalargtt$ be $\scalargttt$ divided by the value of 
the start dot $\lscalar_\bbii{i}(\seadotdown)$
(for any vertex).

Clearly, any consistent rescaling of 
the $\algG$-bimodule maps associated to the generating 
$2$-morphisms of $\dihedralcatbig{\infty}$ gives rise 
to an equivalence of $2$-representations.
\end{proof}

Note that the one-dimensionality of 
the eigenspace of $A(\g)$ for the eigenvalue $-[2]_{\qpar}$ 
fails in general for bipartite graphs. Consequently, 
\fullref{lemma:unique-maps} is not always true for 
bipartite graphs with a given \eqref{eq:barb2}-weighting.
\subsection{The finite case}\label{subsec:finite-case}

We again use the setup from \fullref{subsec:diacataction}.

\begin{proof}[Proof of \fullref{theorem:main-theorem}, part (b)]

By statement (a) of \fullref{theorem:main-theorem} 
-- and the evident analog of \fullref{lemma:bimodule-maps} for the 
maps from \eqref{eq:assignment-2nvertex} --
it only remains to check three extra relations, i.e. 
the relations \eqref{eq:twocolor1} to \eqref{eq:twocolor3} 
need to be preserved under the functor $\functorADE$. The 
relations \eqref{eq:twocolor1} and \eqref{eq:twocolor3} are clearly preserved, and we need 
to prove that \eqref{eq:twocolor2} is preserved if and only if 
$\g$ is of $\ADE$ type.
\medskip

\noindent\textit{``Only if''}. This is clear by 
\fullref{proposition:action-finite}, since 
otherwise $\GG{\cellcatG}$ 
would inherit the structure of an $\mathrm{H}_n$-module.
\medskip

\noindent\textit{``If''}. 
Note that the $\JWst{k}$ are well-defined 
for $k\in\{0,\dots,n\}$, 
both when $\qpar$ is generic and when $\qpar$ is a complex, primitive $2n$th root of unity.

Let us assume that $\g$ is of $\ADE$ type. 
Below we show that $\functorADE(\JWst{n})=0$, when $\qpar$ is 
our fixed complex, primitive $2n$th root of unity. This implies, 
by definition, that \eqref{eq:twocolor2} is preserved under 
$\functorADE$.

We have (see e.g. \cite[Section 3]{Kh2} 
and \cite[Section 2.2]{El1})
\begin{equation}\label{eq:khov1}
[V_{l}] = {\textstyle\sum_{l=0}^{k}}\; d_{l{+}1}^{k{+}1}\cdot[V_1^{\otimes k}],\quad 0\leq k\leq l,
\end{equation}
in the Grothendieck ring of $\slmod$ (the finite-dimensional 
type $1$ modules of quantum $\mathfrak{sl}_2$) for $\qpar$ not a root 
of unity or generic. Here the $V_l$'s are the simple 
objects of highest weight $l$, and the $d_l^k$'s are as in \eqref{eq:numbers}. 

The same holds in the 
semisimplified quotient category $\slmods$, when $\qpar$ is a complex, primitive
$2n$th root of unity 
as long as $l\leq n$. (See e.g. \cite[Theorem 3.1]{An1} for a 
version that works for even roots of unity. 
See also e.g. \cite[Section 5]{Sa1} for the 
semisimplified quotient.) Recall also that $[V_n]=0$ holds in $\GG{\slmods}$.

Note that the sign of $d_{l{+}1}^{k{+}1}$ alternates 
as $k$ runs from $0$ to $l$. Khovanov \cite[Theorem 1]{Kh2} showed that 
the right-hand side of \eqref{eq:khov1} is the Euler characteristic of a complex
\begin{gather*}
C_l^\ast\colon C_l^0 \to C_l^2 \to \dots \to C_l^{\text{\tiny\nicefrac{$l$}{$2$}}}, \;\text{for }l\text{ even},\;
C_l^\ast\colon C_l^0 \to C_l^2 \to \dots \to C_l^{\text{\tiny\nicefrac{$l{-}1$}{$2$}}}, \;\text{for }l\text{ odd},
\end{gather*}
which is acyclic in positive cohomological degree and $H^0(C_l^\ast)\cong V_l$.  More explicitly, 
\[
C_l^k=\left(V_1^{\otimes (l-2k)}\right)^{\oplus d_{l{+}1}^{k{+}1}},
\]
and the differential is defined as a direct sum obtained by applying the evaluation map 
$V_1\otimes V_1 \to \C$,  
multiplied by a suitable sign, to each neighboring tensor pair 
$V_1\otimes V_1$ in each copy of $V_1^{\otimes (l-2k)}$. 

Note that Khovanov proves his theorem for $\qpar=1$ and 
mentions that the proof also works for generic $\qpar$.
However, his proof follows by induction from the fact 
that $V_1\otimes V_l \cong V_{l-1}\otimes V_{l+1}$, which 
is also true in our case (by a special case of 
the so-called linkage principle) as long as $0\leq l\leq n-1$.  

Since $\slmods$ is a semisimple abelian category, its bounded derived category is 
also semisimple \cite[Section III.3]{GM1} and 
equivalent to ${\textstyle \bigoplus_{\Z}}\slmods$ by the homology functor. 
Moreover, the bounded derived category is equivalent to the homotopy category 
of bounded complexes, because all objects are projective, see e.g. \cite[Theorem 10.4.8]{Wei1}. 
Thus, $C_l^\ast$ is equivalent to $V_l$ (seen as a complex concentrated in homological degree zero).

By \cite[Proposition 1.2]{El1}, this implies that, for any $0\leq l\leq n$, the object $(\overline{\dihugen}_l,\JWst{l})$ in $\Kar(\dihedralcat{})$ is homotopy equivalent to a complex 
$D_l^\ast(\dihugen)$
with  
\[
D_l^k(\dihugen)=\overline{\dihugen}_{l}^{\oplus d_{l{+}1}^{k{+}1}},
\quad\text{for }\dihugen=\dihsgen,\dihtgen.
\]

Since this holds in the finite and infinite case alike, the definitions of the complex $D_l^\ast(\dihugen)$ and the homotopy equivalence with $(\overline{\dihugen}_l,\JWst{l})$ do not use any $2n$-vertices. For the same reason, the defining properties of the homotopy equivalence follow from the relations in $\dihedralcat{}$ which do not involve $2n$-vertices.

Since we already proved that $\functorADE$ preserves all the relations which do not involve $2n$-vertices, we see 
that $\functorADE(D_n^\ast(\dihugen))$ is a complex and that it is homotopy equivalent to $\functorADE(\overline{\dihugen}_n,\JWst{n})$. The latter has Euler characteristic equal to zero, by \fullref{proposition:action-finite}. Therefore, $\functorADE(\JWst{n})=0$, which is 
what we had to prove.
\end{proof}

\begin{example}\label{example:JW-is-zero}
In principle one could also check by hand that $\functorADE(\JWst{n})=0$ holds. 
This is a daunting task in general, but let us do a small case which is quite illustrative. 
Fix the type $\typea{2}$ graph with vertices $\bbii{1}$ and 
$\bbjj{2}$, and the weighting is 
$\lscalar_{\bbii{1}}=[1]_{\qpar}=1$ and 
$\lscalar_{\bbjj{2}}=-[2]_{\qpar}=-1$. In this case $n=3$ and $\qpar$ is the 
complex, primitive $6$th root of unity 
$\qpar=\exp(\scalebox{.9}{$\nicefrac{$\pi i$}{$3$}$})$. After a small calculation 
one gets
\[
\begin{tikzpicture}[anchorbase, scale=.3]
	\draw [very thick, sea] (-2,-2) to (0,-.5) to (0,.5) to (-2,2);
	\draw [very thick, tomato] (0,-2) to (0,-1.25);
	\draw [very thick, tomato] (0,1.25) to (0,2);
	\draw [very thick, sea] (2,-2) to (0,-.5) to (0,.5) to (2,2);
	\node at (0,-1.25) {\Large $\bullett$};
	\node at (0,1.25) {\Large $\bullett$};
	\node at (-2,-2.4) {\tiny $\dicatsgen$};
	\node at (-2,2.325) {\tiny $\dicatsgen$};
	\node at (0,-2.35) {\tiny $\dicattgen$};
	\node at (0,2.35) {\tiny $\dicattgen$};
	\node at (2,-2.4) {\tiny $\dicatsgen$};
	\node at (2,2.325) {\tiny $\dicatsgen$};
\end{tikzpicture}
\stackrel{\functorADE}{\rightsquigarrow}
\lscalar_{\bbii{1}}^{-1}\lscalar_{\bbjj{2}}\cdot\idmap_X,
\quad
\begin{tikzpicture}[anchorbase, scale=.3]
	\draw [very thick, tomato] (-2,-2) to (0,-.5) to (0,.5) to (-2,2);
	\draw [very thick, sea] (0,-2) to (0,-1.25);
	\draw [very thick, sea] (0,1.25) to (0,2);
	\draw [very thick, tomato] (2,-2) to (0,-.5) to (0,.5) to (2,2);
	\node at (0,-1.25) {\Large $\bullets$};
	\node at (0,1.25) {\Large $\bullets$};
	\node at (-2,-2.35) {\tiny $\dicattgen$};
	\node at (-2,2.35) {\tiny $\dicattgen$};
	\node at (0,-2.4) {\tiny $\dicatsgen$};
	\node at (0,2.325) {\tiny $\dicatsgen$};
	\node at (2,-2.35) {\tiny $\dicattgen$};
	\node at (2,2.35) {\tiny $\dicattgen$};
\end{tikzpicture}
\stackrel{\functorADE}{\rightsquigarrow}
\lscalar_{\bbii{1}}\lscalar_{\bbjj{2}}^{-1}\cdot\idmap_Y.
\]
Hereby $X$ and $Y$ are the corresponding tensor products 
for the boundary sequences. Hence, with our weighting, we get $\functorGno_3(\JWs{3})=0$ 
and $\functorGno_3(\JWt{3})=0$, since the 
two summands in their expressions in \fullref{example:jw} cancel.
\end{example}

\begin{proof}[Proof of \fullref{lemma:unique-maps2}]
If a $2$-representation of 
$\dihedralcatbig{n}$ agrees with our $\functorGno$ on $1$-morphisms, then its decategorification 
has to be isomorphic to $[\functorGno]$. Therefore, it has 
to kill $[\boldsymbol{\Theta}_{\wnull}]$. But $\boldsymbol{\Theta}_{\wnull}$ corresponds to the 
idempotent defined as the image of the $2$-colored version of the ``Jones--Wenzl projector''
$\JWst{n}$, so that has to be sent to zero by the above. 
This implies that the $2n$-valent vertex has to be sent to zero as well by \cite[(6.16)]{El1}.
\end{proof}
\subsection{Classification of dihedral \texorpdfstring{$2$}{2}-representations}\label{subsec:class}

We work over $\C$ in \fullref{theorem:main-proposition}.

\begin{proof}[Proof of \fullref{theorem:main-proposition}]

We have to prove the statements (a), (b), (c) and (d).
\medskip

\noindent\textit{Proof of statement (a)}\,. By \fullref{proposition:locally-finitary},  
$\Kar(\dihedralcat{n})^{\star}$ and $\Kar(\dihedralcat{\infty})^{\star}$ 
are graded (locally) finitary, and thus, we see that we actually get 
graded finitary, weak $2$-representations in \fullref{theorem:main-theorem}.

Hence, we only need to show that $\functorGno$ is simple transitive. 
Transitivity follows from the connectivity of $\g$. For each $P_{\bbk{i}}$ and $P_{\bbk{j}}$, 
apply an alternating sequence $\Functor{\Theta}_{\mathrm{alt}}$ 
of $\functors$ and $\functort$, starting in the opposite color of 
$\bbk{i}\in\g$, to $P_{\bbk{i}}$ of length determined 
by the minimal path connecting $\bbk{i}$ and $\bbk{j}$ in the graph $\g$. 
Then $P_{\bbk{j}}\{\shiftdeg\}$ will be a 
summand of $\Functor{\Theta}_{\mathrm{alt}}(P_{\bbk{i}})$ for 
some shift $\shiftdeg\in\Z$. This, by \fullref{remark:weak2}, 
shows transitivity.

It remains to show that there are no non-trivial ideals.
This is clearly true in case 
$\g$ has only one vertex. Otherwise, fix one {\color{sea}sea-green s} colored 
vertex $\bbii{i}$. 
Let us restrict the $2$-action to $\functors$ on the additive category 
generated by $P_{\bbii{i}}$. This category is equivalent to the category 
of graded $\dualalg$-modules for $\dualalg=\C[X]/(X^2)$. Hence, the $2$-action of $\functors$ is 
given by tensoring with the 
biprojective $\dualalg$-bimodule $\dualalg\otimes \dualalg$, see 
e.g. \cite[Section 3.3]{MM}. Thus, 
the non-identity endomorphism $\loopy{\bbii{i}}$ on $P_{\bbii{i}}$ cannot belong 
to any ideal which is stable under the $2$-action and 
does not contain any identity morphism 
(this follows from \cite[Proposition 2]{MM}). This, of course, 
holds for any vertex $\bbk{i}$, i.e. none of the maps 
$\loopy{\bbk{i}}$ are in such an ideal. 
Now, if  
$\bbii{i}|\bbjj{j}\colon P_{\bbjj{j}}\to P_{\bbii{i}}$ (or 
$\bbjj{j}|\bbii{i}\colon P_{\bbii{i}}\to P_{\bbjj{j}}$) would be 
contained in such an 
ideal, then, because it is a $2$-ideal, composition with
$\bbjj{j}|\bbii{i}\colon P_{\bbii{i}}\to P_{\bbjj{j}}$ (or 
$\bbii{i}|\bbjj{j}\colon P_{\bbjj{j}}\to P_{\bbii{i}}$) would imply that the 
result 
$\loopy{\bbii{i}}\colon P_{\bbii{i}}\to P_{\bbii{i}}$ 
(or $\loopy{\bbjj{j}}\colon P_{\bbjj{j}}\to P_{\bbjj{j}}$) belongs to 
the ideal, which is a contradiction.
\medskip

\noindent\textit{Proof of statement (b)}\,.
(Below we use ${}^{\prime}$ 
as a notation for anything related to $\g^{\prime}$. 
Similarly in the rest of the whole proof.)
First assume that $\g$ and $\g^{\prime}$ are 
isomorphic as bipartite graphs by some map $f\colon\g\to\g^{\prime}$. Such an isomorphism 
induces a reordering of the indecomposables 
of $\cellcatG$ and $\cellcatGnext$ via
$
P_{\bbii{i}}\mapsto P^{\prime}_{\bbii{f(i)}}$ and 
$
P_{\bbjj{i}}\mapsto P^{\prime}_{\bbjj{f(i)}}.
$ 
Thus, we can define a functor
\[
\functor{F}\colon\cellcatG\to\cellcatGnext,
\quad 
\functor{F}(P_{\bbk{i}})= P^{\prime}_{\bbk{f(i)}},
\]
which maps each morphism in $\cellcatG$ to the evident 
($f$-reordered) morphism in $\cellcatGnext$. 
This is well-defined since $f$ is an isomorphism of bipartite graphs 
(in particular, for each $\bbk{i}\in\g$ the image $f(\bbk{i})\in\g^{\prime}$ has precisely 
the same two-step connectivity neighborhood).
This functor is 
clearly structure preserving, and it is an equivalence  
since for all $\bbk{i}\in\g$ the valencies are preserved 
(which shows that the functor induces isomorphisms between the finite-dimensional hom-spaces).
Thus, $\functor{F}$ is a structure preserving 
equivalence of categories.

By construction, $\functors$ is given by tensoring with 
sums of $\algG$-bimodules of the form 
$P_{\bbk{i}}\{\shiftme{-1}\}\otimes {}_{\bbk{i}}P$,
while $\functors^{\prime}$ 
is given by tensoring with 
sums of $\algG$-bimodules of the form 
$P^{\prime}_{\bbk{i}}\{\shiftme{-1}\}\otimes {}_{\bbk{i}}P^{\prime}$. Thus, there 
is a commuting diagram
\[
\xymatrix{
\cellcatG\ar[r]^{\functors}\ar[d]_{\functor{F}} & \cellcatG\ar[d]^{\functor{F}}\\
\cellcatGnext\ar[r]_{\functors^{\prime}} & \cellcatGnext,
}
\]
which similarly exists for $\functort$ and $\functort^{\prime}$ as well.
Hence, we have $\functors^{\prime}=\functor{F}\functors\functor{F}^{-1}$ 
and $\functort^{\prime}=\functor{F}\functort\functor{F}^{-1}$.
This means that $\functorGno$ and $\functorGno^{\prime}$ agree 
on $1$- and on $2$-morphisms of $\dihedralcatbig{}$ up to $\functor{F}$-conjugation.
Thus, this induces a (degree-zero) modification between 
$\widetilde{\functorGno}$ and $\widetilde{\functorGno} {}^{\prime}$, which is a $2$-isomorphism.

Vice versa, any equivalence between $\widetilde{\functorGno}$ and 
$\widetilde{\functorGno} {}^{\prime}$ sends indecomposables to indecomposables, so it descends 
to an intertwining isomorphism of the form 
\[
\GG{\cellcatG}
\xrightarrow{\cong}
\GG{\cellcatGnext},
\quad
[P_{\bbk{i}}]\mapsto[P^{\prime}_{\bbk{i^{\prime}}}]
\] 
between the Grothendieck groups. Such an isomorphism can neither send a
$[P_{\bbii{i}}]$ to a $[P^{\prime}_{\bbjj{i^{\prime}}}]$ 
nor a $[P_{\bbjj{i}}]$ to a $[P^{\prime}_{\bbii{i^{\prime}}}]$, since 
this will not intertwine 
the action of $[\functors]$ and $[\functort]$ to 
$[\functors^{\prime}]$ and $[\functort^{\prime}]$. 
In particular, the number of {\color{sea}sea-green} $\bbii{i}$ and 
{\color{tomato}tomato} $\bbjj{j}$ colored vertices 
has to be the same for $\g$ and $\g^{\prime}$.
Further, for the same reasons, if such an isomorphism sends  
$[P_{\bbk{i}}]$ and $[P_{\bbk{j}}]$ to 
$[P^{\prime}_{\bbk{i^{\prime}}}]$ and $[P^{\prime}_{\bbk{j^{\prime}}}]$, and $\bbk{i},\bbk{j}$ 
are connected via an edge in $\g$, then 
$\bbk{i}^{\prime},\bbk{j}^{\prime}$ 
are connected via an edge in $\g^{\prime}$. 
In total, such a permutation gives rise to an isomorphism 
of bipartite graphs $f\colon\g\to\g^{\prime}$ 
defined via $f(\bbk{i})=\bbk{i}^{\prime}$.
\medskip

\noindent\textit{Proof of statement (c)}\,.
First --
by using traces -- one can see that two 
bipartite graphs $\g$ and $\g^{\prime}$ which are 
spectrum-color-inequivalent give non-isomorphic 
$\mathrm{H}$-modules. Precisely, for all $k\in\Z_{\geq 0}$, 
we have (top: even $k$; bottom: odd $k$):
\begin{gather}\label{eq:st-matrices}
\begin{gathered}
[\Functor{\Theta}_{\overline{\dihsgen_{k}}}]
=
(AA^{\mathrm{T}})^{\tfrac{k-2}{2}}\cdot\!
\begin{pmatrix} 
AA^{\mathrm{T}} & A[2]_{\vpar}\\ 
0 & 0\\
\end{pmatrix}\!
,\;
[\Functor{\Theta}_{\overline{\dihtgen_{k}}}]
=
(A^{\mathrm{T}}\! A)^{\tfrac{k-2}{2}}\cdot\!
\begin{pmatrix} 
0 & 0\\ 
A^{\mathrm{T}}[2]_{\vpar} & A^{\mathrm{T}}\! A\\
\end{pmatrix}\!,
\\
[\Functor{\Theta}_{\overline{\dihsgen_{k}}}]
=
(AA^{\mathrm{T}})^{\tfrac{k-1}{2}}\cdot\!
\begin{pmatrix} 
[2]_{\vpar} & A\\ 
0 & 0\\
\end{pmatrix}
,\;
[\Functor{\Theta}_{\overline{\dihtgen_{k}}}]
=
(A^{\mathrm{T}}\! A)^{\tfrac{k-1}{2}}\cdot\!
\begin{pmatrix} 
0 & 0\\ 
A^{\mathrm{T}} & [2]_{\vpar}\\
\end{pmatrix}.
\end{gathered}
\end{gather}
(Where $\Functor{\Theta}_{\overline{\dihsgen_{k}}}$ 
and $\Functor{\Theta}_{\overline{\dihtgen_{k}}}$ mean the evident 
analog of the notation from \fullref{subsec:dihedral-stuff}.)
Here $A$ is as in \eqref{eq:isadmatrix} for $A(\g)$. 
This works completely analogously for $[\Functor{\Theta}_{\overline{\dihsgen_{k}}}^{\prime}]$ 
and $[\Functor{\Theta}_{\overline{\dihtgen_{k}}}^{\prime}]$
as well, where one uses $A^{\prime}$ 
-- coming from $A(\g^{\prime})$ -- instead of $A$.

If $\vert \graphs\vert \neq\vert \graphs^{\prime}\vert$ 
or $\vert \grapht \vert \neq\vert \grapht^{\prime}\vert$, then
the bottom row in \eqref{eq:st-matrices} produces different traces, so the 
representations are non-isomorphic. 

If $S_{\g}\neq S_{\g^{\prime}}$, then one also gets non-isomorphic representations. To see this, note that 
the spectrum of the matrices in the top row associated 
to $[\Functor{\Theta}_{\overline{\dihsgen_{k}}}]$ 
and $[\Functor{\Theta}_{\overline{\dihtgen_{k}}}]$, is that of $(AA^{\mathrm{T}})^{\scalebox{.7}{$\nicefrac{$k{+}2$}{$2$}$}}$ and $(A^{\mathrm{T}}A)^{\scalebox{.7}{$\nicefrac{$k{+}2$}{$2$}$}}$, for all $k\in\Z_{\geq 0}$. A similar observation holds for the spectrum of $[\Functor{\Theta}_{\overline{\dihsgen_{k}}}^{\prime}]$ and $[\Functor{\Theta}_{\overline{\dihtgen_{k}}}^{\prime}]$. 
So these spectra consist of powers of the elements of $S_{\g}$ and $S_{\g^{\prime}}$, respectively.  
Since the spectrum is an invariant of the representation, the result follows. 

Next, let us assume 
that $A(\g)$ and $A(\g^{\prime})$ have the same 
spectra. Thus, by \eqref{eq:isadmatrix}, there exists a 
singular value decomposition of the form
\[
A=U\Sigma V^*,\quad\quad A^{\prime}=W\Sigma X^*
\]
with unitary, complex-valued matrices $U,V,W,X$ of the appropriate sizes. 
(Note that we work over $\C(\vpar)$, but $A(\g)$ and $A(\g^{\prime})$ 
are integral matrices. Hence, their singular value decompositions exist.) Thus,
\begin{align*}
\begin{pmatrix}
WU^* & 0\\
0 & XV^*
\end{pmatrix}
\![\functors]\!
\begin{pmatrix}
UW^* & 0\\
0 & VX^*
\end{pmatrix}
&=
\begin{pmatrix}
WU^* & 0\\
0 & XV^*
\end{pmatrix}\!
\begin{pmatrix}
[2]_{\vpar} & A\\
0 & 0
\end{pmatrix}\!
\begin{pmatrix}
UW^* & 0\\
0 & VX^*
\end{pmatrix}\\
&=
\begin{pmatrix}
[2]_{\vpar} & A^{\prime}\\
0 & 0
\end{pmatrix}
=
[\functors^{\prime}],
\end{align*}
and similarly 
-- using the same matrices for conjugation -- for $[\functort]$ and 
$[\functort^{\prime}]$. (Hereby we have also used 
that $\vert \graphs\vert =\vert \graphs^{\prime}\vert$ 
and $\vert \grapht \vert =\vert \grapht^{\prime}\vert$.)

This shows that there 
is a change of basis such 
that $[\functors]$ is sent to $[\functors^{\prime}]$ 
and $[\functort]$ is sent to $[\functort^{\prime}]$, showing that the underlying 
$\mathrm{H}$-modules are in fact isomorphic.
\medskip

\noindent\textit{Proof of statement (d)}\,.
By using the relations of $\dihedralcatbig{}$ one easily 
sees that
\[
\twoEnd_{\dihedralcatbig{}}(\varnothing)\cong\C[\barbells,\barbellt]=\polyalg,
\quad\quad\mathrm{deg}(\barbells)=\mathrm{deg}(\barbellt)=2,
\] 
where 
$\barbells,\barbellt$ are the ``floating barbells'', see \cite[Corollary 5.20 and Lemma 6.23]{El1}.

It follows from the barbell forcing relations \eqref{eq:barb1} 
and \eqref{eq:barb2} that the $2$-hom spaces in $\dihedralcatbig{}$ 
are graded $\polyalg$-bimodules with a generating set 
given by Soergel diagrams without ``floating components''. Hence, $\dihedralcatbig{}$ 
is defined over the polynomial 
algebra $\polyalg$, see \cite[Proposition 5.19 and Proposition 6.22]{El1}.

Now fix $n\in\Z_{>1}$ (the case $n=1$ is discussed 
in \fullref{example:main-2-cats-2-reps}).
Let $\polyalg^{\mathrm{W}_n}\subset \polyalg$ denote the subalgebra of invariant elements. The coinvariant 
algebra is defined as 
\[
\coalg^+_{\mathrm{W}_n}=\polyalg/I^+,
\]
where $I^+$ is the ideal generated by the elements in $\polyalg^{\mathrm{W}_n}$ which are 
homogeneous of positive degree.
Define
\begin{gather*}
z=\barbells\barbells-[2]_{\qpar}\cdot\barbells\barbellt+\barbellt\barbellt\in\C[\barbells,\barbellt],\\
Z={\textstyle\prod_{\word\in\mathrm{W}_n}}\word(2\cdot\barbells+[2]_{\qpar}\cdot\barbellt)\in\C[\barbells,\barbellt],\quad\text{where}
\\
\dihsgen(\barbells)=\barbells,\quad\dihtgen(\barbells)=\barbells-[2]_{\qpar}\cdot\barbellt,\quad
\dihsgen(\barbellt)=\barbellt-[2]_{\qpar}\cdot\barbells,\quad\dihtgen(\barbellt)=\barbellt.
\end{gather*}
(Note that the element $Z$ does not make sense for ``$n=\infty$''.)
By \cite[Claim 3.23]{El1}, the subalgebra $\polyalg^{\mathrm{W}_n}$ is isomorphic 
to $\C[z,Z]\subset\polyalg$. 

Thus, in order to check that our weak $2$-functors $\functorADE$ descend from $\dihedralcatbig{n}$ 
to $(\dihedralcat{n}^{\mathrm{f}})^{\ast}$, we have to check 
that $\functorADE(z)=0=\functorADE(Z)$.

First note that any product of two or more barbells is zero, because 
\[
\bbii{i}|\bbii{i}\pathm\bbii{i}|\bbii{i}=
\bbjj{j}|\bbjj{j}\pathm\bbii{i}|\bbii{i}=
\bbii{i}|\bbii{i}\pathm\bbjj{j}|\bbjj{j}=
\bbjj{j}|\bbjj{j}\pathm\bbjj{j}|\bbjj{j}=0.
\]
By \eqref{eq:floating-barbell1} and \eqref{eq:floating-barbell2}, this implies that $\functorADE(z)=0$. 

The same argument shows that $\functorADE(Z)=0$, e.g. we already have 
\[
(2\cdot\barbells+[2]_{\qpar}\cdot\barbellt)\dihsgen(2\cdot\barbells+[2]_{\qpar}\cdot\barbellt)
=(2\cdot\barbells+[2]_{\qpar}\cdot\barbellt)
((2-[2]_{\qpar}^2)\cdot\barbells+[2]_{\qpar}\cdot\barbellt)=0.
\]

When ``$n=\infty$'', one has $\polyalg^{\mathrm{W}_{\infty}}\cong\C[z]$. This case therefore follows in the same way as in the finite case.
\end{proof}

\begin{example}\label{example:SVD}
Note that the proof of (c) of \fullref{theorem:main-proposition} 
is effective in the sense that one can explicitly 
compute the change of basis matrix via the singular value decomposition. 
For the type $\typee{6}$ situation from \fullref{example-bigraph} 
one gets for instance
\begin{gather*}
UW^*
=
\frac{1}{2\sqrt{6}}\cdot
\begin{pmatrix}
\sqrt{6}-2 & -\sqrt{6}-2 & 2\\
-2 & -2 & 4\\
-\sqrt{6}-2 & \sqrt{6}-2 & 2
\end{pmatrix}
,\quad\quad
VX^*
=(UW^*)^{\mathrm{T}},
\end{gather*}
(using the ordering of the vertices from left to 
right, {\color{sea}sea-green} $\bullets$ 
before {\color{tomato}tomato} $\bullett$)
which gives the (highly non-integral) change of basis matrix.
\end{example}

\begin{proof}[Proof of \fullref{theorem:classification}]
By \fullref{theorem:main-proposition} it remains to rule out the case that there are 
graded simple transitive $2$-representations which are not as in \fullref{theorem:main-theorem}.

In order to rule these out: 
By (d) of \fullref{theorem:main-proposition} 
and \fullref{remark:strictification}, 
any graded simple transitive $2$-representation of $\Kar(\dihedralcat{n})^{\star}$ 
would give the corresponding simple transitive $2$-representation
of $\Kar(\dihedralcat{n}^{\mathrm{f}})$. 
The underlying quiver of such a 
$2$-representation is of $\ADE$ type and their 
action on the level of the Grothendieck groups is fixed up to change of basis, 
see \cite[Sections 6 and 7]{KMMZ}. This structure on the level of $1$-morphisms is preserved under 
strictification. By combining \fullref{lemma:unique-maps} and \fullref{lemma:unique-maps2} 
(``uniqueness of higher structure''), 
these are equivalent to the ones from \fullref{theorem:main-theorem}.
Indeed, we can define a degree-preserving autoequivalence of $\cellcatG$: 
on ob\-jects it is the identity; on morphisms it is given by multiplying 
any homogeneous $\algG$-bimodule map by (suitable) powers of $\scalargt$ and 
$\scalargtt$. By definition, this 
autoequivalence intertwines the $2$-representation with the $\scalargt,\scalargtt$-scaling 
and the one from \fullref{theorem:main-theorem}.
\end{proof}

Note that the proof of \fullref{theorem:classification} 
relies on \fullref{lemma:unique-maps}, whose proof 
(as given in this paper) uses the grading in an essential way. 

\begin{remark}\label{remark:different-qs}
In fact, in order to complete the classification in \cite{KMMZ} 
it remains to compare different choices for $\qpar$, i.e. we fix $\g$ and consider two 
different complex, primitive $2n$th
root of unity and the associated $2$-representations constructed 
in \fullref{subsec:diacataction}. 
We have to show that they are equivalent when working over $\C$ and over Soergel bimodules 
(since the Soergel bimodules 
``do not see the $\qpar$'s'').

Luckily, this follows immediately from \cite[Theorem 6.24]{El1}, which shows that the 
$2$-color Soergel calculi for different choices of $\qpar$ 
are equivalent. This equivalence gives rise to an equivalence of the two $2$-representations 
in question.
\end{remark}


\bibliographystyle{alphaurl}
\bibliography{cat-dihedral}
\end{document}